\documentclass{article}[10pt]

\usepackage{amsmath, amssymb, amsthm}

\usepackage[usenames, dvipsnames]{color}
\usepackage{graphicx}
\usepackage{enumerate}
\usepackage[toc,title,page]{appendix}

\usepackage{hyperref}
\hypersetup{
    colorlinks=true, 
    linktoc=all,     
    linkcolor=blue,  
}

\newtheorem*{rep@theorem}{\rep@title}
\newcommand{\newreptheorem}[2]{%
\newenvironment{rep#1}[1]{%
 \def\rep@title{#2 \ref{##1}}%
 \begin{rep@theorem}}%
 {\end{rep@theorem}}}
\makeatother

\theoremstyle{plain}
\newtheorem{theorem}{Theorem}[section]
\newreptheorem{theorem}{Theorem}
\newtheorem{prop}[theorem]{Proposition}
\newtheorem{lem}[theorem]{Lemma}
\newtheorem{cor}[theorem]{Corollary}

\theoremstyle{definition}
\newtheorem{defn}[theorem]{Definition}
\newtheorem{rem}[theorem]{Remark}
\newtheorem{notation}[theorem]{Notation}
\newtheorem{example}[theorem]{Example}

\newcommand{\C}{\ensuremath{{\mathbb{C}}}}
\newcommand{\Chat}{\ensuremath{{\hat{\mathbb{C}}}}}
\newcommand{\D}{\ensuremath{\mathbb{D}}}
\newcommand{\N}{\ensuremath{\mathbb{N}}}
\newcommand{\T}{\ensuremath{\mathbb{T}}}
\newcommand{\R}{\ensuremath{\mathbb{R}}}
\newcommand{\Z}{\ensuremath{\mathbb{Z}}}

\newcommand{\F}{\mathcal F}

\newcommand{\Dtwo}{\ensuremath{{\mathbb{D}^2}}}
\newcommand{\Ttwo}{\ensuremath{{\mathbb{T}^2}}}
\newcommand{\Ctwo}{\ensuremath{{\mathbb{C}^2}}}
\newcommand{\Rtwo}{\ensuremath{{\mathbb{R}^2}}}
\newcommand{\Ztwo}{\ensuremath{{\mathbb{Z}^2}}}

\newcommand{\Ctwohat}{\ensuremath{{\hat{\mathbb{C}}^2}}}
\newcommand{\cc}{{\subset\!\!\!\subset}}
\newcommand{\Dsd}{\ensuremath{D^\sigma_\delta}}
\newcommand{\Tsd}{\ensuremath{\mathbb{T}^\sigma_\delta}}
\newcommand{\TsdP}{\ensuremath{\mathbb{T}^\sigma_{P,\delta}}}

\newcommand{\Ds}[1]{\ensuremath{D^\sigma_{#1}}}

\newcommand{\cl}[1]{\operatorname{cl}({#1})}
\newcommand{\inter}[1]{\operatorname{int}{(#1)}}
\newcommand{\cj}[1]{\overline{#1}}

\newcommand{\GLtwo}{\ensuremath{{\operatorname{GL}_2}}}

\newcommand{\Aut}{\ensuremath{{\operatorname{Aut}}}}

\makeatletter
\newcommand{\subjclass}[2][1991]{%
  \let\@oldtitle\@title%
  \gdef\@title{\@oldtitle\footnotetext{#1 \emph{Mathematics subject classification.} #2}}%
}
\newcommand{\keywords}[1]{%
  \let\@@oldtitle\@title%
  \gdef\@title{\@@oldtitle\footnotetext{\emph{Key words.} #1.}}%
}

\date{}
\begin{document}

\title{Resonances for rational Anosov maps on the torus}
\setcounter{tocdepth}{2}
\author{Julia Slipantschuk\footnote{Department of Mathematics, University of Warwick, Coventry, UK. Email: julia.slipantschuk@warwick.ac.uk},
Oscar F. Bandtlow\footnote{School of Mathematical Sciences, Queen Mary University of London, UK. Email: o.bandtlow@qmul.ac.uk} and
Wolfram Just\footnote{Institut f\"ur Mathematik, Universit\"at Rostock, Rostock, Germany. Email: wolfram.just@uni-rostock.de}}


\maketitle

\begin{abstract}
A complete description of resonances for rational toral Anosov diffeomorphisms preserving certain
Reinhardt domains is presented. As a consequence it is shown that every homotopy class of
two-dimensional Anosov diffeomorphisms contains maps with the sequence of resonances decaying stretched-exponentially.
This is achieved by introducing a certain group of rational toral diffeomorphisms and computing the
resonances of the respective composition operator considered on suitable anisotropic spaces of hyperfunctions.
The class of examples is sufficiently rich to also include non-linear Anosov maps with trivial resonances, or
resonances decaying exponentially, as well as with or without area-preservation or reversing symmetries.
\end{abstract}

\tableofcontents

\section{Introduction}

Anosov diffeomorphisms are the simplest truly hyperbolic dynamical systems, whose
long term asymptotic behaviour characterized by correlation decay or rates of mixing are classical topics
in smooth ergodic theory and statistical physics.
The main tool for studying statistical properties for Anosov maps $T$ acting on a compact manifold $M$ is the {\it weighted composition operator}\footnote{This operator is also referred to as the transfer operator or Koopman operator, depending on the weight.} defined by
\[C_{T,w}\colon f \mapsto w \cdot f\circ T, \]
where $w$ is a smooth function on $M$ and $C_{T,w}$ acts on a suitable space of distributions.
There is a
considerable amount of recent literature devoted to the construction of so-called anisotropic spaces for
hyperbolic dynamical systems, on which the above operator is quasicompact,
implying a spectral gap and exponential decay of correlations. The general idea behind the construction of these spaces is to create sufficient
smoothness in the expanding direction and dual smoothness in the contracting direction, see \cite{Ru,BKL,GL,BT,FRS,B}, to name but a few.

If the Anosov map is real-analytic, it is sometimes possible to prove compactess and even nuclearity
of the above operator, see \cite{Ru, Je, FR}, implying that its non-zero spectrum is a sequence of isolated eigenvalues
known as {\it Pollicott-Ruelle resonances}, which determine all intrinsic
exponential mixing rates of the given system. However, quantitative results
such as location, or the very existance, of non-trivial resonances are rare, a few instances in different hyperbolic settings being \cite{FGL,DFG}.
Even for Anosov diffeomorphisms on
the torus $\T^2$, arguably the simplest setting of uniformly hyperbolic dynamical systems, it was
established only recently in \cite{A} (after an idea of F.~Naud) that non-trivial resonances
exist generically. In \cite{SlBaJu_NONL17} the authors presented a one-parameter family of Anosov maps,
proving the presence of infinitely many distinct resonances by calculating their location explicitly, and
also conjectured locations of resonances for another family of Anosov maps, which was recently proven in \cite{PoS}.
In both cases, after establishing compactness of the transfer operator on a suitable
anisotropic Hilbert space (originally introduced in \cite{FR}), the eigenvalues are read off from an upper triangular matrix
representation of the operator with respect to a weighted Fourier basis. These results,
though valuable for rigorously establishing the location of resonances for these particular families, are rather ad hoc as
they do not reveal the underlying spectral structure of the associated operator.
In this context, a key contribution of this article is to explain
the underlying structure of resonances
for a class of rational Anosov diffeomorphisms. For this it will be
helpful to work with simple anisotropic spaces, which, albeit less general,
interact well with the analytic structure of the underlying map.

For analytic Anosov maps on $\T^2$ with constant invariant stable and unstable cone fields we
define anisotropic Hilbert spaces of hyperfunctions as
a closure of Laurent polynomials under a weighted $L^2$ norm with the weight function adapted to the invariant cones,
and show that the respective weighted composition operator is well defined and trace-class. These spaces are isometrically
isomorphic to a direct sum of Hardy-Hilbert spaces on log-conical Reinhardt domains, with the logarithmic base
induced by the stable and unstable cones. Using this viewpoint and assuming additionally that these Anosov maps extend
holomorphically to certain domains in $\Ctwohat$, we are able to explicitly compute all the eigenvalues of the respective
composition operator. In addition we present a group of (non-linear) toral diffeomorphisms which satisfy these assumptions
and provide examples in each homotopy class of toral Anosov diffeomorhisms for which the composition operator has
infinitely many distinct non-trivial resonances $(\lambda_n)_{n \in \N}$ whose rate of decay is stretched-exponential, that
is, the upper bound of $\exp(-a n^{1/2})$ for some $a>0$ is tight.

Results in \cite{BJS} on resonances for analytic expanding circle maps of degree $d$, with $|d|>2$, arising from finite
Blaschke and anti-Blaschke products, made it possible to establish exponential lower bounds for the decay rate
on a dense set of analytic expanding circle maps \cite{BN}. In the same vein, the class of non-linear toral diffeomorphisms presented
in this paper will be an essential ingredient in proving that, generically, the stretched-exponential decay rate of resonances is
optimal within the class of analytic Anosov diffeomorphisms on $\T^2$. Moreover, as all constructions are very explicit
most of our results should extend to Anosov maps in higher dimensions. These are, however, beyond the current scope and
will be pursued in subsequent works.

\subsection{Statement of results}
We will only consider toral Anosov diffeomorphisms with constant expanding
and contracting cone fields. Restricting to constant invariant cone fields enables us to
work with simple anisotropic Hilbert spaces $H_\nu$, which can
be seen as the completion of Laurent polynomials
under the $\nu$-weighted $L^2_\nu$ inner product
for a \textit{cone-wise exponential weight}
function $\nu\colon\Ztwo \to \R_{\geq 0}$, that is $\nu(n) = e^{f(n)}$ where
$f\colon \Ztwo \to \mathbb{R}$ is piecewise linear, with the pieces being cones in $\R^2$.
We will make an additional assumption on the invariant cone fields, termed
{\it strongly expanding constant invariant cone field} condition or short-hand {\it (sec)},
see Definition \ref{def:strongly_exp} for the precise definition.
If the (constant) unstable invariant cone field can be chosen
to correspond to the positive/negative quadrants $\pm \mathbb{R}_{>0}^2$, we refer to these as positive, and call
the respective condition {\it (p-sec)}.

\begin{theorem}\label{thm:thmA}
Let $T$ be an analytic Anosov diffeomorphism of $\Ttwo$
satisfying the \textit{(sec)} assumption and $w\colon\Ttwo \to \C$ an analytic function.
Then there exists a cone-wise exponential $\nu\colon \Ztwo \to \R_{\geq 0}$ such that
\[ f \mapsto w \cdot f \circ T\] is a well-defined
trace-class operator on $H_\nu$.
\end{theorem}

For a special subclass of analytic Anosov diffeomorphisms where cone fields
can be chosen to correspond to quadrants of $\mathbb{R}^2$ and the diffeomorphisms extend holomorphically
to certain domains of $\Ctwohat$, we are able to compute
all eigenvalues of the above operator explicitly in terms of multipliers
of fixed points on these domains.
To state the results we introduce the notation
$\Sigma = \{\sigma = (\sigma_1, \sigma_2)\colon \sigma_1, \sigma_2 \in \{\pm1\}\}$, and
$D^\sigma =\{z \in \Ctwohat \colon |z_1|^{\sigma_1} > 1,
|z_2|^{\sigma_2} > 1\}$
with $\sigma\in \Sigma$ for the four bidisks in $\Ctwohat$. It will further be convenient to
write $\mathcal{N}^{1} = \mathbb{N}^2_0\setminus\{(0,0)\}$,
$\mathcal{N}^{-1} = \mathbb{N}^2$, and decompose $\Sigma$ as $\Sigma=\Sigma^{1} \cup \Sigma^{-1}$
with $\Sigma^{\ell} = \{\sigma \in \Sigma \colon \sigma_1 \cdot \sigma_2 =\ell\}, \ell \in \{\pm 1\}$.
The notation $T^\ell$ for $\ell=-1$ stands for the inverse of $T$. We write $C_T = C_{T,1}$ for the unweighted composition operator.

\begin{theorem}\label{thm:thmB}
Let $T$ be an analytic Anosov diffeomorphism on $\Ttwo$ satisfying the \textit{(p-sec)} assumption.
Then there exists a cone-wise exponential $\nu\colon \Ztwo \to \R_{\geq 0}$ such that
$C_T$ is a well-defined trace-class operator on $H_\nu$, and its spectral
determinant is an entire function of the form
\[\det (\operatorname{Id} - z C_T) =
(1-z) \chi_T(z),\]
where $\chi_T$ is an entire function with zeros outside of $\cl{\D}$.
Moreover, if $T^\ell$ holomorphically extends\footnote{Here we view
$\mathbb{T}^2$ as a subset of $\hat\C \times \hat\C$ and consider
extensions of $T$ and $T^{-1}$ to subsets of $\hat\C \times \hat\C$ as holomorphic
in the same sense as for mappings of the Riemann sphere.} to
$D^\sigma$ for all $\sigma \in \Sigma^\ell$ and $\ell \in \{\pm 1\}$ then $\chi_T$ is the product of two entire functions
$\chi^{+1}_T$ and  $\chi^{-1}_T$, whose zeros are given explicitly. Specifically,
for every $\ell\in\{\pm 1\}$ exactly one of the following cases holds:
\begin{enumerate}[(i)]
  \item if $T^{\ell}(D^\sigma) \subseteq D^\sigma$ for all $\sigma \in \Sigma^\ell$, then
  \[\chi_T^\ell(z) =  \prod_{n \in \mathcal{N}^{\ell}} \prod_{\sigma \in \Sigma^{\ell}} (1 - z\ell^s \lambda^n_\sigma) \]
  where $\lambda_\sigma = (\lambda_{\sigma,1}, \lambda_{\sigma,2})$ are the multipliers\footnote{The multipliers of $T$ at a point
$z$ are the eigenvalues $\lambda = (\lambda_1, \lambda_2)$ of
$D_zT = \left(\frac{\partial T_k}{\partial z_l}\right)_{k,l}$ at $z$.
We write $\lambda^n = \lambda_{1}^{n_1} \cdot \lambda_{2}^{n_2}$ for $n \in \Z^2$.} of the unique attracting
fixed point $z^*_\sigma \in D^\sigma$ of $T^\ell$, and $s = 0$ if $T$ is orientation-preserving and $s=1$ if it is orientation-reversing.
Additionally, it holds that $z_{\sigma}^*  = \cj{1/z^*_{-\sigma}}$ and $\lambda_\sigma = \cj{\lambda_{-\sigma}}$
for $\sigma \in \Sigma^\ell$.
\item if $T^{\ell}(D^\sigma) \subseteq D^{-\sigma}$ for all $\sigma \in \Sigma^{\ell}$, then
\[\chi^{\ell}_T(z) = \prod_{n \in \mathcal{N}^{\ell}} \prod_{\sigma \in \Sigma^{\ell}}  (1 - z^2\lambda^n_\sigma)^{1/2}, \]
where $\lambda_\sigma = \lambda_{-\sigma}$ are the multipliers of the unique attracting
fixed point $z^*_\sigma \in D^\sigma$ of $T^{2\ell}$.
Additionally, $\lambda_{\sigma,1}$ and $\lambda_{\sigma,2}$ are either real, or complex conjugates of each other.
\end{enumerate}
\end{theorem}

Clearly, certain linear toral diffeomorphisms satisfy the assumptions of this theorem.
For example,
the well-known cat map $(z_1, z_2) \mapsto (z_1^2z_2, z_1z_2)$ satisfies assumption {\it (i)} and
hence we can compute all fixed points and their multipliers, which however are all trivial (that is, zero) in this case.
The group $\Aut(\T^2)$ of linear diffeomorphisms (automorphisms) of the torus is isomorphic to $\GLtwo(\Z)$.
In order to construct
non-linear maps, we shall consider a group of diffeomorphisms generated by a finite set $\Gamma$ of generators of
$\Aut(\T^2)$ and a (uncountably infinite) set $\mathcal{G}$ of certain rational diffeomorphisms preserving $\Ttwo$.
One particular choice for this set is \[\mathcal{G} = \{(z_1, z_2)\mapsto(b_a(z_1), z_2) \colon a \in \D\}\]
with $b_a(z) = (z-a)/(1-\cj{a} z)$. We call $\mathcal{F}$ the group of diffeomorphisms
generated by $\Gamma \cup \mathcal{G}$.
A certain subset of $\mathcal{F}$ comprises
of hyperbolic diffeomorphisms satisfying the assumptions of Theorem~\ref{thm:thmB}.
The explicit construction of elements of this set (see Section \ref{sec:blaschke_examples})
allows us to compute the sequence of eigenvalues of the
corresponding composition operators and present examples with qualitatively different decay rates of this sequence (stretched-exponential, exponential, and trivially super-exponential/all-zero).
Using the structure of conjugacy classes of $\GLtwo(\Z)$ we obtain the following theorem.

\begin{theorem} \label{thm:thmC}
Every homotopy class of analytic Anosov diffeomorphisms on $\Ttwo$ contains (non-linear)
Anosov diffeormorphisms $T\in \mathcal{F}$, such that for suitable $H_\nu$ the corresponding operator $C_T$ is well defined and trace-class,
with the entire function $z \mapsto \det(\operatorname{Id} - z C_T) = (1-z) \chi_T(z)$ as its spectral determinant.
In particular, denoting by $(\lambda_n)_{n\in\N}$ the sequence of eigenvalues of $C_T$
ordered by modulus in decreasing order, and counted with multiplicities, we obtain the following.
\begin{enumerate}[(i)]
  \item For every homotopy class $\mathcal{H}$ and $\eta > 0$, there exists $T\in \mathcal{H} \cap \mathcal{F}$ such that
  the eigenvalue sequence of $C_T$ satisfies
  \[\lim_{n \to \infty} \frac{- \log |\lambda_n|}{n^{1/2}} = \eta.\]

  \item For every homotopy class $\mathcal{H}$ of
  orientation-preserving Anosov diffeomophisms and $\eta > 0$, there exists $T\in \mathcal{H} \cap \mathcal{F}$ such that
  the eigenvalue sequence of $C_T$ satisfies
  \[\lim_{n \to \infty} \frac{- \log |\lambda_n|}{n} = \eta.\]

  \item Every homotopy class $\mathcal{H}$ of Anosov diffeomorphisms not containing a linear conjugate of one of $\{(z_1, z_2) \mapsto (z_1^kz_2, z_1), k\in \N\}$ has an element $T\in \mathcal{H} \cap \mathcal{F}$ not smoothly conjugated to a linear Anosov diffeomorphism, such that \[\chi_T(z) = 1\] for all $z \in \C$.
\end{enumerate}
\end{theorem}

\begin{rem}
We note that the maps $T\in \mathcal{F}$ above can be written in closed form.
In contrast to the more common setting of analytic perturbations of linear maps (see, for example,  \cite{A,FR}),
these maps are not required to be $C^1$ close to the respective linear automorphisms, that is, the
$C^1$ distance can be arbitrarily large. In fact, as the proof of Theorem \ref{thm:thmC} will show,
the spectral gap for the composition operator associated to the maps in $(i)$ and $(ii)$ can take an arbitrary value in $(0,1)$.
Moreover, under the additional assumption that $T$ is orientation-preserving, in case $(i)$ the second-largest eigenvalue $\lambda_2 \in \D$
and the decay rate $\eta > 0$ can be chosen arbitrarily, independently of each other.
The behaviour in $(i)$ is believed to be generic for two-dimensional
Anosov maps, whereas the cases $(ii)$ and $(iii)$ are exceptional.
We also note that under the assumptions of Theorem~\ref{thm:thmB} (as well as of Theorem~\ref{thm:thmC}), the spectral determinant can be shown
to coincide with the usual dynamical determinant, as the usual trace formula (see, for example, \cite[Proposition 5]{FR}) can
be established in this setting.
\end{rem}

\begin{rem}
As shown in \cite{FG}, any homotopy class~$\mathcal{H}$ of analytic Anosov diffeomorphisms on $\T^2$ is path-connected,
so that in fact, in the cases $(i)$-$(iii)$ in Theorem \ref{thm:thmC}, any $T' \in \mathcal{H}$
is homotopic via a continuous path of toral diffeomorphisms to a $T \in \mathcal{H}$ with the respective spectral property.
Moreover, by the Franks-Newhouse classification theorem, every Anosov diffeomorphism of codimension $1$ on
a closed Riemannian manifold\footnote{An Anosov diffeomorphism is said to be of codimension $1$ if the codimension of
either its stable or unstable foliations has dimension $1$.} is topologically conjugated to
a hyperbolic toral automorphism, see for example \cite{Hi}.
Thus, since any two smooth surfaces that are homeomorphic are also diffeomorphic, in the above theorem $\Ttwo$ can be replaced with any two-dimensional closed Riemannian manifold.
\end{rem}

This paper is organised as follows.
Starting with a motivational example of the composition operator for the cat map in Section \ref{sec:pedestrian}, we discuss its
boundedness and compactness on an anisotropic space of hyperfunctions relevant to the current work.
In Section \ref{sec:cones_nonped} we consider a class of analytic toral Anosov diffeomorphisms with strongly expanding
constant invariant cone \textit{(sec)} fields and show how properties on the tangent bundle are translated into properties of the map in
a small complex neighbourhood of the torus. We start Section \ref{sec:comp_compact} by summarising properties of Hardy-Hilbert spaces
on log-conical Reinhardt domains and realizing the anisotropic Hilbert spaces of hyperfunctions from Section \ref{sec:pedestrian}
as direct sums of these Hardy-Hilbert spaces. The main theorem of this section is Theorem \ref{thm:boundedness_large2small},
which together with a standard factorisation argument yields a trace-class weighted composition operator for analytic
Anosov maps with the \textit{(sec)} property, thus proving Theorem \ref{thm:thmA}. Section \ref{sec:resonances_rational} is devoted
to the proof of Theorem \ref{thm:thmB}, that is, the computation of resonances of the composition operator
associated to rational Anosov diffeomorphisms satisfying the \textit{(p-sec)} condition and preserving certain polydisks. In Section \ref{sec:resonances_blaschke}
we prove Theorem \ref{thm:thmC}. For this, we first introduce in Section~\ref{sec:blaschke_examples} the group of diffeomorphisms $\mathcal{F}$, discuss their properties and present examples of Anosov maps with and without additional properties
such as area-preservation and symmetry reversal.
In Section \ref{sec:red_maps}, using conjugacy classes of toral Anosov automorphisms, we construct non-linear Anosov maps in $\mathcal{F}$
satisfying the assertions of Theorem \ref{thm:thmC}.

\section{Cone conditions for Anosov diffeomorphisms on \texorpdfstring{$\Ttwo$}{the two-torus}}\label{sec:cones}

\subsection{A pedestrian approach}\label{sec:pedestrian}
In this section we want to explore functional-analytic properties
of composition operators associated to Anosov automorphisms on the torus
on Hilbert spaces that can be defined as the completion of the space of Laurent polynomials with
a norm depending on a certain weight function. We take the well-known cat map
as an example of a toral automorphism, and discuss boundedness and compactness
of the associated composition operator depending on the weight function. The main
result of this section is to show that for a suitable weight function the composition operator is
Hilbert-Schmidt.
Since most of our work in this paper focuses on the two-dimensional case, for convenience
we introduce the following shorthands.

\begin{notation}
For $\alpha, \beta \in \Rtwo$ (or $\Z^2$) we write $\alpha > \beta$ if $\alpha_1 > \beta_1$ and $\alpha_2 > \beta_2$,
and moreover for $c \in \R$ the notation $\alpha > c$ will be used as a shorthand for $\alpha_1 > c$ and $\alpha_2 > c$
(and analogously for other comparison operators).
For $z \in \C^2$ (or $\R^2$, $\Z^2$) and $n \in \Z^2$ we use the multiindex notation $z^n = z_1^{n_1} z_2^{n_2}$, and
we write $|z| = (|z_1|, |z_2|)$.
\end{notation}

\subsubsection{Weighted Hilbert spaces}

Let $\T^2 = \{z \in \C^2: |z_1| = |z_2| = 1\}$, and let $\mathcal{P}$ denote the space of Laurent polynomials on $\T^2$,
\begin{equation*}
\label{calL}
\mathcal{P}=\{f:\mathbb{T}^2\to \mathbb{C}:
f(z)=\sum_{n \in \Z^2, |n|\leq N}f_nz^n,
\text{ with $f_n\in\mathbb{C}, N\in \mathbb{N}$} \}\,.
\end{equation*}
For any $\nu: \mathbb{Z}^2 \to \R_{>0}$, we define an inner product on $\mathcal{P}$ by
\[\langle f, g \rangle_\nu = \sum_{n\in\mathbb{Z}^2} f_n \bar{g}_n \nu(n)^2,  \]
where $(f_n)_{n\in\mathbb{Z}^2}$ and $(g_n)_{n\in\mathbb{Z}^2}$ are the Fourier coefficients
of $f$ and $g$, and we denote by $\|\cdot\|_\nu$ the corresponding norm.
Note that $\nu(n) = \|p_n\|_\nu$ for $n \in \Z^2$, where $p_n$ is the monomial $z\mapsto z^n$.

\begin{defn}\label{defn:Hnu}
We write $\mathcal{H}_\nu$ for the completion of $\mathcal{P}$ with
respect to the norm $\|\cdot\|_\nu$.
\end{defn}
It turns out that $\mathcal{H}_\nu$ is a separable Hilbert space, which contains the Laurent
polynomials as a dense subset, with an orthonormal basis given by the normalised monomials
\[e_n(z) = \frac{z^n}{\nu(n)}, \qquad n \in \Z^2.\]

\subsubsection{Composition operator for the cat map}

The hyperbolic matrix
\[A = \begin{pmatrix}
2 & 1 \\
1 & 1
\end{pmatrix}\] induces the well-known toral
automorphism $T \colon \T^2 \to \T^2$ given by $(z_1, z_2) \mapsto (z_1^2 z_2, z_1, z_2)$, known as the cat map.
We shall study the properties of the
composition operator $C_T$ associated to $T$ when considered on different Hilbert spaces
$H_\nu$. In particular, we will see that for certain choices of $\nu$ the operator $C_T$ is
Hilbert-Schmidt.
A bounded operator $L$ on a Hilbert space $H$ is Hilbert-Schmidt if
$\sum_n \| L e_n\|^2 < \infty$ where $\{e_n\}$ is an orthonormal basis of $H$.
A sufficient condition for this is that there are $\delta, c >0$ such that for all $n$ it holds that
\begin{align}\label{eq:Len}
\| L e_n \| \leq \delta \exp(-c \|n\|).
\end{align}
We can compute $\|C_T e_n\|_\nu$ explicity, as
\begin{align}\label{eq:Cen}
 \|C_T e_n\|^2_\nu = \sum_{m\in\mathbb{Z}^2} \left(\frac{\nu(m)}{\nu(n)}\right)^2 \delta_{2n_1+n_2, m_1}\delta_{n_1+n_2, m2}
 = \left(\frac{\nu(A n)}{\nu(n)}\right)^2.
\end{align}

We first consider weight functions $\nu$ adapted to the dynamics of $T$ introduced in \cite{FR}. For this,
we denote the unstable/stable eigenvalues of $A$ by
$\lambda_{u/s} = \varphi^{\pm2}$, where $\varphi = (1+\sqrt{5})/2$ is the golden mean, and
write $V = (v_{u}, v_{s})$ for the matrix with the corresponding
(normalised) unstable/stable eigenvectors as its column vectors. We write $\langle \cdot, \cdot\rangle$ for
the standard inner product on $\mathbb{R}^2$.

\begin{lem}\label{lem:cat_HS_old_space}
For any $a = (a_1, a_2) >0$ and $\nu(n) = \nu_a(n) = \exp(-a_1 |\langle n, v_u \rangle| + a_2 |\langle n, v_s\rangle|)$,
the composition operator $C_T$ is a well-defined Hilbert-Schmidt operator on $(H_\nu, \|\cdot\|_\nu)$.
\end{lem}

\begin{proof}
By using $\langle An, v_{u/s} \rangle = \langle n, A^T v_{u/s} \rangle = \lambda_{u/s}\langle n, v_{u/s} \rangle$ and
\eqref{eq:Cen}, we obtain
\begin{align*}
\|C_T e_n\|_\nu &=
\exp(-a_1 \lambda_u|\langle n, v_u \rangle| + a_2 \lambda_s |\langle n, v_s\rangle|) \cdot
\exp(a_1 |\langle n, v_u \rangle| - a_2 |\langle n, v_s\rangle|)\\
&= \exp(-a_1 (\lambda_u - 1)|\langle n, v_u \rangle| - a_2(1-\lambda_s)|\langle n, v_s\rangle|).
\end{align*}
As $(\lambda_u - 1)>0$, $(1-\lambda_s)>0$, and all norms on $\mathbb{R}^2$ are equivalent,
we obtain inequality \eqref{eq:Len}.
\end{proof}

It turns out that replacing $V$ by the identity matrix yields an operator that
is not even compact. We omit a proof, as this follows by a direct calculation.

\begin{lem}
Let $a = (a_1, a_2) \in \R^2_{>0}$ and $\nu(n) = \nu_a(n) = \exp(-a_1 |n_1| + a_2 |n_2|)$.
The operator $C_T$ is bounded on $(H_\nu, \|\cdot\|_\nu)$ if and only if
$a_1 = a_2$. It is never compact on $(H_\nu, \|\cdot\|_\nu)$.
\end{lem}

On the other hand, as will become apparent in the next section, working with a diagonal matrix $V$ yields more convenient
function spaces (specifically, Hilbert spaces of holomorphic functions on polydisks).
We can restore the nice properties of the compositon operator in this setting,
by allowing $a$ to be a function of $n\in\mathbb{Z}^2$.

\begin{defn}[Quadrant-wise exponential weight]\label{defn:weight}
Let $\alpha, \gamma \in \mathbb{R}^2$ and define $a\colon\mathbb{Z}^2\to \mathbb{R}^2$ as
\[a_{\alpha,\gamma}(n) =
\begin{cases}
    \alpha, & \text{if } n_1\cdot n_2 \geq 0\\
    \gamma, & \text{if }  n_1\cdot n_2 <0,
\end{cases}
\]
and $\nu(n) = \nu_{\alpha,\gamma}(n) = \exp(-\langle a_{\alpha,\gamma}(n), |n| \rangle)$.
We call such weight function {\it quadrant-wise exponential}.
\end{defn}

\begin{lem}\label{lem:main_cat}
Let $\nu = \nu_{\alpha,\gamma}$ be a quadrant-wise exponential weight.
If $\alpha \in \R^2_{>0}$ and $\gamma \in \R^2_{<0}$, then
the operator $C_T$ is a well-defined Hilbert-Schmidt operator on $H_\nu$.
\end{lem}

\begin{proof}
For $n\in\mathbb{Z}^2$ we denote
\[\varphi(n) = \langle a_{\alpha,\gamma}(n), |n|\rangle  -  \langle a_{\alpha,\gamma}(An), |An|\rangle,\]
so that
\[\|C_T e_n\|_\nu = \frac{\nu(A n)}{\nu(n)} = \exp(\varphi(n)).\]
To prove the lemma, it suffices to show that there exists $c>0$ such that
\begin{equation}\label{eq:hs}
\varphi(n) < -c (|n_1| + |n_2| ) \quad   (\forall n \in \mathbb{Z}^2).
\end{equation}
Set $m = An$. We note that since $\| n \|_2 \leq \| A^{-1} \| \| m \|_2$ and by the equivalence of norms in $\R^2$,
there exists $\tilde c > 0$ such that $|m_1| + |m_2| = \| m \|_1 \geq \tilde c \| n \|_1 = \tilde c (|n_1| + |n_2|)$, for all $n \in \Z^2$
and $m = An$.
There are three cases we need to take care of.
\begin{enumerate}[(i)]
  \item $n_1\cdot n_2 \geq  0$. Note that in this case $|m| = |An| = A|n|$ and $m_1\cdot m_2 \geq 0$. It follows that
  \[
  \varphi(n) = \langle \alpha, |n| - A|n| \rangle
  =  \langle \alpha - A^T \alpha, |n|\rangle =
  -(\alpha_1 + \alpha_2) |n_1| - \alpha_1|n_2| < -c_1 (|n_1| + |n_2|),
  \]
for any $0 < c_1 < \alpha_1$.

  \item $m_1 \cdot m_2 < 0$, which implies $n_1\cdot n_2 < 0$ as
  $n=A^{-1}m$ and $n_1 \cdot n_2 = (m_1 - m_2) (2m_2 - m_1) = 3m_1m_2 - 2m_2^2 - m_1^2 < 0$.
  Noting that $|n_1| = |m_1 - m_2| = |m_1| + |m_2|$ and $|n_2| = |2 m_2 - m_1 | = 2 |m_2| + |m_1|$ we obtain
  \[
\varphi(n) = \langle \gamma, |n| - |m| \rangle
= \gamma_1 |m_2| + \gamma_2 ( |m_1| + |m_2|) < -c_2 (|n_1| + |n_2|),
  \]
  for any $ 0 < c_2 < -\gamma_2 \tilde c$.

\item $n_1\cdot n_2 < 0$ and $m_1\cdot m_2 \geq 0$.
In this case
\[
\varphi(n) = \langle \gamma, |n| \rangle - \langle \alpha, |m| \rangle
\leq \max(\gamma_1, \gamma_2) (|n_1| + |n_2|) - \min(\alpha_1, \alpha_2) (|m_1| + |m_2|) < -c_3 (|n_1| + |n_2|),
\]
for any $0 < c_3 < 2 \min(-\gamma_1, -\gamma_2, \alpha_1\tilde c , \alpha_2\tilde c)$.
\end{enumerate}
Combining the above, we have that \eqref{eq:hs} holds for all $n \in \mathbb{Z}^2$ with $c = \min(c_1, c_2, c_3)$.
\end{proof}

\begin{rem}
It turns out that the use of quadrant-wise exponential weight functions is not
restricted to linear toral Anosov diffeomorphisms, but can also be applied to certain
non-linear maps.
It is possible to show that for any map in the family of non-linear maps studied
in \cite{SlBaJu_NONL17}, one can find suitable $\alpha$ and $\gamma$, such that the associated
composition operator considered on the Hilbert space $H_\nu$ with $\nu = \nu_{\alpha,\gamma}$ is Hilbert-Schmidt,
and even trace-class.
The proof is a straightforward but lengthy calculation involving properties of the underlying map summarized
in \cite[Lemma~2.3]{SlBaJu_NONL17}. See also \cite{PoS} for this and related results.
\end{rem}

\subsection{Anosov diffeomorphisms with strong mapping conditions}\label{sec:cones_nonped}
In this section, we establish more general conditions on the toral diffeomorphisms which will
be sufficient to prove that the associated composition operator is trace-class on a suitable
weighted Hilbert space. In particular, we characterise maps that satisfy the conditions
of the main result of the next section (Theorem \ref{thm:boundedness_large2small}).
We start by recalling some well-known facts about cones and
Reinhardt domains.

\subsubsection{Convex cones in $\R^2$}

A \emph{cone} $\Lambda \subset \Rtwo$ is a set such that if $v\in \Lambda$, then $\lambda v \in \Lambda$ for all $\lambda > 0$.
A cone shifted by a vector, that is a set of the form $x + \Lambda$, where $x \in \R^2$ and $\Lambda \subset \Rtwo$ is a cone,
is called an \emph{affine cone}.
For $p_{u}, p_{s} \in \mathbb{R}^2$ denote by $P = (p_u, p_s)$ the matrix having $p_u$
and $p_s$ as its column vectors. For an invertible matrix $P$ denote by
$\Lambda_P$ the convex open polyhedral cone in $\mathbb{R}^2$ (positively) spanned by $p_u$ and $p_s$, that is,
\[\Lambda_P = \{P x = x_1 p_u + x_2 p_s \colon x > 0\} = P(\R^2_{>0}) .\]
Writing $W = (w_u, w_s) = (P^T)^{-1}$, this can equivalently be expressed as
\[\Lambda_P = \{x \in \Rtwo\colon \langle w_u, x\rangle > 0, \langle w_s, x\rangle > 0\}.\]
Its {\it polar cone} is given by
\begin{equation*}
(\Lambda_P)^o
= \{x \in \R^2\colon \langle x, p\rangle \leq 0 \text{ for all } p\in \Lambda_P\}
= \{x \in \R^2\colon P^T x \leq 0 \}
= W(\R^2_{\leq 0}).
\end{equation*}
While $\Lambda_P$ is the image of the positive quadrant of $\Rtwo$ under $P$, it will later be useful
to consider the images of all quadrants. For this, we write
$I^\sigma = \begin{pmatrix} \sigma_1 & 0 \\ 0 & \sigma_2 \end{pmatrix}$
for $\sigma \in \Sigma,$ where\footnote{For brevity, when using $\sigma = (\sigma_1, \sigma_2) \in \Sigma$
as an index, we will often just write out its signs,
e.g.~writing $R^{++}$ for $R^{(+1, +1)}$, $R^{+-}$ for $R^{(+1, -1)}$, etc.}
\[\Sigma := \{\sigma = (\sigma_1, \sigma_2)\colon \sigma_1, \sigma_2 \in \{\pm1\}\},\]
and denote by $\Lambda^{\sigma}_P$ the image of the quadrant $R^\sigma = I^\sigma (\R^2_{>0})$
under $P$, that is $\Lambda^{\sigma}_P = P(R^\sigma)$ for $\sigma \in \Sigma$. A short calculation reveals that
the cone $\Lambda^{\sigma}_P$ can be written as
\[\Lambda^{\sigma}_P = \{x \in \Rtwo\colon \sigma_1\langle w_u, x\rangle > 0, \sigma_2\langle w_s, x\rangle > 0\}.\]

All the above cones have $(0, 0)$ as their apex. As we shall see shortly we will need to
work with cones translated by a vector. For this, let us denote by $R^\sigma_\delta$ the image under $I^\sigma$
of the first quadrant translated by $\delta\in \Rtwo$, that is $R^\sigma_\delta = I^\sigma(\R^2_{>0} + \delta) = R^\sigma + v_\delta^\sigma$
with apex $v_\delta^\sigma = I^\sigma \delta$.

\begin{defn}\label{defn:cone_vertex}
For $\delta \in\R^2$, $\sigma\in \Sigma$ and $P\in \GLtwo(\R)$
we denote $\Lambda^\sigma_{P,\delta}$ the convex affine cone
\begin{align}\label{eq:general_cone}
\Lambda^\sigma_{P,\delta} = P(R_\delta^\sigma) = \Lambda_P^\sigma + v^\sigma_{P,\delta},
\end{align}
with apex $v^\sigma_{P,\delta} = P v^\sigma_{\delta}$.
\end{defn}
We note that $\sigma_1 \langle w_u, v^\sigma_{P,\delta}\rangle = \sigma_1 \langle P^T w_u, I^\sigma \delta \rangle
= \sigma_1 \langle (1, 0)^T, I^\sigma \delta \rangle = \delta_1$
and $\sigma_2 \langle w_s, v^\sigma_{P,\delta}\rangle = \delta_2$.
Using these equalities, $\Lambda^\sigma_{P, \delta}$ can be rewritten as
\begin{align}\label{eq:shifted_general_cone}
\begin{split}
\Lambda^\sigma_{P, \delta}
&= \{x \in \Rtwo\colon \sigma_1\langle w_u, x -  v^\sigma_{P,\delta}\rangle > 0, \sigma_2\langle w_s, x - v^\sigma_{P,\delta}\rangle > 0\} \\
&= \{x \in \Rtwo\colon \sigma_1\langle w_u, x\rangle  > \delta_1, \sigma_2\langle w_s, x \rangle  > \delta_2\}.
\end{split}
\end{align}

\subsubsection{Log-conical Reinhardt domains of $\Ctwohat$}

The translated convex cones in \eqref{eq:general_cone} will be useful for defining certain two-dimensional complex domains.
For this we first require some definitions.
\begin{notation}
We denote the Riemann sphere by $\Chat = \C \cup \{\infty\}$, and write $\Chat^2 = \Chat \times \Chat$.
For $z\in \Ctwohat$, $v\in \Rtwo$ and $a \in \Z^2$ we write $|z| = (|z_1|, |z_2|)$, $z^a = z_1^{a_1} z_2^{a_2}$, $e^v = (e^{v_1}, e^{v_2})$,
and for $z\in (\mathbb{C}\setminus\{0\})^2$ we write $\log |z| = (\log|z_1|, \log|z_2|)$.
For any domain $D \subseteq \Ctwohat$, we write $D^+ =  D \cap (\mathbb{C}\setminus\{0\})^2$.
\end{notation}

\begin{defn}
A domain $D \subset \Ctwohat$ is called {\it polycircular} or a {\it Reinhardt domain} if it is
invariant under polyrotations, that is, if $z\in D$ implies $\omega z = (\omega_1 z_1, \omega_2 z_2) \in D$ for all $\omega \in \Ttwo$.
The set
$|D| := \{|z|\colon z\in D\} \subset (\R_{\geq 0}\cup \{\infty\})^2 $ is called {\it absolute domain} of $D$,
and $\Lambda = \log |D^+|:=\{\log |z|: z\in D^+\} \in \Rtwo$ the {\it logarithmic base} of
$D^+$. A Reinhardt domain $D$ is called {\it log-conical} if the
 logarithmic base of $D^+$ is a convex open affine cone.
\end{defn}

We define $\T^2_\rho = \{z \in \C^2: |z_1|=\rho_1, |z_2| = \rho_2\}$ for $\rho \in \R^2_{>0}$, and
write $e^\Lambda \Ttwo:= \bigcup_{r\in\Lambda} \mathbb{T}^2_{e^r}$.
Further let $\mathcal{E}(\Lambda) \subset \Rtwo$ denote the union of the set of \emph{faces} of $\Lambda$,
that is $\mathcal{E}(\Lambda) = \partial \Lambda \setminus \{p\}$ with $p$ the apex of $\Lambda$.
Then, every convex open affine cone $\Lambda\subset \Rtwo$ induces a log-conical Reinhardt domain in $\Chat^2$ via
\[
D= \cl{e^\Lambda \Ttwo}\setminus e^{\mathcal{E}(\Lambda)} \Ttwo,
\]
where the closure is taken in $\Chat^2$ (that is, it may contain points of the form $(z_1, z_2)$,
with either $z_1$, or $z_2$, or both, taking the values $0$ or $\infty$).
We shall denote by $D^\sigma_{P,\delta}$ the log-conical Reinhardt domain induced by the convex affine cone
$\Lambda^\sigma_{P, \delta}$ in \eqref{eq:general_cone}.
We can calculate
\begin{align*}
(D^{\sigma}_{P,\delta})^+ &= e^{\Lambda^\sigma_{P,\delta}} \mathbb{T}^2 =
\bigcup_{r\in\Lambda^{\sigma}_{P, \delta}} \mathbb{T}^2_{e^r} =
\bigcup_{x\in\Lambda^\sigma_P} \mathbb{T}^2_{\exp(v^\sigma_{P,\delta}+x)}\\
&= \{z\in(\C\setminus \{0\})^2 \colon |z_1| = e^{(v^\sigma_{P,\delta})_1 + x_1}, |z_2| = e^{(v^\sigma_{P,\delta})_2 + x_2},\, x\in \Lambda^\sigma_P\}\\
&= \{z\in(\C\setminus \{0\})^2 \colon
|z|^{\sigma_1w_u} >  e^{\delta_1}, |z|^{\sigma_2 w_s} > e^{\delta_2} \},
\end{align*}
which yields
\[
D^\sigma_{P,\delta} = \{z\in\Chat^2 \colon |z|^{\sigma_1w_u} >  e^{\delta_1}, |z|^{\sigma_2 w_s} > e^{\delta_2} \}.
\]
The {\it distinguished boundary} or {\it Shilov boundary} of $D^{\sigma}_{P, \delta}$ is a torus in $\Ctwo$
given by
\[
\mathbb{T}_{P, \delta}^\sigma:=\partial^*D^{\sigma}_{P,\delta} :=
\{z\in\mathbb{C}^2 \colon |z_1| = e^{(v^\sigma_{P,\delta})_1},  |z_2| = e^{(v^\sigma_{P,\delta})_2}\}.
\]

\subsubsection{Anosov toral diffeomorphisms with constant invariant cone fields}\label{sec:sec_cond}

\begin{defn}[Toral Anosov diffeomorphisms]\label{def:Anosov}
Let $M = ([0, 2\pi]/\sim)^2$. A smooth diffeomorphism $\tilde{T}\colon M\to M$ is called {\it Anosov} if
there exist two uniformly transversal open
continuous cone fields $\mathcal{K}^u = \{K^u(x)\}$,
$\mathcal{K}^s = \{K^s(x)\}$ with cones $K^u(x), K^s(x) \subset T_xM$,
a norm $\| \cdot \|$ on $T_xM$ and $\lambda>1$ such that, for all $x \in M$,
\begin{enumerate}[(i)]
  \item $D_x\Tilde{T}(\cl{K^u(x)}) \subset K^u(\tilde{T}(x)) \cup \{0\},
  D_x\Tilde{T}^{-1} (\cl{K^s(x)}) \subset K^s(\tilde{T}^{-1}(x)) \cup \{0\}$ and
  \item $\| D_x\tilde{T} (v)\| > \lambda \|v\|\; \forall v \in \cl{K^u(x)} \text{ and }
  \| D_x\tilde{T}^{-1} (v)\| > \lambda \|v\|\; \forall v \in \cl{K^s(x)}$.
\end{enumerate}
Without loss of generality, the cone fields can be chosen to be complementary, that is, $K^s(x) = T_xM \setminus \cl{K^u(x)}$
for all $x \in M$.

If the expanding and contracting cones $K^u(x)$ and $K^s(x)$ can
be chosen independently of $x$, that is\footnote{Here we slightly abuse notation, using the canonical identification $T_xM \cong \mathbb{R}^2$ for all $x \in M$.} $K^u, K^s\subset T_xM$ such that $K^u(x) = K^u$ and
$K^s(x) = K^s$ for all $x\in M$, then we say $\tilde{T}$ is an {\it Anosov diffeomorphism with constant invariant cone fields}.
\end{defn}

Let $\tilde{T}$ be a toral Anosov diffeomorphism with constant complementary cone fields $\mathcal{K}^u = \{K^u\}$ and
$\mathcal{K}^s = \{K^s\}$.
Then $K^u$ can be decomposed as $K^u = K^u_+ \cup -K^u_+$, where $K^u_+$ is a convex cone, so there exists a matrix
$P \in \GLtwo(\R)$ such that $K^u_+ = \Lambda^\sigma_P = P I^\sigma (\R^2_{>0})$ with $\sigma=(+1, +1)$.
Adopting the notation
\[\Sigma^{1} = \{\sigma \in \Sigma: \sigma_1 = \sigma_2\} \quad \text{ and } \quad
\Sigma^{-1} = \{\sigma \in \Sigma : \sigma_1 = -\sigma_2\} \]
we can write $K^u = \bigcup_{\sigma \in \Sigma^1} \Lambda^{\sigma}_{P}$ and
$K^s = \bigcup_{\tilde{\sigma} \in \Sigma^{-1}} \Lambda^{\tilde{\sigma}}_P$.

\begin{defn}\label{def:strongly_exp}
Let $\tilde{T}:M\to M$ be a smooth Anosov diffeomorphism.
\begin{enumerate}[(i)]
\item We say that $\tilde{T}$ has {\it $P$-induced constant invariant cone fields} if
it has constant invariant cone fields $\mathcal{K}^u = \{K^u_0\}$ and $\mathcal{K}^s = \{K^s_0\}$
given by $K^u_{\delta} = \bigcup_{\sigma \in \Sigma^1} \Lambda^{\sigma}_{P, \delta}$ and{}
 $K^s_\delta = \bigcup_{\tilde{\sigma} \in \Sigma^{-1}} \Lambda^{\tilde{\sigma}}_{P, \delta}$ with
 $P\in \GLtwo(\R)$ and $\delta\in \Rtwo$.
\item We say that $\tilde{T}$ has {\it $P$-induced strongly expanding constant invariant cone fields \textit{(sec)}} if it has
$P$-induced constant invariant cone fields and if there are
$\delta, \tilde{\delta} \in \R^2_{>0}$ and $\sigma \in \Sigma^1$, $\tilde \sigma \in \Sigma^{-1}$ such that, for all $x\in M$,
\begin{align}\label{eq:secc}
(D_x\tilde{T}) v \in K^u_\delta \text{ and } (D_x\tilde{T}^{-1}) \tilde{v} \in  K^s_{\tilde{\delta}},
\end{align}
where $v = v^\sigma_{P,\delta}$, $\tilde v = v^{\tilde \sigma}_{P,\tilde \delta}$. In the special case $P = \mathbb{I}$, we
refer to $K_0^u = \R^2_{>0} \cup \R^2_{<0}$ as {\it positive}, and call \eqref{eq:secc} the {\it (p-sec)} condition.
\end{enumerate}
\end{defn}

Figure~\ref{fig:secc_cond} illustrates the {\it (p-sec)} condition.
Note that (i) is satisfied by all smooth toral Anosov diffeomorphisms with constant invariant cone fields,
whereas (ii) is a stronger condition, which eventually will be required to prove compactness
of the associated weighted composition operator (see Section \ref{sec:comp_compact}).

\begin{figure}[ht!]
  \centering
  \includegraphics[width=0.4\textwidth]{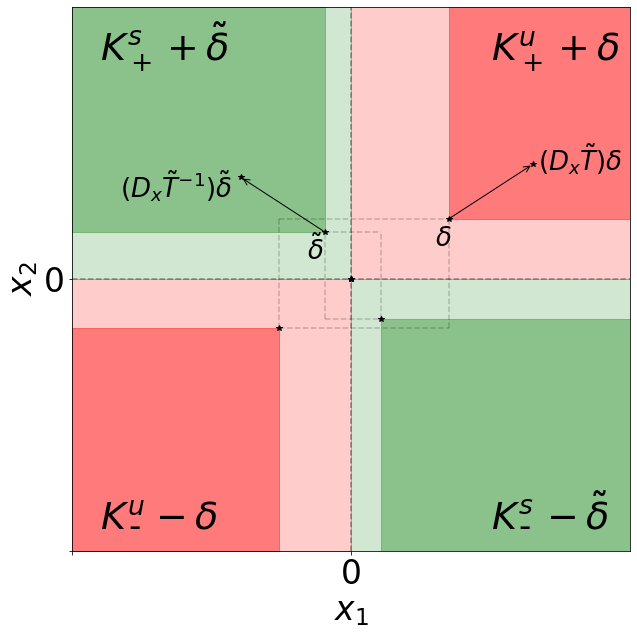}
  \caption{Illustration of the {\it (sec)} condition in Definition \ref{def:strongly_exp}(ii)
  for $P=\mathbb{I}$ and $K^{u}_\delta = K^{u}_{\pm} \pm \delta$.}
  \label{fig:secc_cond}
\end{figure}

\begin{rem}\label{rem:secc} The cone conditions in \eqref{eq:secc} are often easy to check.
For example, the {\it (p-sec)} condition requires $K^u_0 = \R^2_{>0} \cup \R^2_{<0}$ and $K^s_0=\Rtwo \setminus \cl{K^u_0}$.
Assuming $D_x\tilde{T} = \begin{pmatrix} a_x & b_x \\ c_x & d_x \end{pmatrix}$ preserves $\R^2_{\geq 0}$ for all $x \in M$,
the first part of \eqref{eq:secc} is automatically satisfied if either $\inf_{x \in M}{a_x} \geq 1$, or $\inf_{x\in M}{d_x} \geq 1$,
or if
\[\sup_{x \in M} \left(\frac{1-a_x}{b_x}\right) \cdot \sup_{x\in M} \left(\frac{1-d_x}{c_x}\right) < 1.\]
A similar condition for the first part of \eqref{eq:secc} in the case when $D_x \tilde T$ maps $\R^2_{\geq 0}$ to $\R^2_{\leq 0}$
can be deduced easily, as can be analogous conditions for the second part of \eqref{eq:secc}.
\end{rem}

We shall next state conditions that can be easily deduced from the fact that the map $\tilde{T}$ possesses
constant invariant expanding and (co)-expanding cone fields, that is, from (i) in Definition \ref{def:strongly_exp}.

\begin{lem}\label{lem:exists_q}
Let $\tilde{T}\colon M\to M$ be a smooth Anosov diffeomorphism with $P$-induced constant invariant cone fields for some $P \in \GLtwo(\R)$.
Then, for any $\delta, \tilde{\delta} \in \R^2_{>0}$, $\sigma \in \Sigma^1$ and $\tilde{\sigma} \in \Sigma^{-1}$, there exists $q\in \Lambda^{\tilde{\sigma}}_{P, \tilde{\delta}}$ with $(D_x \tilde{T})
q \in \Lambda^{\sigma}_{P, \delta}$ for all $x\in M$.
\end{lem}

\begin{proof}
Fix $\delta, \tilde{\delta} \in \R^2_{>0}$, $\sigma \in \Sigma^1$ and $\tilde{\sigma} \in \Sigma^{-1}$,
and recall that the (constant) $(D_x \tilde{T})$-invariant cone given by the matrix $P$ is
$K^u = \Lambda_P \cup -\Lambda_P$, with $\Lambda_P = P(\R^2_{>0})$,
and either $(D_x \tilde{T})(\Lambda_P) \subset \Lambda_P$ or $(D_x \tilde{T})(\Lambda_P) \subset -\Lambda_P$ (for all $x \in M$).
For $x \in M$, we define $N_x = P^{-1} (D_x \tilde{T}) P$, and observe that either
$N_x (\R^2_{\geq 0}) \subset \R^2_{>0} \cup \{0\}$ or
$N_x (\R^2_{\geq 0}) \subset \R^2_{<0} \cup \{0\}$, for all $x \in M$.

Applying Lemma \ref{lem:matrix} to $\{N_x: x\in M\}$, there exists $q' \in R^{\tilde \sigma}_{\tilde \delta}$
such that $N_x(q') \in R^\sigma_\delta$ for all $x \in M$. Using \eqref{eq:general_cone},
we obtain that $P q'\in P(R^{\tilde \sigma}_{\tilde \delta}) = \Lambda^{\tilde \sigma}_{P,\tilde \delta}$,
and $ (D_x \tilde{T}) (P q') \in P(R^\sigma_\delta) = \Lambda^\sigma_{P,\delta}$ for all $x \in M$,
finishing the proof with $q = Pq'$.
\end{proof}

Let $\pi(x_1, x_2) = (e^{ix_1}, e^{ix_2})$ be the canonical diffeomorphism from
$M$ to $\Ttwo$. We denote by $T$ the diffeomorphism on $\Ttwo$ uniquely determined
by the equation $T \circ \pi = \pi \circ \tilde{T}$.
We will say that $T \colon \T^2 \to \T^2$ satisfies the {\it (sec)} (or {\it (p-sec)}) condition, if the corresponding $\tilde T \colon M \to M$ does.
As the next (key) lemma will show, for
analytic toral Anosov diffeomorphisms with constant invariant cone fields, properties of the derivative $D \tilde{T}$ on the tangent bundle can
be translated into properties of the map $T$ in a small neighbourhood of $\Ttwo$.
With slight abuse of notation, we continue writing $T$ for its analytic extension to such small neighbourhood of $\Ttwo$.

\begin{defn}\label{def:orient_area}
We say that a diffeomorphism $T \colon \Ttwo \to \Ttwo$ is \emph{orientation-preserving} (\emph{orientation-reversing})
  if the determinant of $D_x \tilde T$ of the conjugated map $\tilde T \colon M \to M$
  is positive (negative) for all $x \in M$. We say $T$ is \emph{area-preserving} if $| \det D_x \tilde T | = 1$ for all $x \in M$.
\end{defn}

\begin{lem}\label{lem:lin_to_nonlin}
Let $\tilde{T}\colon M\to M$ be an analytic Anosov diffeomorphism with $P$-induced constant invariant cone fields for some $P \in \GLtwo(\R)$, and $T$
the conjugated diffeomorphism on $\Ttwo$.
If there are $\delta\in \R^2_{>0}$,
$\sigma \in \Sigma$, and $q \in \Rtwo$ such that
$(D_x \tilde{T}) q \in \Lambda^\sigma_{P, \delta}$ for all $x\in M$, then there exist $\epsilon > 0$ and $\delta' > \delta$ such that
for all $t\in(0, \epsilon)$ we have \[T(\mathbb{T}^2_{e^{tq}}) \subset D^{\sigma}_{P, t \delta'},\]
where $D^{\sigma}_{P, t \delta'}$ is the Reinhardt domain induced by the logarithmic base $\Lambda^{\sigma}_{P, t \delta'}.$
\end{lem}

\begin{proof}

We denote $w^{(1)} = \sigma_1 w_u$ and $w^{(2)} = \sigma_2 w_s$ (recalling $W = (w_u, w_s) = (P^T)^{-1}$).
Using \eqref{eq:shifted_general_cone}, we observe that $(D_x \tilde{T}) q \in \Lambda^\sigma_{P, \delta}$ translates to the inequalities
\[
\sum_{l=1,2} \sum_{k=1,2} \frac{\partial \tilde{T}_l}{\partial x_k} q_k w^{(j)}_l > \delta_j, \quad \text{ for } j=1, 2 \text { and all } x \in M.
\]
Next, for $x \in M$,  $s \in \mathbb{R}$ and $a \in \mathbb{R}^2$, we define
\[
f_x^a(s) = \sum_{l=1,2} a_l \log | T_l(\xi_x(s)) |
\]
with $\xi_x(s) = (e^{sq_1+ix_1}, e^{sq_2+ix_2}) \in \mathbb{T}^2_{e^{sq}}$.
We note that $f^a_x$ is continuously differentiable on
$(0, \epsilon')$ for sufficiently small $\epsilon'$, with
\[
\frac{\partial}{\partial s} f^a_x(s) =
\Re \left( \sum_{l=1,2} \sum_{k=1,2}
\frac{e^{ix_k}{\partial_k} {T_l (\xi_x(s))}}{T_l(\xi_x(s))} q_k e^{sq_k} a_l
\right),
\]
where $\partial_k T_l$ denotes the (complex) derivative of the $l$-th component of $T$ with respect to the $k$-th variable.
It follows that
\[
\frac{\partial}{\partial s} f^{w^{(j)}}_x(s)\vert_{s=0}
= \Re \left(\sum_{l=1,2} \sum_{k=1,2} \frac{e^{ix_k}{\partial_k} {T_l (e^{ix_1}, e^{ix_2})}}{T_l(e^{ix_1}, e^{ix_2})} q_k w^{(j)}_l \right)
= \Re \left(\sum_{l=1,2} \sum_{k=1,2} \frac{\partial \tilde{T}_l}{\partial x_k} q_k w^{(j)}_l \right) > \delta_j
\]
for $j=1, 2$ and all $x \in M$.
By compactness of $M$ and continuity, we can fix $\delta' > \delta$, so that $\frac{\partial}{\partial s} f^{w^{(j)}}_x(s) > \delta'_j$, $j=1,2$,
holds for all $x \in M$ and all
$s \in (0, \epsilon'')$ with a sufficiently small $\epsilon'' \in (0, \epsilon')$.
Since $f^a_x(0) = 0$, we obtain that
\[
\sum_{l=1,2} w_l^{(j)}\log|T_l(\xi_x(t))| = f_x^{w^{(j)}}(t) - f_x^{w^{(j)}}(0)
=  \int_0^{t}  \frac{\partial}{\partial s} f_x^{w^{(j)}}(s) \,ds > \delta'_j t,
\]
for $j=1,2$ and $t \in (0, \epsilon'')$. Exponentiating both sides, it follows that $T(\xi_x(t)) \in D^\sigma_{P,t\delta'}$
for all $x \in M$. Since $\{\xi_x(t): x \in M\} = \mathbb{T}^2_{e^{tq}}$, this completes the proof.
\end{proof}

If in addition $\tilde{T}$ satisfies the {\it (sec)} condition,
then one can establish all the required properties (for Theorem \ref{thm:thmA})
for $T$ in a neighbourhood of $\Ttwo$. First, we need to introduce corresponding notions on Reinhardt domains
induced by convex cones.
For $\delta \in \R^2$ and $P\in \GLtwo(\R)$, let $\mathcal{A}_{P, \delta}$ denote a two-dimensional `annulus' containing
$\Ttwo$, given by
\[\mathcal{A}_{P,\delta} =
\{z\in \Ctwo \colon -|\delta| < P^{-1} \log |z| < |\delta|\}.\]
For brevity we also write $\mathcal{A}_\delta = \mathcal{A}_{\mathbb{I},\delta} = \{ z \in \Ctwo \colon -|\delta| < \log |z| < |\delta|\}$.

\begin{defn}\label{defn:sem}
Let $T\colon\Ttwo \to \Ttwo$ be an analytic map with holomorphic extension to a neighbourhood of $\cl{\mathcal{A}_{P,\delta}}$ for
some $\delta \in \R^2_{>0}$ and $P\in \GLtwo(\R)$.
For $\ell\in\{1, -1\}$ and $\Delta\in\R^2_{>0}$ the map $T$
is said to have
the \textit{($\ell, \delta, \Delta, P$)-strongly expanding mapping} property if
one of the following two alternatives holds:
\begin{enumerate}[(EP)]
\item[(EP)] $T(\TsdP) \subseteq D^\sigma_{P,\Delta} \cc D^\sigma_{P, \delta}$ for all $\sigma \in \Sigma^\ell$
(``$T$ preserves expanding direction''),
\item[(ER)] $T(\TsdP) \subseteq D^{-\sigma}_{P,\Delta} \cc D^{-\sigma}_{P,\delta}$ for all $\sigma \in \Sigma^\ell$
(``$T$ reverses expanding direction'').
\end{enumerate}
\end{defn}

\begin{prop}\label{prop:strongexp}
Let $\tilde{T}\colon M \to M$ be an analytic Anosov diffeomorphism with
constant $P$-induced strongly expanding invariant cone fields for some $\delta, \tilde{\delta}\in \R^2_{>0}$ in \eqref{eq:secc}.
Then there exists $\eta \in \R^2_{>0}$ such that the corresponding diffeomorphisms $T$ and $T^{-1}$ on $\mathbb{T}^2$
can be analytically extended to $\mathcal{A}_{P,\eta}$, and there are $\Delta, \tilde{\Delta} \in \R^2_{>0}$
with $\Delta > \delta$ and $\tilde{\Delta} < \tilde{\delta}$, such that
\begin{enumerate}[(i)]
  \item $T$ is $(1, t\delta, t\Delta, P)$-strongly expanding for all sufficiently small $t>0$,
  \item $T^{-1}$ is $(-1, t\tilde{\Delta}, t\tilde{\delta}, P)$-strongly expanding for all sufficiently small $t>0$.
\end{enumerate}
Moreover, for any $\sigma \in \Sigma^1, \tilde{\sigma} \in \Sigma^{-1}$ and any $\tilde{\delta}, \Delta \in \R^2_{>0}$,
\begin{enumerate}[(i)]
  \setcounter{enumi}{2}
  \item there exists $q\in \Lambda^{\tilde{\sigma}}_{P,\tilde\delta}$ such that
  $\mathbb{T}^2_{e^{tq}} \subset D^{\tilde{\sigma}}_{P, t\tilde{\delta}}$ and
  $T(\mathbb{T}^2_{e^{tq}}) \subset D^{\sigma}_{P, t \Delta}$
 for all sufficiently small $t$.
\end{enumerate}
\end{prop}

\begin{proof}
To prove $(i)$, we apply Lemma \ref{lem:lin_to_nonlin} with $q = v^\sigma_{P,\delta}$, $\sigma \in \Sigma^1$.
Since $tq = t v^\sigma_{P,\delta} = v^\sigma_{P,t\delta}$ and therefore $\mathbb{T}^2_{e^{tq}} = \mathbb{T}^\sigma_{P, t\delta}$,
there exist $\epsilon > 0$ and $\Delta > \delta$, such that for all $t \in (0, \epsilon)$,
$T(\mathbb{T}^\sigma_{P, t\delta}) \subset D^\sigma_{P,t\Delta} \cc D^\sigma_{P,t \delta}$.
Item $(ii)$ follows analogously by applying Lemma \ref{lem:lin_to_nonlin} with $\tilde q = v^{\tilde\sigma}_{P,\tilde\delta}$,
$\tilde \sigma \in \Sigma^{-1}$, to the map $T^{-1}$.
Finally, for $(iii)$, by Lemma \ref{lem:exists_q} there exists $q \in \R^2$
satisfying $tq \in \Lambda^{\tilde\sigma}_{P,t\tilde\delta}$ (i.e. $\T^2_{e^{tq}} \subset D^{\tilde\sigma}_{P,t\tilde\delta}$) and
$D_x \tilde T (tq) \in \Lambda^\sigma_{P,t\Delta}$ for all $t > 0$ and  $x \in M$.
Then, by Lemma \ref{lem:lin_to_nonlin}, there exists $\Delta' > \Delta$ and $\epsilon' > 0$ such that
$T(\T^2_{e^{tq}}) \subset D^\sigma_{P,t\Delta'} \subset D^\sigma_{P,t\Delta}$ for all $t \in (0, \epsilon')$,
as required.
\end{proof}

\begin{cor}\label{cor:exp_const_cone_prop}
If an analytic Anosov diffeomorphism $\tilde{T}$ on $M$ satisfies the {\it (sec)} condition for some $P \in \GLtwo(\R)$,
then there exist $\alpha, A, \gamma, \Gamma, \eta \in \R^2_{>0}$ with $\alpha < A < \eta$ and $\Gamma < \gamma < \eta$,
such that the corresponding diffeomorphisms $T$ and $T^{-1}$
on $\Ttwo$ can be analytically extended to $\mathcal{A}_{P,\eta}$ and the following mapping properties hold:
\begin{enumerate}[(i)]
  \item $T$ is  $(1, \alpha, A, P)$-strongly expanding;
  \item $T^{-1}$ is $(-1, \Gamma, \gamma, P)$-strongly expanding;
  \item For any $\sigma \in \Sigma^1, \tilde{\sigma} \in \Sigma^{-1}$, there exists $q$ such that $\mathbb{T}^2_{e^q} \subset
  D^{\tilde{\sigma}}_{P, \gamma} \cap \mathcal{A}_{P, \eta}$ and $T(\mathbb{T}^2_{e^q}) \subset D^{\sigma}_{P, A}$.
\end{enumerate}
\end{cor}

\begin{proof}
Using notation from Proposition \ref{prop:strongexp}, $(i)$-$(iii)$ are satisfied with
$\alpha = t \delta$, $A = t \Delta$, $\Gamma = t \tilde \Delta$, $\gamma = t \tilde \delta$, for sufficiently small $t > 0$.
\end{proof}

\section{Composition operator for Anosov maps with constant invariant cone fields} \label{sec:comp_compact}

In this section, we shall consider (weighted) composition operators associated to analytic Anosov diffeomorphisms satisfying the {\it (sec)} condition.
In this setting, using Corollary \ref{cor:exp_const_cone_prop}, we will show that there exist
Hardy-Hilbert spaces induced by a suitable cone-wise exponential weight, such that the operator is trace-class.

\subsection{Hardy-Hilbert spaces on Reinhardt domains}

Before moving to the more general case of Reinhardt domains,
we present a few simplified examples in which the domains are chosen to be polydisks.

\begin{figure}[ht!]\
  \centering
  \includegraphics[width=0.49\textwidth]{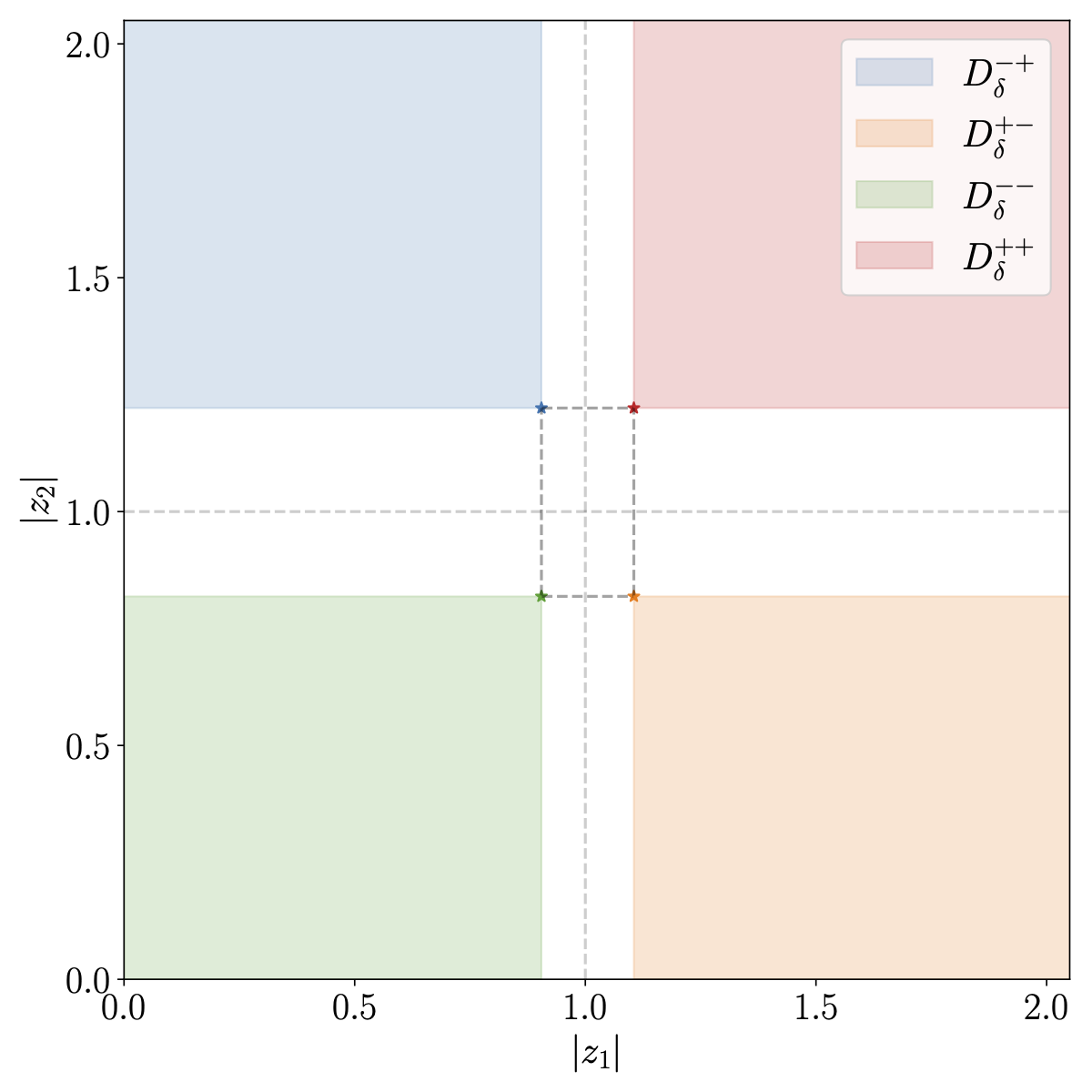}
  \includegraphics[width=0.49\textwidth]{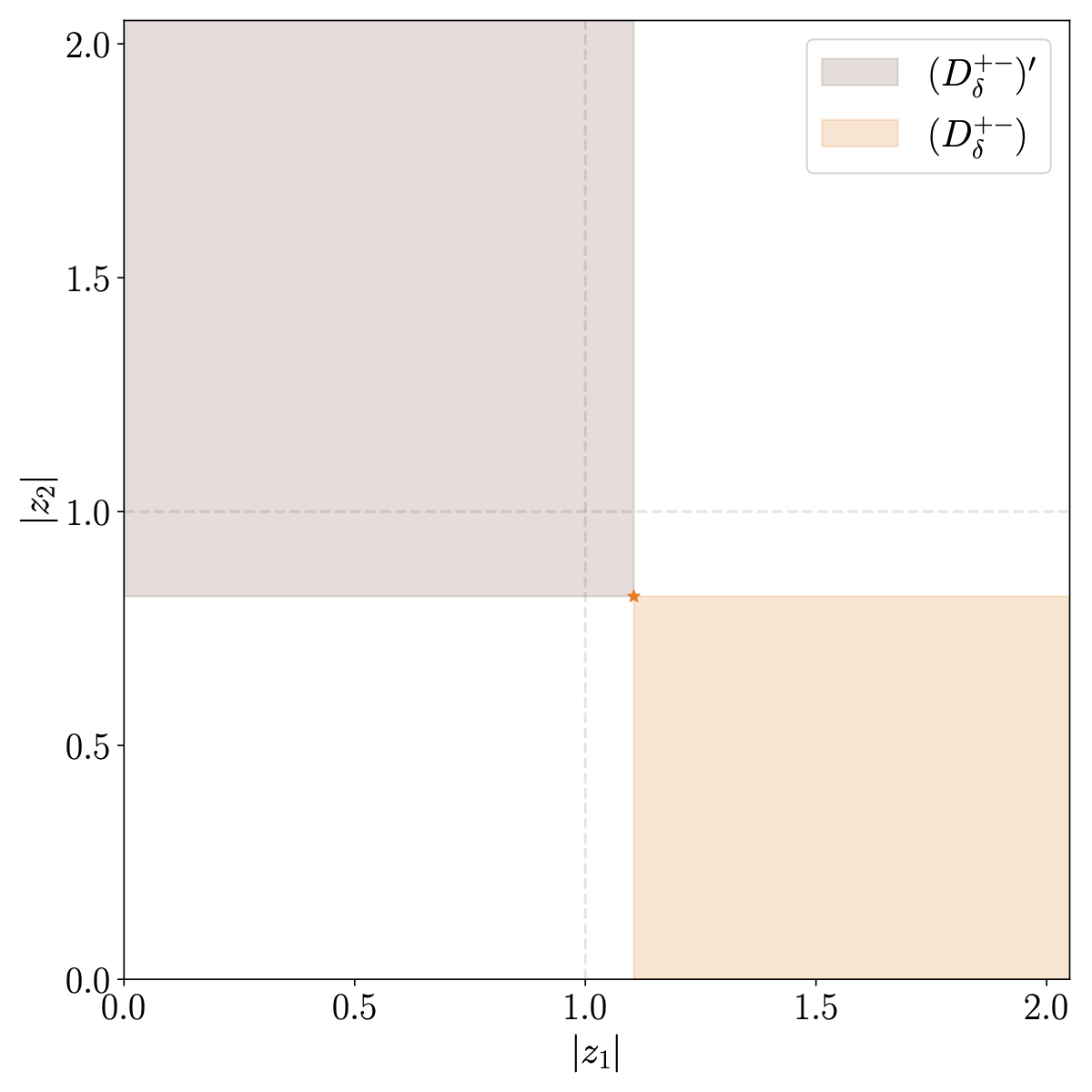}
  \caption{Absolute domains of the four polydisks $\Dsd$ for $\sigma \in \Sigma$ and $\delta = (0.1, 0.2)$ (left) and
  depiction of the dual domain $(D^{+-}_\delta)'$ of $D^{+-}_\delta$ (right).}
  \label{fig:all_domains_and_dual}
\end{figure}

\begin{example}
Let $P$ be the identity matrix $\mathbb{I}$. Then $\Lambda_{P, \delta}^\delta = R^{\sigma}_\delta$, that is, it is one
of the four  quadrants translated by $(\pm \delta_1, \pm \delta_2)$.
For simplicity we write $D^\sigma_\delta = D^\sigma_{\mathbb{I},\delta}$
for the disk induced by the cone $R^{\sigma}_\delta$.
Its distinguished boundary is given by
\[\partial^* D^\sigma_\delta = \mathbb{T}^\sigma_\delta = \{z\in \Ctwo\colon |z_1| = e^{\sigma_1 \delta_1}, |z_2| = e^{\sigma_2 \delta_2}\}.\]
\begin{enumerate}[(i)]
\item Fix $\sigma = (-1, -1)$. Then, for $\delta = (0, 0)$ we obtain a unit bidisk, that is,
$\Dsd = \mathbb{D}^2$. If $\delta \neq 0$, then $\Dsd$ is a bidisk centered at $(0,0)$ with radii
$e^{\sigma_1 \delta_1}$ and $e^{\sigma_2 \delta_2}$.
\item The four domains depicted in the left panel of Figure \ref{fig:all_domains_and_dual}
(reduced representation in the plane of absolute values $(|z_1|, |z_2|)$)
 corresponds to $\delta = (0.1, 0.2)$  for all $\sigma\in \Sigma$.
\end{enumerate}
\end{example}

To proceed to the general case of Reinhardt domains, we introduce some general notation and list simple facts about toral automorphisms.

\begin{defn}
For
\[
A = \begin{pmatrix} a_{11} & a_{12} \\ a_{21} & a_{22} \end{pmatrix} \in \GLtwo(\Z)
\]
 we define $\tau_A \colon \Ttwo \to \Ttwo$ to be the toral automorphism given by
 $\tau_A(z_1, z_2) = (z_1^{a_{11}} z_2^{a_{12}}, z_1^{a_{21}} z_2^{a_{22}})$. With slight abuse of notation we
 also write $\tau_A$ for the extension of the map to $\Chat^2$.
 \end{defn}
\begin{lem}
Let $A,B \in \GLtwo(\Z)$ and $f \in L^2(\T^2)$. Then:
\begin{enumerate}[(i)]
  \item $\tau_A \circ \tau_B = \tau_{A\cdot B}$, and in particular $\tau_A^{-1} = \tau_{A^{-1}}$.
  \item The map $\tau_A$ preserves $\T^2$ and is holomorphic in a neighbourhood.
  \item $f \circ \tau_A \in L^2(\T^2)$,
and $\| f \circ \tau_A \|_{L^2(\T^2)} = \| f \|_{L^2(\T^2)}$.
\end{enumerate}
\end{lem}

We note that by Lemma \ref{lem:paq} in the appendix, every Reinhardt domain
$D^\sigma_{P,\delta} \subset \hat{\C}^2$ is the image of a bounded Reinhardt domain in
$\C^2$ under a holomorphic map of the form $\tau_A, A \in \GLtwo(\Z)$.

\begin{rem}\label{rem:hol_ext}
In analogy to the case of the one-dimensional Riemann sphere,
here we extend the notion of holomorphic functions on domains in $\Ctwo$
to those on domains in $\Ctwohat$ in the obvious way.
For a domain $\hat D \subset \Ctwohat$, a function $f \colon \hat D \to \C$ is holomorphic if
there exists a domain $D \subset \Ctwo$ and a biholomorphic mapping $\phi \colon D \to \hat D$
such that $f \circ \phi \colon D \to \C$ is holomorphic.
\end{rem}

\begin{defn}
For $\delta \in \mathbb{R}^2$, $\sigma \in \Sigma$ and $P \in \GLtwo(\R)$,
the \textit{Hardy-Hilbert
space} $H^\sigma_{P,\delta} := H^2(D_{P,\delta}^\sigma)$ consists
of all holomorphic functions $f\colon D_{P,\delta}^\sigma \to\mathbb{C}$ such that
\[\sup_{r\in \Lambda^\sigma_{P,\delta}} \frac{1}{(2\pi)^2}\int_{[0, 2\pi]^2} |f(e^{r+it})|^2 \, dt < \infty.\]
In the case $P = \mathbb{I}$ (that is, $D^\sigma_{P,\delta}$ a polydisk),
we will write $H^\sigma_\delta = H^\sigma_{P,\delta}$ for brevity.
\end{defn}

\begin{rem}
As the definition indicates, the next results (Lemma \ref{lem:boundaryvalue} and Proposition \ref{prop:characterisations}) will establish
that the space $H^\sigma_{P,\delta}$ with the inner product \eqref{eq:h2product} indeed forms a Hilbert space.
\end{rem}

We recall that
$\partial^* D^\sigma_{P,\delta} = \T^\sigma_{P,\delta} = \T^2_{\exp(v^\sigma_{P,\delta})}$,
and proceed with the following generalization
of a classical result for the polydisk $\D^2 = D^{--}_{\mathbb{I},0}$.

\begin{lem}\label{lem:boundaryvalue}
For any $f \in H^\sigma_{P,\delta}$ there exists $f^* \in L^2(\T_{P,\delta}^\sigma)$
(the ``boundary value function'' of $f$), which satisfies
\[
\lim_{r \in \Lambda^\sigma_{P,\delta}, r\to v^\sigma_{P,\delta}} \int_{[0, 2\pi]^2} | f(e^{r + it}) - f^*(e^{v^\sigma_{P,\delta} + it}) |^2 \, dt = 0.
\]
\end{lem}

\begin{proof}
Since $D^\sigma_{P,\delta} = D^{(-1,-1)}_{P',\delta}$ and for $P' = P I^\sigma \in \GLtwo(\R)$, we can assume
without loss of generality that $\sigma=(-1, -1)$.
The case of $P$ being a non-negative matrix and $\delta=0$
(i.e., $D^\sigma_{P,\delta} \subseteq \mathbb{D}^2$ a bounded Reinhardt domain) is classical,
see \cite[Section 2.5]{U} or \cite{La}, and the more general case of $\delta \in \R^2$ follows immediately.
For general $P \in \GLtwo(\R)$ and $f \in H^\sigma_{P,\delta}$, we write $P = A \tilde P$
with $A \in \GLtwo(\Z)$ and $\tilde P \in \GLtwo(\R)$ non-negative (Lemma \ref{lem:paq}).
We use $\tau_A \colon D_{\tilde P,\delta}^\sigma \to D_{P,\delta}^\sigma$
and consider the function $\tilde f = f \circ \tau_A \in H^\sigma_{\tilde P,\delta}$.
By the previous case,
there exists $\tilde f^* \in L^2(\mathbb{T}^\sigma_{\tilde P,\delta})$ such that
\[
\lim_{r\in \Lambda^\sigma_{\tilde P,\delta}, r\to v^\sigma_{\tilde P,\delta}}
\int_{[0, 2\pi]^2} | \tilde f (e^{r+it}) - \tilde f^* (e^{v^\sigma_{\tilde P,\delta} + it}) |^2 \, dt = 0.
\]
Writing $f^* = \tilde f^* \circ \tau_A^{-1}$, we obtain
\begin{align*}
&\lim_{r \in \Lambda^\sigma_{P,\delta}, r\to v^\sigma_{P,\delta}}
\int_{[0, 2\pi]^2} | f(e^{r + it}) - f^*(e^{v^\sigma_{P,\delta} + it}) |^2 \, dt \\
= &\lim_{r \in \Lambda^\sigma_{\tilde P,\delta}, r\to v^\sigma_{\tilde P,\delta}}
\int_{[0, 2\pi]^2} | f\circ \tau_A (e^{r + it}) - f^*\circ \tau_A(e^{v^\sigma_{\tilde P,\delta} + it}) |^2 \, dt = 0. \qedhere
\end{align*}

\end{proof}

We define an inner product on $H^\sigma_{P,\delta}$ by setting
\begin{equation}\label{eq:h2product}
(f,g)_{H^\sigma_{P,\delta}} := \langle f^*, \cj{g^*} \rangle_{\TsdP} := (f^*, g^*)_{L^2(\TsdP)} ,
\end{equation}
and we will omit the star notation for the boundary value function whenever there is no ambiguity.
Here the $L^2$ inner product between $f^*$ and $g^*$ is defined as
\[
(f^*, g^*)_{L^2(\TsdP)}
= \int_{[0, 2\pi]^2} f^*(e^{v^\sigma_{P,\delta}+it}) \cj{g^*(e^{v^\sigma_{P,\delta}+it})} \, \frac{dt}{(2\pi)^2}
= \int_{\TsdP} f^*(z) \cj{g^*(z)} \, dm(z),
\]
where $dm(z) = \frac{dz_1}{2\pi i z_1} \frac{dz_2}{2\pi i z_2}$ is the normalised Lebesgue measure on $\TsdP$.

\begin{notation}
We denote by $\Lambda^{\sigma,o}_P$ the polar cone of $\Lambda^\sigma_P = P(R^\sigma)$, and note that
$\Lambda^{\sigma,o}_P = (P^T)^{-1}(\cl{R^{-\sigma}})$. We write $\mathbb{Z}^\sigma_P = \Ztwo \cap \Lambda^{\sigma,o}_P$.
For $P = \mathbb{I}$, we will use the shorthand $\Z^\sigma = \Z^\sigma_\mathbb{I}$.
\end{notation}
~\newline

\hspace{0.1cm}
\begin{prop}[Alternative charaterisations]\label{prop:characterisations}
Let $\delta \in \Rtwo$, $\sigma \in \Sigma$ and $P \in \GLtwo(\R)$.
\begin{enumerate}[(i)]
  \item The function $p_n(z) =  z^n$ is in $H^\sigma_{P,\delta}$ iff $n \in \mathbb{Z}^\sigma_P$.
  \item  Let $f$ be given by $f(z) = \sum_{n\in\Ztwo} f_n z^n$. Then,
  $f\in H^\sigma_{P,\delta}$ iff $f_n = 0$ for $n\notin \mathbb{Z}^\sigma_P$ and
  $\sum_{n\in \Ztwo} |f_n|^2 e^{2 \langle n, v^\sigma_{P,\delta} \rangle} <\infty$.
  \item $H^\sigma_{P,\delta}$ is a Hilbert space with inner product given by \eqref{eq:h2product}.
 Moreover, for $f\in H^\sigma_{P,\delta}$ with $f(z) = \sum_{n\in\mathbb{Z}^\sigma_P} f_n z^n$, we have
  \[
  \|f\|^2_{H^\sigma_{P,\delta}}
  = \sup_{r\in \Lambda^\sigma_{P,\delta}} \frac{1}{(2\pi)^2}\int_{[0, 2\pi]^2} |f(e^{r+it})|^2 \, dt
  = \sum_{n\in \mathbb{Z}^\sigma_P} |f_n|^2 e^{2 \langle n, v^\sigma_{P,\delta} \rangle}.
  \]
  \item $H^\sigma_{P,\delta}$ is the closure of $\{p_n\colon n\in \mathbb{Z}^\sigma_P\}$ with respect to the norm
  $\|\cdot\|_{H^\sigma_{P,\delta}}$.
\end{enumerate}
\end{prop}

\begin{proof}
We begin by recalling that $\Lambda^\sigma_{P,\delta} = \Lambda^\sigma_P + v^\sigma_{P,\delta}$, and observe that
\[
\sup_{r\in \Lambda^\sigma_{P,\delta}} \frac{1}{(2\pi)^2}\int_{[0, 2\pi]^2} |p_n(e^{r+it})|^2 \, dt
= \sup_{r\in \Lambda^\sigma_{P,\delta}} e^{2\langle  n, r \rangle}
= e^{2\langle n, v^\sigma_{P,\delta} \rangle} \sup_{r \in \Lambda_P^\sigma} e^{2\langle  n, r \rangle}
\]
for any $n \in \Ztwo$. Since the supremum is finite if and only if $\langle n, r \rangle \leq 0$ for all $r \in \Lambda^\sigma_P$, that is, if and only if $n \in \mathbb{Z}^\sigma_P = \mathbb{Z}^2 \cap \Lambda_P^{\sigma,o}$, statement $(i)$ follows.
Next, we formally calculate for $f(z) = \sum_{n\in\mathbb{Z}^2} f_n z^n$ that
\begin{align*}
\frac{1}{(2\pi)^2}\int_{[0, 2\pi]^2} |f(e^{r+it})|^2 \,dt
&= \frac{1}{(2\pi)^2}\int_{[0, 2\pi]^2}  \sum_{n\in\mathbb{Z}^2} f_n e^{n(r+it)}\cj{\sum_{m\in\mathbb{Z}^2} f_m e^{m(r+it)}}  \,dt \\
&= \sum_{n\in\mathbb{Z}^2} |f_n|^2 \frac{1}{(2\pi)^2}\int_{[0, 2\pi]^2} |p_n(e^{r+it})|^2  \,dt \\
&= \sum_{n\in\mathbb{Z}^2} |f_n|^2 e^{2\langle n,r \rangle} .
\end{align*}
If $f_n \neq 0$ for some $n \notin \mathbb{Z}^\sigma_P$, by the above argument, the supremum over $r \in \Lambda^\sigma_{P,\delta}$ is infinite, and hence
$f \notin H^\sigma_{P,\delta}$.
On the other hand, for $n \in \mathbb{Z}^\sigma_P$ and $r \in \Lambda^\sigma_{P,\delta}$, we have
$\langle n, r \rangle = \langle P^T n, P^{-1} r \rangle = \langle I^\sigma P^T n , I^\sigma P^{-1}r \rangle$ with $I^\sigma P^T n \in \mathbb{R}^2_{\leq 0}$ and $I^\sigma P^{-1}r \in \mathbb{R}^2_{>0} + \delta$, so that
\[
\langle n,r \rangle
 \leq \langle I^\sigma P^T n, \delta \rangle
= \langle P^T n, I^\sigma \delta \rangle = \langle n, v^\sigma_{P,\delta} \rangle.
\]
Therefore, if $f_n = 0$ for all $n \notin \mathbb{Z}^\sigma_P$ and
$\sum_{n\in \Ztwo} |f_n|^2 e^{2 \langle n, v^\sigma_{P,\delta} \rangle} <\infty$, then
\[
\sup_{r\in \Lambda^\sigma_{P,\delta}} \frac{1}{(2\pi)^2}\int_{[0, 2\pi]^2} |f(e^{r+it})|^2 \,dt
=  \sup_{r\in \Lambda^\sigma_{P,\delta}} \sum_{n\in\Ztwo} |f_n|^2 e^{2\langle n,r \rangle}
= \sum_{n\in \Ztwo} |f_n|^2 e^{2 \langle n, v^\sigma_{P,\delta} \rangle}
\]
implies $f \in H^\sigma_{P,\delta}$, proving $(ii)$.
For $(iii)$, it remains to show the last equality, which follows immediately from the last calculation and \eqref{eq:h2product}, as
for any $f \in H^\sigma_{P,\delta}$ we have that
\[
\sup_{r\in \Lambda^\sigma_{P,\delta}} \frac{1}{(2\pi)^2} \int_{[0, 2\pi]^2} |f(e^{r+it})|^2 \, dt
= \sum_{n\in \mathbb{Z}^\sigma_P} |f_n|^2 e^{2 \langle n, v^\sigma_{P,\delta} \rangle}.
\]
Finally for $(iv)$, we assume without loss of generality that $\sigma = (-1, -1)$
(since $H^\sigma_{P,\delta} = H^{(-1,-1)}_{PI^\sigma,\delta}$).
The case of $P$ a non-negative matrix is proved in \cite[Proposition 3.6]{MX}.
For general $P \in \GLtwo(\R)$, we write $P = A \tilde P$ with $A \in \GLtwo(\Z)$ and $\tilde P \in \GLtwo(\R)$ non-negative,
and use the biholomorphic mapping $\tau_A \colon D^\sigma_{\tilde P,\delta} \to D^\sigma_{P,\delta}$,
noting that $\| f \circ \tau_A\|_{H^\sigma_{\tilde P,\delta}} = \| f \|_{H^{\sigma}_{P,\delta}}$
for any $f \in H^\sigma_{P,\delta}$.
Since $A$ maps $\Z^{\sigma}_{\tilde P}$ bijectively onto $\Z^{\sigma}_{P}$,
we have $\{p_n\colon n\in \Z^\sigma_P\} = \{p_n \circ \tau_A\colon n\in \Z^{\sigma}_{\tilde P}\}$, and the general case follows.
\end{proof}

We next show shows that two Hardy-Hilbert $H^\sigma_{P,\delta}$ and $H^\sigma_{P',\delta}$ are isomorphic,
whenever $P^{-1}P' \in \GLtwo(\Z)$.
For any $\delta \in \Rtwo$, $\sigma \in \Sigma$, $P \in \GLtwo(\R)$ and $A \in \GLtwo(\Z)$,
we observe that $A (\Lambda_{P,\delta}^\sigma) = (A P)(R^\sigma_\delta) = \Lambda_{AP,\delta}^\sigma$,
and $A P \in \GLtwo(\R)$, from which it follows that
\[
\tau_A(D_{P,\delta}^\sigma) = \tau_A(e^{\Lambda_{P,\delta}^\sigma} \Ttwo) = e^{\Lambda_{AP,\delta}^\sigma} \Ttwo =  D_{AP,\delta}^\sigma.
\]

\begin{prop}\label{prop:isoiso}
Let $\delta \in \Rtwo$, $\sigma \in \Sigma$, $P \in \GLtwo(\R)$ and $A \in \GLtwo(\Z)$. Then the operator $C_{\tau_A}$
given by $C_{\tau_A} f = f \circ \tau_A$ is an isometric isomorphism from $H^\sigma_{AP,\delta}$ to $H^\sigma_{P,\delta}$.
\end{prop}

\begin{proof}
Let $f = \sum_{n\in \Ztwo} f_n p_n \in H^\sigma_{AP,\delta}$, that is, by Proposition \ref{prop:characterisations},
$f_n = 0$ for all $n \notin \mathbb{Z}^\sigma_{AP}$, and $\sum_{n\in \Ztwo} |f_n|^2 e^{2 \langle n, v^\sigma_{AP,\delta} \rangle} < \infty$,
and write $g = C_{\tau_A} f = f \circ \tau_A$. For $n \in \Ztwo$ we have
\[
p_n \circ \tau_A(z) = (z_1^{a_{11}} z_2^{a_{12}})^{n_1} (z_1^{a_{21}} z_2^{a_{22}})^{n_2}
= z_1^{a_{11}n_1 + a_{21} n_2} z_2^{a_{21}n_1 + a_{22} n_2} = p_{A^T n}(z).
\]
Therefore, it holds that
\[
g = \sum_{n \in \Ztwo} f_n p_n \circ \tau_A = \sum_{n \in \Ztwo} f_n p_{A^T n} = \sum_{n \in \Ztwo} f_{(A^T)^{-1} n} p_{n},
\]
and hence $g = \sum_{n\in \Ztwo} g_n p_n$ with $g_n = f_{(A^T)^{-1} n}$.
Further, since
\[
\Lambda^{\sigma,o}_{AP}
= ((A P)^T)^{-1} (\cl{R^{-\sigma}})
= (A^T)^{-1} (P^T)^{-1} (\cl{R^{-\sigma}})
= (A^T)^{-1} (\Lambda^{\sigma,o}_{P}),
\]
it follows that $(A^T)^{-1} n \in \Lambda^{\sigma,o}_{AP}$ iff $n \in \Lambda^{\sigma,o}_P$. Therefore, $g_n = f_{(A^T)^{-1} n} = 0$ whenever $n \notin \mathbb{Z}^\sigma_{P}$. Moreover,
\begin{align*}
\|g\|^2_{H^\sigma_{P,\delta}}
&= \sum_{n\in \Ztwo} |g_n|^2 e^{2 \langle n, v^\sigma_{P,\delta} \rangle} \\
&= \sum_{n\in \Ztwo} |f_n|^2 e^{2 \langle A^T n, Pv^\sigma_\delta \rangle}
= \sum_{n\in \Ztwo} |f_n|^2 e^{2 \langle n, APv^\sigma_\delta \rangle}
= \sum_{n\in \Ztwo} |f_n|^2 e^{2 \langle n, v^\sigma_{AP,\delta} \rangle}
= \|f\|^2_{H^\sigma_{AP,\delta}},
\end{align*}
which by Proposition \ref{prop:characterisations}$(ii)$ implies that $g = C_{\tau_A} f \in H^\sigma_{P,\delta}$, and
$\|C_{\tau_A} f\|_{H^\sigma_{P,\delta}} =  \|f\|_{H^\sigma_{AP,\delta}}$ for all $f \in H^\sigma_{AP,\delta}$,
finishing the proof.
\end{proof}

The next two propositions are important ingredients for proving the main result of this section.

\begin{prop}\label{prop:sup_comp}
Let $K$ be a compact subset of $D_{P,\delta}^\sigma \subset \Chat^2$ and $f\in H^\sigma_{P,\delta} = H^2(D_{P,\delta}^\sigma)$.
Then, there is a $C_K > 0$ such that
\[\sup_{z\in K} |f(z)| \leq C_K  \|f\|_{H^\sigma_{P,\delta}}.\]
\end{prop}

\begin{proof}
For $P = \mathbb{I}$ and $\sigma = (-1,-1)$ (that is, $D_{P,\delta}^\sigma$
a polydisk with radii $(e^{-\delta_1}, e^{-\delta_2})$) the result follows directly from \cite[Lemma 2.9]{BJ}.
In particular, with $D' \cc D_{P,\delta}^\sigma$ a domain containing $K$, and $U(D')$ the space of functions holomorphic
on $D'$ and continuous on $\cl{D'}$ endowed with the supremum norm,
 the natural embedding $H^{--}_{\mathbb{I},\delta} \hookrightarrow U(D')$ is bounded, and its operator norm
 provides the constant $C_K$.

For the general case, let $\psi(z) = z^{-\sigma} = (z_1^{-\sigma_1}, z_2^{-\sigma_2})$.
By Proposition \ref{prop:isoiso}, the operator $C_{\psi \circ \tau_P} \colon H^{--}_{\mathbb{I},\delta} \to H^\sigma_{P,\delta}$
is an isometric isomorphism, and the result follows by observing that
\[\sup_{z\in K\subset D_{P,\delta}^\sigma} |f(z)| = \sup_{z \in (\psi \circ \tau_P)^{-1}(K)} | f \circ \psi \circ \tau_P (z) | \]
with $(\psi \circ \tau_P)^{-1}(K) \cc D^{--}_\delta$.
\end{proof}

\begin{defn}[Dual domains and dual space]~
\begin{enumerate}[(i)]
\item We denote by $(H^\sigma_{P,\delta})'$ the dual space of
$H^\sigma_{P,\delta} = H^2(D_{P,\delta}^\sigma)$, that is the space of all
continuous linear functionals on $H^\sigma_{P,\delta}$.
\item We denote by $(D_{P,\delta}^\sigma)'$ the dual domain of
$D_{P,\delta}^\sigma$, which is given by $(D_{P,\delta}^\sigma)' = D_{P,-\delta}^{-\sigma}$
(see the right panel of Figure \ref{fig:all_domains_and_dual} for an illustration of the case $P = \mathbb{I}$).
\end{enumerate}
\end{defn}
Clearly, $((D_{P,\delta}^\sigma)')' = D_{P,\delta}^\sigma$ and
$\T_{P, \delta}^{\sigma} = \T^{-\sigma}_{P, -\delta}$.

\begin{prop}\label{prop:iso}
$(H^\sigma_{P,\delta})' = (H^2(D_{P,\delta}^\sigma))'$ is isometrically isomorphic to $H^{-\sigma}_{P,-\delta} = H^2((D_{P,\delta}^\sigma)')$,
via the isomorphism $J\colon H^{-\sigma}_{P,-\delta} \to (H^\sigma_{P,\delta})'$ given by
\[g \mapsto \langle \cdot, g \rangle_{\TsdP}.\]
\end{prop}

\begin{proof}
We begin by showing that $J$ is injective. Assume $J(g) = 0$ for some $g \in H^{-\sigma}_{P,-\delta}$,
$g(z) = \sum_{n\in \mathbb{Z}_P^{-\sigma}} g_n z^n$, that is,
$\langle f, g \rangle_{\TsdP} = 0$
for all $f \in H^\sigma_{P,\delta}$. In particular,
for every $m \in \mathbb{Z}_P^\sigma$ we have
\[
0 = \langle p_m, g \rangle_{\TsdP}
= \sum_{n\in \mathbb{Z}_P^{-\sigma}} g_n \langle p_m, p_n \rangle_{\TsdP}
= \sum_{n\in \mathbb{Z}_P^{-\sigma}} g_n e^{\langle m+n, v^\sigma_{P,\delta} \rangle} \delta_{m,-n}
= g_{-m}.
\]
It follows that $g_n = 0$ for all $n \in \mathbb{Z}_P^{-\sigma}$ and hence $g = 0$, proving injectivity of $J$.

To show that $J$ is surjective, let $\ell \in (H^\sigma_{P,\delta})'$. By the Riesz representation theorem,
there exists $g \in H^\sigma_{P,\delta}$, so that $\ell  = \ell_g = (\cdot, g)_{H^\sigma_{P,\delta}}$, and we can write
$g(z) =\sum_{n\in\mathbb{Z}_P^\sigma} g_n z^n$. We define
$h(z) = \sum_{n\in\mathbb{Z}_P^{-\sigma}} \cj{g_{-n}} e^{-2 \langle n, r \rangle} z^n$ with $r = v^\sigma_{P,\delta}$,
and observe that $h \in H^2((D_{P,\delta}^\sigma)') =  H^{-\sigma}_{P,-\delta}$. For $z = e^r e^{it}\in \TsdP$, $t \in [0, 2\pi)^2$, we have
$e^{-2r} z = e^{-r} e^{it} = 1 / \cj{z}$, and hence
\[
\cj{h(z)}
= \sum_{n\in\mathbb{Z}_P^{-\sigma}} g_{-n} \cj{e^{-2\langle n, r \rangle} z^n}
= \sum_{n\in\mathbb{Z}_P^{-\sigma}} g_{-n} z^{-n}
= \sum_{n\in\mathbb{Z}_P^{\sigma}} g_n z^n = g(z).
\]
By \eqref{eq:h2product} this implies
\[
(J(h))(f) = \langle f, h \rangle_{\TsdP} = (f, g)_{H^\sigma_{P,\delta}} = \ell(f),
\]
for any $f \in H^\sigma_{P,\delta}$. Hence $J(h) = \ell$, proving surjectivity of $J$.

Note that the Riesz representation theorem also yields
$\|\ell_g\|_{(H^\sigma_{P,\delta})'} = \|g\|_{H^\sigma_{P,\delta}}$.
On the other hand, for $J(h) = \ell_g$ as above, we have
\[
\| h \|^2_{H^{-\sigma}_{P,-\delta}}
= \sum_{n\in\mathbb{Z}_P^{-\sigma}} |h_n|^2 |e^{2\langle n, r\rangle}|
= \sum_{n\in\mathbb{Z}_P^{-\sigma}} |g_{-n} e^{-2\langle n, r  \rangle}|^2 |e^{2\langle n, r\rangle}|
= \sum_{n\in\mathbb{Z}_P^\sigma} |g_n |^2 |e^{2\langle n, r\rangle}| = \| g \|^2_{H^\sigma_{P,\delta}},
\]
and so we obtain $\|J(h)\|_{(H^\sigma_{P,\delta})'} = \| h \|_{H^{-\sigma}_{P,-\delta}}$, proving that $J$ is an isometry.
\end{proof}

\begin{rem}
Using the above proposition we have the following reformulation:
if $f\in H^\sigma_{P,\delta}$ then
\begin{equation}\label{eq:Dsd_norm_dual}
\|f\|_{H^\sigma_{P,\delta}} = \sup \left\{ \left| \langle f, g \rangle_{\TsdP} \right|
\colon g\in H^2((D_{P,\delta}^\sigma)') \cap \mathcal{P}\, , \|g\|_{H^2( (D_{P,\delta}^\sigma)')} \leq 1 \right\},
\end{equation}
where as before, $\mathcal{P}$ denotes the space of Laurent polynomials.
\end{rem}

\subsection{Hardy-Hilbert spaces for toral diffeomorphisms}

For convenience, here we introduce notation to succinctly express monomial bases for various Hilbert spaces which we will require below.

\begin{notation}\label{not:zhat}
Let $R^{\sigma, o}$ be the (closed) polar cone of $R^{\sigma}$ for any $\sigma\in \Sigma$.
Define
\begin{equation*}
\hat{R}^{\sigma, o} =
\begin{cases}
R^{\sigma,o} \text{ if } \sigma=(-1,-1),\\
R^{\sigma,o}\setminus{\{(0,0)\}} \text{ if } \sigma=(+1,+1),\\
\inter{R^{\sigma,o}} \text{ if } \sigma\in \Sigma^{-1},
\end{cases}
\end{equation*}
and let $\hat{\mathbb{Z}}^\sigma_P = \Ztwo \cap P(\hat{R}^{\sigma, o})$, noting that the $\hat{R}^{\sigma,o}$, $\sigma \in \Sigma$, form a partition of $\R^2$.
In analogy to the characterization of $H^\sigma_{P,\delta}$ in Proposition \ref{prop:characterisations}$(ii)$,
for $\sigma \in \Sigma$, $\delta \in \R^2$ and $P \in \GLtwo(\R)$ we define
\[
\hat{H}^\sigma_{P,\delta} := \hat{H}^2(D^\sigma_{P,\delta}) := \{f \in H^\sigma_{P,\delta}: (f, p_n)_{H^\sigma_{P,\delta}} = 0 \text{ for } n \notin \hat{\Z}^\sigma_P\}.
\]
Writing $e_n = \frac{p_n}{ \nu(n)}$ with $\nu(n) = \|p_n\|_{H^\sigma_{P,\delta}} = e^{\langle n, v^\sigma_{P,\delta}\rangle}$,
we note that $\{e_n : n \in \hat{\Z}^\sigma_P \}$ forms an orthonormal basis for $\hat{H}^\sigma_{P,\delta}$,
which is a Hilbert space with the same norm $\| \cdot \|_{H^\sigma_{P,\delta}}$.
Furthermore, for conveniently handling dual spaces, we set $\check{\mathbb{Z}}^\sigma_P = -\hat{\mathbb{Z}}^\sigma_P$, and
\[\check{H}^\sigma_{P,\delta} = \{f\in H^2(D^\sigma_{P,\delta})\colon \langle f, p_n \rangle = 0, n\notin \check{\Z}^\sigma_P\}.\]
\end{notation}

\begin{rem}
With the above notation, the isomorphism defined in Proposition \ref{prop:iso} also forms an isometric isomorphism
between $\check{H}^{-\sigma}_{P,-\delta}$ and $(\hat{H}^\sigma_{P,\delta})'$.
\end{rem}

\begin{notation}
For $\ell \in \{1, -1\}$, we write
\[
\mathcal{D}^\ell_{P,\delta} := \bigcup_{\sigma \in \Sigma^{\ell}} D^\sigma_{P,\delta},
\qquad \text{ and } \qquad
(\mathcal{D}^\ell_{P,\delta})' := \bigcup_{\sigma \in \Sigma^{\ell}} (D^\sigma_{P,\delta})' =
\bigcup_{\sigma \in \Sigma^{\ell}} D^{-\sigma}_{P,-\delta},
\]
and note that the distinguished boundary
of $\mathcal{D}^\ell_{P,\delta}$ is $\partial^*\mathcal{D}^\ell_{P,\delta} = \bigcup_{\sigma \in \Sigma^\ell} \T^{\sigma}_{P,\delta}$.
\end{notation}
See the left panel of Figure \ref{fig:all_domains_and_dual} for an illustration in the case $P = \mathbb{I}$:
the green and pink rectangles represent $\mathcal{D}^1_{P,\delta}$,
whereas the orange and blue rectangles represent $\mathcal{D}^{-1}_{P,\delta}$,
with correspondingly coloured stars representing the respective distinguished boundaries.

\begin{defn}\label{def:combined_spaces}
For $\ell \in \{1, -1\}$, define
$\mathcal{H}^\ell_{P,\delta} = \bigoplus_{\sigma \in \Sigma^\ell} \hat{H}^\sigma_{P,\delta}$,
which is a Hilbert space with the
inner product of $f = (f^\sigma)_{\sigma \in \Sigma^\ell}, g = (g^\sigma)_{\sigma \in \Sigma^\ell} \in \mathcal{H}^\ell_{P,\delta}$ given by
\[
(f,g)_{\mathcal{H}^\ell_{P,\delta}} = \sum_{\sigma \in \Sigma^\ell} (f^\sigma, g^\sigma)_{H^\sigma_{P,\delta}}.
\]
\end{defn}

As before, for the case $P = \mathbb{I}$ where the domains $D^\sigma_{P,\delta}$ are polydisks,
we will use the shorthands
$\hat{\mathbb{Z}}^\sigma = \hat{\mathbb{Z}}^\sigma_{\mathbb{I}}$, $D^\sigma_\delta = D^\sigma_{\mathbb{I}, \delta}$,
$\mathcal{D}^\ell_\delta = \mathcal{D}^\ell_{P,\delta}$,
$\hat{H}^\sigma_\delta = \hat{H}^\sigma_{\mathbb{I},\delta}$, and $\mathcal{H}^\ell_\delta = \mathcal{H}^\ell_{\mathbb{I},\delta}$.

\begin{rem}\label{rem:tuplespace}
Nominally, an $f \in \mathcal{H}^\ell_{P,\delta}$ is a tuple
$f = (f^{\sigma})_{\sigma \in \Sigma^\ell}$ of two holomorphic functions with distinct domains,
$f^\sigma \in \hat{H}^\sigma_{P,\delta} = \hat{H}^2(D^\sigma_{P,\delta})$.
It will be useful to alternatively consider $\mathcal{H}^\ell_{P,\delta}$ as (isomorphic to) a function space,
which requires us to distinguish two cases:
\begin{enumerate}[(i)]
  \item If $\delta \in \R^2_{<0}$, then $\mathcal{A}_{P,\delta} = D^\sigma_{P,\delta} \cap D^{-\sigma}_{P,\delta} \neq \emptyset$,
  and for any $f = (f^\sigma, f^{-\sigma}) \in \mathcal{H}^\ell_{P,\delta}$
  we can define a holomorphic function $\tilde f$ on $\mathcal{A}_{P,\delta}$
  by $\tilde f(z) = f^{\sigma}(z) + f^{-\sigma}(z)$, yielding an isometric isomorphism between
  $\mathcal{H}^\ell_{P,\delta} = \hat{H}^2(D^\sigma_{P,\delta}) \oplus \hat{H}^2(D^{-\sigma}_{P,\delta})$
  and (a subspace of) $H^2(\mathcal{A}_{P,\delta})$.

  \item If $\delta_k \geq 0$ for some $k \in  \{1, 2\}$, then $D^\sigma_{P,\delta} \cap  D^{-\sigma}_{P,\delta} = \emptyset$,
  and for any
  $f = (f^\sigma, f^{-\sigma}) \in \mathcal{H}^\ell_{P,\delta}$ we can define a holomorphic function $\tilde f$ on
  $\mathcal{D}^\ell_{P,\delta} = D^\sigma_{P,\delta} \cup D^{-\sigma}_{P,\delta}$ by
  $\tilde f(z) = f^\sigma(z)$ for $z \in D^\sigma_{P,\delta}$, $\sigma \in \Sigma^\ell$,
  yielding an isomorphism between $\mathcal{H}^\ell_{P,\delta}$, and (a subspace of) the space of
  holomorphic functions on $\mathcal{D}^\ell_{P,\delta}$ which extend to
  square-integrable functions on
  $\partial^* \mathcal{D}^\ell_{P,\delta} = \partial^* D^\sigma_{P,\delta} \cup \partial^* D^{-\sigma}_{P,\delta}$.
\end{enumerate}
\end{rem}

\begin{figure}[ht!]
  \centering
  \includegraphics[width=0.49\textwidth]{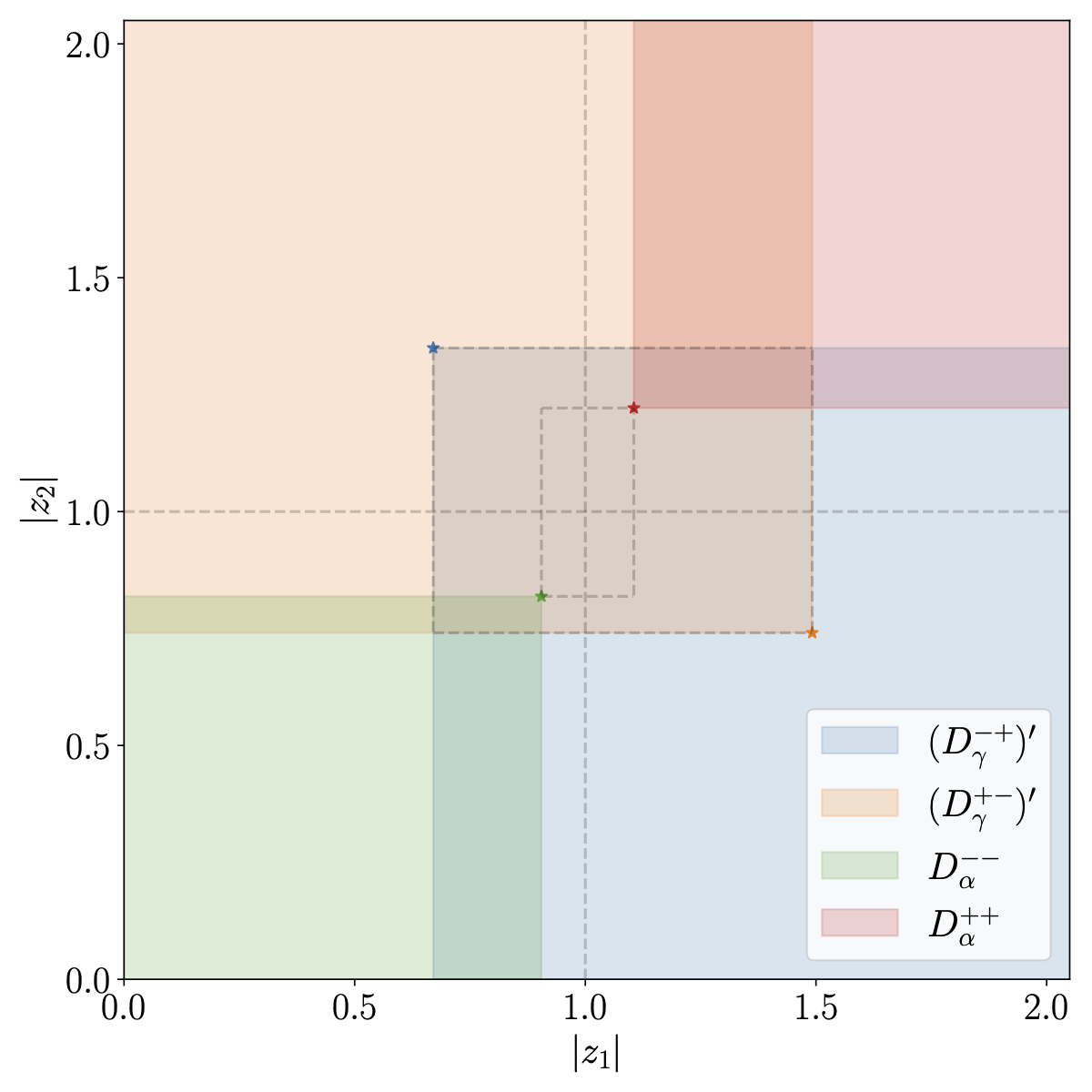}
  \includegraphics[width=0.49\textwidth]{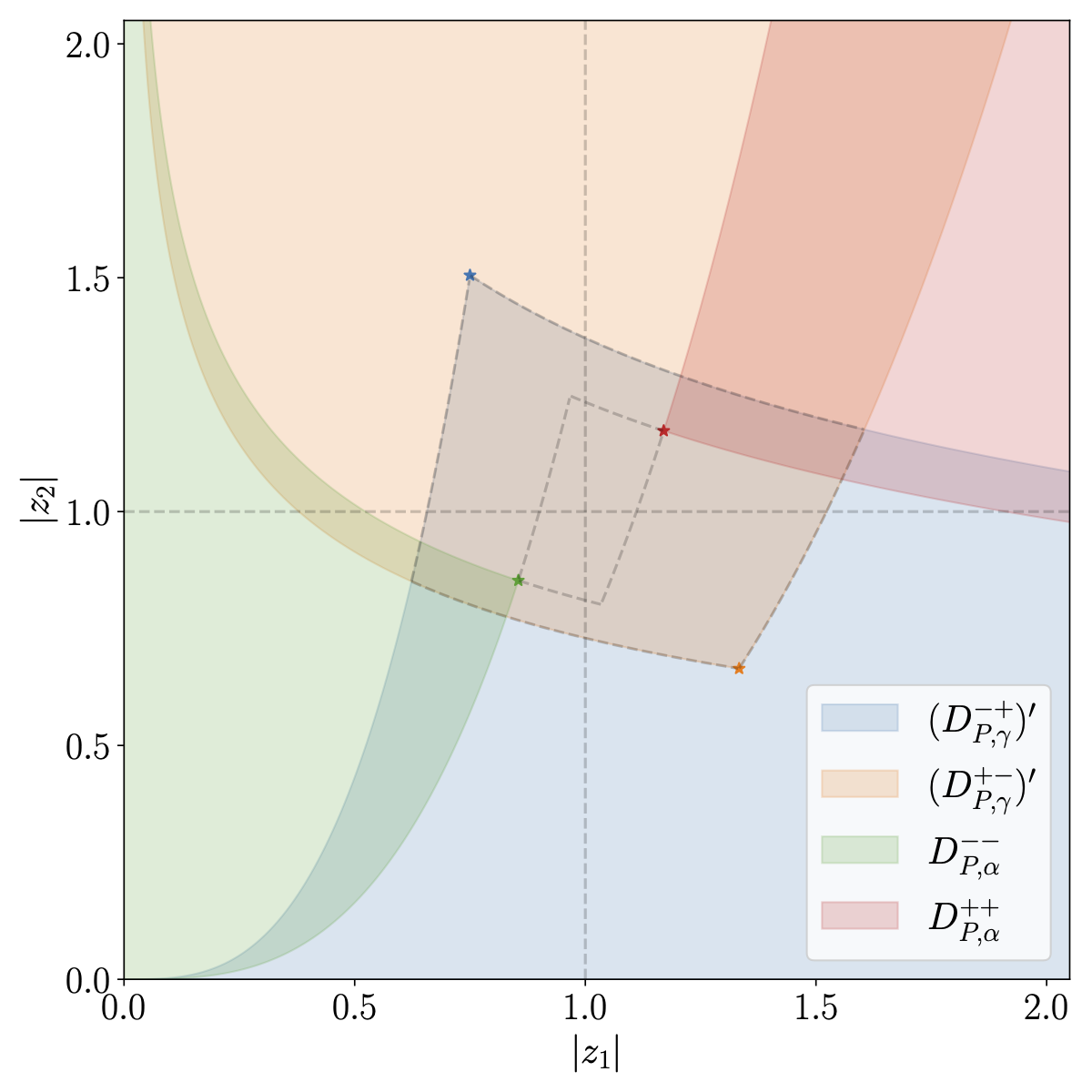}
  \caption{Absolute domains for the Reinhardt domains underlying $H_\nu = H_{P, \alpha, -\gamma}$ with $\alpha = (0.1, 0.2)$, $\gamma=(0.4, 0.3)$, for $P=\mathbb{I}$ (left) and the rotation matrix with rotation angle $\pi/10$ (right).}
  \label{fig:Hnu_domains}
\end{figure}

We can now express the weighted function spaces with respect to quadrant-wise exponential weights considered
in Section \ref{sec:pedestrian} as direct sums of Hardy-Hilbert spaces on certain disks, see Figure \ref{fig:Hnu_domains} (left).
For this, let $\alpha, \gamma \in \mathbb{R}^2$, $\nu = \nu_{\alpha, -\gamma}$ the associated quadrant-wise exponential weight
(see Definition \ref{defn:weight}) and $H_\nu = H_{\alpha, -\gamma}$ the resulting Hilbert space, given as the
completion of $\mathcal{P}$ with respect to the norm $\|\cdot\|_\nu$ (see Definition \ref{defn:Hnu}).
Then
\begin{align*}
H_{\alpha, -\gamma}
\cong \mathcal{H}^1_{\alpha} \oplus \mathcal{H}^{-1}_{-\gamma}
&= \hat{H}^{--}_{\alpha} \oplus \hat{H}^{++}_{\alpha} \oplus \hat{H}^{-+}_{-\gamma} \oplus \hat{H}^{+-}_{-\gamma}\\
&= \hat{H}^2(D^{--}_\alpha) \oplus \hat{H}^2(D^{++}_{\alpha}) \oplus \hat{H}^2((D^{+-}_{\gamma})') \oplus \hat{H}^2((D^{-+}_{\gamma})')
\end{align*}
with the isometric isomorphism given by
\[
f = \sum_{n\in\mathbb{Z}^2} f_n p_n \mapsto ( f^\sigma )_{\sigma \in \Sigma}, \qquad f^\sigma = \sum_{n\in\hat{\Z}^\sigma} f_n p_n.
\]
With the obvious generalization of the weight function $\nu = \nu_{P,\alpha,-\gamma}$
one can also define this for $P \neq \mathbb{I}$,
obtaining the more general space $H_\nu = H_{P,\alpha,-\gamma} \cong \mathcal{H}^1_{P,\alpha} \oplus \mathcal{H}^{-1}_{P,-\gamma}$, see Figure \ref{fig:Hnu_domains} (right).

\begin{rem}
The above isomorphic representation of $H_\nu$ reveals an intuitive structure of this Hilbert space.
For $\alpha, \gamma \in \R^2_{>0}$, by Remark \ref{rem:tuplespace}, the first part $\mathcal{H}^1_{P,\alpha}$ can be
viewed as (isomorphic to) ``$H^2(\mathcal{D}^1_{P,\alpha})$'', the space of functions holomorphic on $\mathcal{D}^1_{P,\alpha}$
extending to square-integrable functions on $\partial^* \mathcal{D}^1_{P,\alpha}$. The second part $\mathcal{H}^{-1}_{P,-\gamma}$, on the other hand,
can be seen to be isomorphic to the dual $(\mathcal{H}^{-1}_{P,\gamma})'$ (see Lemma \ref{lem:tuple_isom}), with $\mathcal{H}^{-1}_{P,\gamma}$
 isomorphic to a space ``$H^2(\mathcal{D}^{-1}_{P,\gamma})$'':
the space of holomorphic functions on $\mathcal{D}^{-1}_{P,\gamma}$ extending to
square-integrable functions on $\partial^* \mathcal{D}^{-1}_{P,\gamma}$.
\end{rem}

\begin{notation}
For $f \in \mathcal{H}^\ell_{P,\delta}$ for some $\ell \in \{1, -1\}$, or $f \in \mathcal{H}^1_{P,\alpha} \oplus \mathcal{H}^{-1}_{P,-\gamma}$,
we denote the canonical projection onto one of the components as $\hat{\Pi}^\sigma$ (omitting the dependence on $P$), given by
\[ \hat{\Pi}^\sigma p_n =
\begin{cases}
p_n & \text{ if } n\in \hat{\Z}^\sigma_P, \\
0 &\text{ otherwise,}
\end{cases}\]
so that if $f = \sum_{n\in\Z^2} f_n p_n$ (in the sense of Remark \ref{rem:tuplespace}),
then $\hat{\Pi}^\sigma f = \sum_{n\in \hat{\Z}^\sigma_P} f_n p_n$, where the operator's domain and range
can be inferred from context. Similarly, we write $\check{\Pi}^\sigma f = \sum_{n\in \check{\Z}^\sigma_P} f_n p_n$.
\end{notation}

\begin{rem}\label{rem:tuple2func}
With the isometric isomorphism $\Phi \colon \mathcal{H}^1_{P, \alpha} \oplus \mathcal{H}^{-1}_{P, -\gamma} \to H_{P, \alpha,-\gamma}$,
we can now use any well-defined operator $L$ on $H_{P, \alpha,-\gamma}$ to define a
conjugated operator on the respective space of function tuples. With slight abuse of notation,
in such cases we will continue denoting the respective operator on $\mathcal{H}^1_{P,\alpha} \oplus \mathcal{H}^{-1}_{P,-\gamma}$
by the same symbol $L$.
Furthermore, every toral automorphism $\tau_Q$, $Q \in \GLtwo(\Z)$,
yields an isometric isomorphism
$C_{\tau_Q}\colon H^\sigma_{QP,\delta} \to H^\sigma_{P,\delta}$
for any $P \in \GLtwo(\R)$, $\delta \in \Rtwo$, $\sigma \in \Sigma$ (Proposition \ref{prop:isoiso}), which can be used to define an operator $\mathcal{H}^1_{QP,\alpha} \oplus \mathcal{H}^{-1}_{QP,-\gamma} \to \mathcal{H}^1_{P,\alpha} \oplus \mathcal{H}^{-1}_{P,-\gamma}$ given by $(f^\sigma)_{\sigma \in \Sigma} \mapsto (\hat{\Pi}^\sigma(f^\sigma \circ \tau_Q))_{\sigma \in \Sigma}$, conjugated to the composition operator $f \mapsto f \circ \tau_Q$
viewed as an operator from $H_{QP,\alpha,-\gamma}$ to $H_{P,\alpha,-\gamma}$.
We will refer to all three of these operators as $C_{\tau_Q}$.
\end{rem}

The above decomposition of the space $H_\nu$ will allow us to prove the main result of this section
(Theorem \ref{thm:boundedness_large2small}) for holomorphic maps on the torus
satisfying the strongly expanding mapping property from Definition \ref{defn:sem}, by adapting
a method previously used in the one-dimensional setting
of analytic expanding circle maps $\tau\colon \T \to \T$, see \cite{BJS}.
To summarize, writing $U_r = \{z\in \C \colon |z|<r\}$
and $\mathcal{U}_r = U_r \cup (\Chat\setminus \cl{U_{1/r}})$ with $r \in (0, 1)$,
analyticity and expansivity of $\tau$ imply that there exists $r \in (0,1)$
such that $\tau$ extends holomorphically to a suitable neighbourhood of $\T$, and  $\tau(\partial\mathcal{U}_r)\cc \mathcal{U}_r$.
This in turn guarantees compactness of $C_\tau$ on $H^2(U_r) \oplus H_0^2(\Chat\setminus \cl{U_{1/r}})$.
In the same spirit, for $T\colon \Ttwo \to \Ttwo$ an analytic Anosov map, if
$T^{\ell}$ is $(\ell, \delta_\ell, \Delta_\ell, P)$-strongly expanding for
$\ell \in \{1, -1\}$ and suitable $\delta_{\ell}, \Delta_{\ell} \in \R^2_{>0}$, then
\[T \left(\partial^* \mathcal{D}^{\ell}_{P,\delta_{\ell}}\right)
\subset \mathcal{D}^{\ell}_{P,\Delta_\ell} \cc \mathcal{D}^\ell_{P,\delta_{\ell}},\]
which will be used to prove compactness of $C_T$ on $\mathcal{H}^{1}_{\delta_1} \oplus (\mathcal{H}^{-1}_{\delta_{-1}})'$
with similar techniques as in~\cite{BJS}.

\begin{lem}\label{lem:contour_deform}
Let $T$ be a map with $(\ell, \delta, \Delta, P)$-strongly expanding mapping property, then
for every $\sigma, \tilde{\sigma} \in \Sigma^\ell$ there is $\hat{\sigma} \in \Sigma^\ell$ such that
\[\mathbb{T}^{\hat{\sigma}}_{P,\delta} \subset \cl{(D^{\sigma}_{P,\delta})'} \text{ and }\,
T(\mathbb{T}^{\hat{\sigma}}_{P,\delta}) \subset D^{\tilde{\sigma}}_{P,\Delta} \cc D^{\tilde{\sigma}}_{P,\delta}.\]
\end{lem}

\begin{proof}
The first property is obvious as $\T^{\hat{\sigma}}_{P,\delta} \subset \cl{(D^{\sigma}_{P,\delta})'}$
for all $\hat{\sigma} \in \Sigma^\ell$. Next, if $T$ satisfies the (EP) property, pick $\hat{\sigma} = \tilde{\sigma}$,
whereas if $T$ satisfies (ER) pick $\hat{\sigma} = - \tilde{\sigma}$.
\end{proof}

\begin{theorem}\label{thm:boundedness_large2small}
Let $T$ be an analytic diffeomorphism of $\Ttwo$. Further assume that
there are $\alpha,\gamma, A, \Gamma, \eta \in \mathbb{R}^2_{>0}$ with $\alpha, \gamma < \eta$, and $P \in \GLtwo(\R)$, such
that $T$ and $T^{-1}$ can be analytically extended to $\mathcal{A}_{P,\eta}$ and
the following mapping properties hold:
\begin{enumerate}[(i)]
  \item $T$ is  $(1, \alpha, A, P)$-strongly expanding;
  \item $T^{-1}$ is $(-1, \Gamma, \gamma, P)$-strongly expanding;
  \item For any $\tilde{\sigma} \in \Sigma^1, \sigma \in \Sigma^{-1}$, there exist $\mathbb{T}^2_q \subset
  D^{\sigma}_{P,\gamma} \cap \mathcal{A}_{P,\eta}$ with $T(\mathbb{T}^2_{q}) \cc D^{\tilde{\sigma}}_{P,A}$.
\end{enumerate}
Then, the composition operator $C_T$  given by \[f \mapsto  f\circ T\] maps
$H_{P, A, -\Gamma}$ continuously to $H_{P, \alpha, -\gamma}$.
\end{theorem}

\begin{figure}[ht!]
  \centering
  \includegraphics[width=\textwidth]{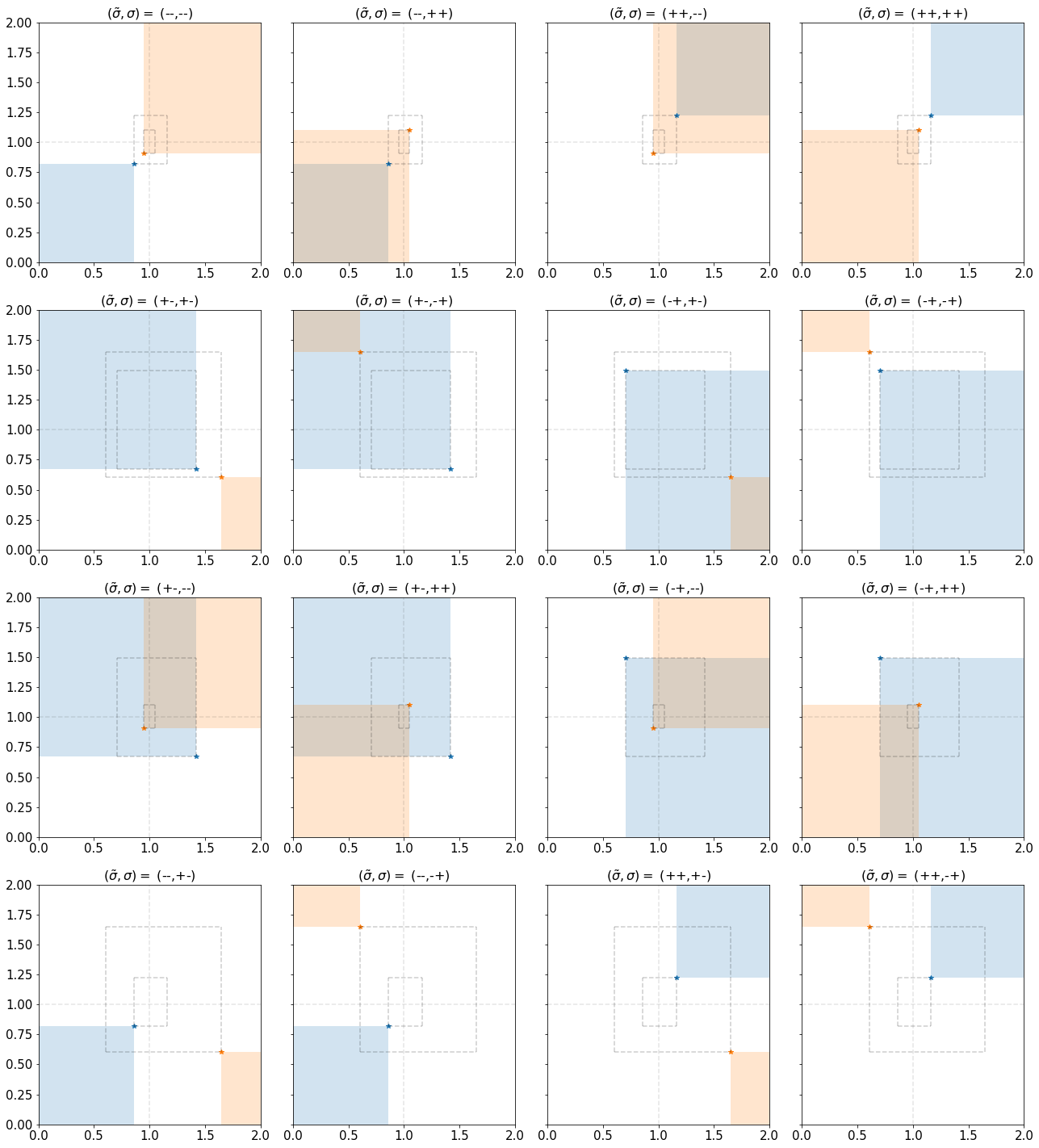}
  \caption{Depiction of the $16$ different cases for the absolute domains in the proof of Theorem \ref{thm:boundedness_large2small}.
  Each row corresponds to $\tilde{\sigma} \in \Sigma^{k}$ and $\sigma\in \Sigma^l$, for $k,l \in \{1, -1\}$. The blue rectangle represents the
  domain of holomorphicity for the function $f$, and the orange rectangle the domain
  of holomorphicity for the function $g$. We chose
  $P=\mathbb{I}$, $\alpha = (0.05, 0.1), A =(0.15, 0.2), \Gamma=(0.35, 0.4)$ and $\gamma=(0.5, 0.5)$.}
  \label{fig:proof_domains}
\end{figure}

\begin{proof}
Let $\mathcal{S}_{\delta, \delta'} = \{\hat{H}^\sigma_{\delta}: \sigma \in \Sigma^1\} \cup \{\hat{H}^\sigma_{\delta'}: \sigma \in \Sigma^{-1}\}$
be the collection of four spaces such that $H_{A,-\Gamma} = \bigoplus \mathcal{S}_{A,-\Gamma}$ and
$H_{\alpha,-\gamma} = \bigoplus \mathcal{S}_{\alpha,-\gamma}$.
By definition of the norm on $H_{A, -\Gamma}$ it is enough to prove that there is a constant $K > 0$ such that
\begin{equation}
\| \Pi C_T \tilde{\Pi}f\|_{H}  \leq K \|f\|_{\tilde{H}}, \qquad f \in \tilde{H},
\end{equation}
for any $H \in \mathcal{S}_{\alpha,-\gamma}$ and  $\tilde{H} \in \mathcal{S}_{A,-\Gamma}$, where
$\Pi\colon H_{\alpha,-\gamma} \to H$ and $\tilde{\Pi}\colon H_{A,-\Gamma}\to \tilde{H}$
are the respective projection operators.
We will denote by $D$, $\tilde{D}$ the domains of holomorphicity of $H$ and $\tilde{H}$, respectively,
and we will write $D'$ for the dual
domain
of $D$.
Using \eqref{eq:Dsd_norm_dual}, for any $f \in \tilde{H}$ we have
\[
\| \Pi C_T f\|_{H}
= \sup \left\{ \left| \langle C_T f, g \rangle_{\partial^*D} \right|
\colon g\in H^2(D') \cap \mathcal{P}, \|g\|_{H^2(D')} \leq 1 \right\}.
\]
(Note that this holds for all $H \in \mathcal{S}_{A,-\Gamma}$ without the need to adapt the function space for $g$.)
By the density of Laurent polynomials $\mathcal{P}$ in $H_{A,-\Gamma}$, see Proposition \ref{prop:characterisations}$(iv)$,
it is enough to establish
\[
\left| \langle C_T f, g \rangle_{\partial^* D} \right| \leq \tilde{C} \|f\|_{\tilde{H}} \|g\|_{H^2(D')}
\]
 for some $\tilde{C} > 0$, for all $f\in \tilde{H} \cap \mathcal{P}$ and
$g\in H^2(D') \cap \mathcal{P}$.
We shall prove this by breaking the 16 possible different configurations (see Figure \ref{fig:proof_domains})
into several cases.

Fix $H \in \mathcal{S}_{\alpha,-\gamma}$ and $\tilde H \in \mathcal{S}_{A,-\Gamma}$
(with corresponding domains $D$ and $\tilde{D})$ and let $g\in H^2(D')$.

\begin{enumerate}[(i)]
\item We first consider the case $D = D^\sigma_{P,\alpha}$ with $\sigma \in \Sigma^1$ (so that $\partial^* D = \mathbb{T}^\sigma_{P,\alpha}$) and $\tilde D =  D^{\tilde \sigma}_{P,A}$ with $\tilde \sigma \in \Sigma^1$.
Since $T$ is $(1,\alpha,A,P)$-strongly expanding,
Lemma \ref{lem:contour_deform} yields that there is a $\hat \sigma \in \Sigma^1$ such that $\mathbb{T}^{\hat\sigma}_{P,\alpha} \subset \cl{D'}$
and $T(\mathbb{T}^{\hat\sigma}_{P,\alpha})$ is a compact subset of $\tilde D$.

We obtain
\begin{align*}
  | \langle C_T f, g \rangle_{\mathbb{T}^\sigma_{P,\alpha}}| &= \left|\int_{\mathbb{T}^\sigma_{P,\alpha}} (f \circ T) g \, dm\right| =
  \left|\int_{\mathbb{T}^{\hat{\sigma}}_{P,\alpha}} (f \circ T) g \, dm\right| \\
&\leq \left( \int_{\mathbb{T}^{\hat{\sigma}}_{P, \alpha}}|f \circ T|^2 \,dm \right)^{1/2}
\left( \int_{\mathbb{T}^{\hat{\sigma}}_{P,\alpha}}|g|^2 \,dm \right)^{1/2},
\end{align*}
where the integral equality (in the case $\hat \sigma \neq \sigma$)
follows by Cauchy's Theorem and the holomorphicity of $T$ on $\cl{\mathcal{A}_{P,\alpha}}$,
and the last step is the Cauchy-Schwarz inequality.

As $T(\mathbb{T}^{\hat{\sigma}}_{P,\alpha}) \subset \tilde{D}$ is compact, Proposition \ref{prop:sup_comp}
  yields a $C_1 > 0$ such that $\sup_{z\in T(\mathbb{T}^{\tilde{\sigma}}_{P,\alpha})} |f(z)| \leq C_1 \|f\|_{\tilde{H}}$.
  Since $\mathbb{T}^{\hat{\sigma}}_{P,\alpha} \in \cl{D'}$ and $g\in H^2(D')$, we obtain
  \[| \langle C_T f, g \rangle_{\mathbb{T}^\sigma_{P,\alpha}}|  \leq C_1 \|f\|_{\tilde{H}} \|g\|_{H^2(D')}.\]

\item The case $D = D^\sigma_{P,\alpha}$ with $\sigma \in \Sigma^1$ and
  $\tilde D = (D^{\tilde \sigma}_{P,\Gamma})'$ with $\tilde \sigma \in \Sigma^{-1}$ is similar to the previous one.
  In this case, it holds that $T(\mathbb{T}^2) = \mathbb{T}^2 \cc D' \cap \tilde D$, and using holomorphicity of $T$ on
  $\cl{\mathcal{A}_{P,\alpha}}$ and the Cauchy-Schwarz inequality, as before we obtain
  \begin{equation*}
  | \langle C_T f, g \rangle_{\mathbb{T}^\sigma_{P,\alpha}}| =
  \left|\int_{\mathbb{T}^2} (f \circ T) g \, dm\right|
\leq \left( \int_{\mathbb{T}^2}|f \circ T|^2 \,dm \right)^{1/2}
\left( \int_{\mathbb{T}^2}|g|^2 \,dm \right)^{1/2}.
\end{equation*}
By Proposition \ref{prop:sup_comp}, we again have that
$| \langle C_T f, g \rangle_{\mathbb{T}^\sigma_{P,\alpha}}|  \leq C_2 \|f\|_{\tilde{H}} \|g\|_{H^2(D')}$
for some $C_2 > 0$.

\item Next we consider the case $D = (D^\sigma_{P,\gamma})', \sigma \in \Sigma^{-1}$ and
  $\tilde D =  D^{\tilde \sigma}_{P,A}, \tilde \sigma \in \Sigma^1$. From assumption $(iii)$ we have that there exists a torus
  $\mathbb{T}_q^2 \subset D^{\sigma}_{P,\gamma} \cap \mathcal{A}_{P,\eta}$ such that $T(\mathbb{T}^2_q)$ is a compact subset of $\tilde D$.
  As in case (i), we obtain
 \begin{align*}
  | \langle C_T f, g \rangle_{\mathbb{T}^\sigma_{P,\gamma}}|
  &\leq \left( \int_{\mathbb{T}^{2}_q}|f \circ T|^2 \,dm \right)^{1/2}
\left( \int_{\mathbb{T}^{2}_q}|g|^2 \,dm \right)^{1/2}
\leq C_3 \| f \|_{\tilde H} \| g \|_{H^2(D')}
\end{align*}
for some $C_3 > 0$.

\item Finally we consider the case $D = (D^\sigma_{P,\gamma})',\tilde{D} = (D^{\tilde \sigma}_{P,\Gamma})'$
with $\sigma, \tilde \sigma \in \Sigma^{-1}$. As $T$ and $T^{-1}$ are
holomorphic on a neighbourhood of $\cl{\mathcal{A}_{P,\gamma}}$, by Cauchy's Theorem we have
\begin{align*}
  \langle C_T f, g \rangle_{\mathbb{T}^\sigma_{P,\gamma}} &= \int_{\mathbb{T}^\sigma_{P,\gamma}} (f \circ T) g \, dm =
  \int_{\Ttwo} (f \circ T) g \, dm = \int_{\Ttwo} f (g \circ T^{-1}) w \, dm \\
 & = \int_{ \mathbb{T}^\sigma_{P,\Gamma}} f (g \circ T^{-1}) w \, dm,
\end{align*}
where $w(z) = \omega_T \det (DT^{-1}(z)) \cdot z / T^{-1}(z)$ with
$\omega_T = 1$ if $T$ is orientation-preserving and $\omega_T = -1$ otherwise.
As $T^{-1}$ is $(-1, \Gamma, \gamma, P)$-strongly expanding, by Lemma \ref{lem:contour_deform}, there is
$\hat{\sigma}\in\Sigma^{-1}$ such that $ \mathbb{T}^{\hat{\sigma}}_\Gamma \in \cl{(D^{\tilde{\sigma}}_{P,\Gamma})'} = \cl{\tilde D}$ and
$T^{-1}(\mathbb{T}^{\hat{\sigma}}_{P,\Gamma})$ is a compact set in $D^\sigma_{P,\gamma} = D'$. By the same argument as before
we obtain a $C_4 > 0$ such that
\begin{align*}
  | \langle C_T f, g \rangle_{\mathbb{T}^\sigma_{P,\gamma}}|
  &\leq  \sup_{z\in\T^\sigma_{P,\Gamma}} | w(z) | \left(\int_{ \mathbb{T}^\sigma_{P,\Gamma}} |f|^2 \,dm\right)^{1/2}
  \left(\int_{ \mathbb{T}^\sigma_{P,\Gamma}} |g \circ T^{-1}|^2\, dm\right)^{1/2} \\
  &\leq C_4 \|f\|_{\tilde{H}} \|g\|_{H^2(D')}.
\end{align*}
\end{enumerate}

Setting $\tilde C = \max\{C_1, C_2, C_3, C_4\}$, we obtain the required inequality.
\end{proof}

\begin{cor}
Let $T$ be an analytic Anosov diffeomorphism of $\T^2$
satisfying the {\it (sec)} condition for some $P \in \GLtwo(\R)$.
Then there are $\alpha, A, \gamma, \Gamma \in \R^2_{>0}$ with
$\alpha < A$ and $\Gamma < \gamma$ such that
the associated composition operator $C_T$ is a well-defined and bounded operator from $H_{P, A, -\Gamma}$ to~$H_{P,\alpha, -\gamma}$.
\end{cor}

\begin{rem} \label{rem:no_sepcond}
Note that the {\it (sec)} condition is sufficient, but not necessary for the above corollary. For example, take $P = \mathbb{I}$ and $\tilde{T}(x, y) = (x+y, x)$, then
the union of the first and third quadrants of $\Rtwo$ is not an expanding invariant cone, as
$D_x\tilde{T} (0, 1)^T = (1, 0)^T$ for all $x\in M$. However, it is not difficult to find
$\alpha, A, \gamma, \Gamma$ satisfying the assumptions of Theorem \ref{thm:boundedness_large2small} (for $P = \mathbb{I}$)
for the corresponding map $T \colon (z_1, z_2) \mapsto (z_1 z_2, z_1)$ on~$\Ttwo$.
\end{rem}

Using a standard factorisation argument we can now deduce that the composition
operator from the above theorem is trace-class when considered as an operator on $H_{A, -\Gamma}$.
We defer the proof of the following lemma to the appendix.

\begin{lem}\label{lem:J_stretched}
For $\alpha, A, \gamma, \Gamma \in \R^2_{>0}$ with $\alpha < A$ and $\Gamma < \gamma$ let $J\colon H_{P, \alpha, -\gamma} \to H_{P, A, -\Gamma}$ be the
canonical embedding operator. For $n \in \N$, denote by $s_n(J)$ the $n$-th singular value
of $J$. Then
\[
\lim_{n\to\infty} \frac{- \log s_n(J)}{n^{1/2}} = \eta,
\]
where $\eta = \left( \frac{1}{\log (A_1 - \alpha_1) \log (A_2 - \alpha_2)} + \frac{1}{\log (\gamma_1 - \Gamma_1) \log (\gamma_2 - \Gamma_2)}\right)^{-1/2}$.
\end{lem}

\begin{cor} \label{cor:traceclass}
Let $T$ be as in Theorem \ref{thm:boundedness_large2small}.
Then the are $\alpha, \gamma \in \R^2_{>0}$ such that
the composition operator $C_T$ associated to $T$ is well-defined as an operator from $H_{P, \alpha, -\gamma}$ to $H_{P, \alpha, -\gamma}$.
Moreover, there are constants $\tilde{c}_1, \tilde{c}_2, \hat{c}_1, \hat{c}_2 >0$ such that
\[s_n(C_T) \leq \tilde{c}_1 e^{-\tilde{c}_2 n^{1/2}} \quad (n\in \mathbb{N}),\]
and
\[ |\lambda_n(C_T)| \leq \hat{c}_1 e^{-\hat{c}_2 n^{1/2}}\quad (n\in \mathbb{N}),\]
where $s_n(C_T)$ and $\lambda_n(C_T)$ are the $n-$th largest (counted with multiplicity) singular values and eigenvalues of $C_T$, respectively.
In particular, $C_T: H_{P, \alpha, -\gamma} \to H_{P, \alpha, -\gamma}$ is trace-class.
\end{cor}

\begin{proof}
By Theorem \ref{thm:boundedness_large2small} the composition operator lifts to a continuous operator
 $\tilde{C}_T\colon H_{P, A, -\Gamma} \to H_{P, \alpha, -\gamma}$. Let $J$ be the embedding operator from
Lemma~\ref{lem:J_stretched}, then $C_T = \tilde{C}_T J$ is a well-defined trace-class operator from $H_{P, \alpha, -\gamma}$ to
$H_{P, \alpha, -\gamma}$, as $\tilde{C}_T$ is bounded and $J$ is trace-class. By \cite[2.2]{Pi} we have
$s_n(C_T) \leq C s_n(J)$ with $C = \|\tilde{C}_T\|_{H_{P, A, -\Gamma}\to H_{P, \alpha, -\gamma}}$,
and by Lemma~\ref{lem:J_stretched} we have $s_n(J) \leq c_1 e^{-\eta \cdot n^{1/2}}$ for some $c_1 > 0$.
Using the multiplicative Weyl inequality \cite[3.5.1]{Pi} we obtain
$|\lambda_n(C_T)| \leq \hat{c}_1 e^{-\hat{c}_2 n^{1/2}}$ with  $\hat{c}_1 = C c_1, \hat{c}_2 = 2/3 \eta$
(see, for example, \cite[Lemma 5.11]{BJ2}).
\end{proof}

We are now ready to precisely state and prove our first main theorem.

\begin{reptheorem}{thm:thmA}
Let $T$ be an analytic Anosov diffeomorphism of $\T^2$
satisfying the {\it (sec)} condition for some $P \in \GLtwo(\R)$,
and let $w \colon \T^2 \to \C$ be analytic. Then there exist $\alpha, \gamma \in \R^2_{>0}$ such that the weighted composition operator
\[ f \mapsto w \cdot f \circ T\]
is a well-defined trace-class operator on $H_{P, \alpha,-\gamma}$.
\end{reptheorem}

\begin{proof}
Corollary \ref{cor:traceclass} proves the theorem for $w \equiv 1$, and remains valid
if the operator $C_T$ is replaced by $M_w C_T$, with $M_w$ the multiplication operator with a weight function
that is holomorphic on $\mathcal{A}_{P,\eta}$ for some $\eta > \alpha, \gamma$. Since $w$ is analytic on $\T^2$, for $\eta$ sufficiently small
we can assume without loss of generality that $w$ holomorphically extends to $\mathcal{A}_{P,\eta}$,
proving the general case.
\end{proof}

\subsection{Relation to transfer operator}

The following lemma is analogous to Proposition \ref{prop:iso},
replacing the Hilbert space $H^\sigma_{P,\delta}$ by $\mathcal{H}^1_{P,\delta}$, $\mathcal{H}^{-1}_{P,\delta}$ or
$\mathcal{H}^1_{P,\alpha} \oplus \mathcal{H}^{-1}_{P,-\gamma}$, with the respective inner products
as in Definition \ref{def:combined_spaces}.

\begin{lem}\label{lem:tuple_isom}~
\begin{enumerate}[(i)]
\item For $\ell\in\{1, -1\}$ and $\delta\in \Rtwo$, the dual
$(\mathcal{H}^\ell_{P,\delta})' = (\bigoplus_{\sigma \in \Sigma^\ell} \hat{H}^\sigma_{P,\delta})'$ is
isometrically isomorphic to
$\bigoplus_{\sigma \in \Sigma^\ell} \check{H}^{-\sigma}_{P,-\delta}$
via the isomorphism
$\mathcal{J}^{\ell}_{P,\delta}\colon \bigoplus_{\sigma \in \Sigma^\ell} \check{H}^{-\sigma}_{P,-\delta} \to (\mathcal{H}^\ell_{P,\delta})'$, $g \mapsto l_g$, given by
\[l_g(f) = \sum_{\sigma \in \Sigma^\ell} \langle \hat{\Pi}^\sigma f,  \check{\Pi}^{-\sigma} g \rangle_{\T^\sigma_{P,\delta}}.\]

\item For $\alpha, \gamma \in \Rtwo$, the dual
$(\mathcal{H}^1_{P,\alpha} \oplus \mathcal{H}^{-1}_{P,-\gamma})'
= (\hat{H}^{--}_{P,\alpha} \oplus \hat{H}^{++}_{P,\alpha} \oplus \hat{H}^{-+}_{P,-\gamma} \oplus \hat{H}^{+-}_{P,-\gamma})'$
is isometrically isomorphic to
$\check{H}^{++}_{P,-\alpha} \oplus \check{H}^{--}_{P,-\alpha} \oplus \check{H}^{+-}_{P,\gamma} \oplus \check{H}^{-+}_{P,\gamma}$
via the isomorphism
$\mathcal{J}_{P, \alpha, \gamma}\colon g \mapsto l_g$,
given by
\[l_g(f) =  \sum_{\sigma \in \Sigma^1} \langle \hat{\Pi}^\sigma f, \check{\Pi}^{-\sigma} g \rangle_{\mathbb{T}^\sigma_{P,\alpha}}
+  \sum_{\sigma \in \Sigma^{-1}} \langle \hat{\Pi}^{\sigma} f,  \check{\Pi}^{-\sigma} g \rangle_{\mathbb{T}^\sigma_{P,-\gamma}}.\]
\end{enumerate}
\end{lem}

\begin{rem}
As before we can identify the Hardy-Hilbert space $H_{P, -\alpha, \gamma}$
associated to a cone-wise exponential weight (similar to Definition \ref{defn:weight}), with a topological direct sum
of Hardy-Hilbert spaces on corresponding Reinhardt domains, that is,
\[
H_{P, -\alpha, \gamma}
\cong \check{H}^{++}_{P,-\alpha} \oplus \check{H}^{--}_{P,-\alpha} \oplus \check{H}^{+-}_{P,\gamma} \oplus \check{H}^{-+}_{P,\gamma} \\
\]
Note that in particular this implies $(H_{P, \alpha, -\gamma})' \cong H_{P, -\alpha, \gamma}$.
\end{rem}

Let $\tilde{T}\colon M \to M$ be a smooth diffeomorphism of $M$, then the associated
Perron-Frobenius operator $\mathcal{L}_{\tilde{T}}$ given by
$g \mapsto | \det {D \tilde{T}^{-1}} | \cdot (g\circ\tilde{T}^{-1})$ is a well-defined operator on $L^2(M)$.
The respective operator on $L^2(\Ttwo)$ is given by\footnote{Use the relation $\pi \circ \tilde{T} = T \circ \pi$
with $\pi(x,y) = (\exp(i x), \exp(i y))$. Also observe that
$\det {D \tilde{T}^{-1}}(x, y) =
(\det {DT^{-1}} (z)) \cdot \frac{z_1 z_2}{(T^{-1}(z))_1(T^{-1}(z))_2}$
for $z=(e^{ix}, e^{iy})$.}
\begin{equation}\label{eq:PF_op}
(\mathcal{L}_T g)(z) = w(z) \cdot (g\circ T^{-1})(z),
\end{equation}
with $w(z) = \omega_T \cdot (\det {DT^{-1}} (z)) \cdot \frac{p_{1,1}(z)}{p_{1,1}(T^{-1}(z))}$,
where $p_{1,1}(z) = z_1 \cdot z_2$ and
$\omega_T = 1$ if $T$ is orientation-preserving and $\omega_T = -1$ otherwise.

\begin{prop}
Let $T$ satisfy the assumptions of Theorem \ref{thm:boundedness_large2small} and let $C_T$ be
the respective operator on a suitable space $H_{P, \alpha, -\gamma}$. Then the isomorphism $\mathcal{J} = \mathcal{J}_{P, \alpha, \gamma}$ from Lemma \ref{lem:tuple_isom}
conjugates the adjoint $(C_T)^*$ of $C_T$ to the operator $\mathcal{L}_T$ given by \eqref{eq:PF_op}, which is well defined and bounded{} as an operator on $H_{P, -\alpha, \gamma}$.
\end{prop}

\begin{proof}
For notational simplicity, we assume $P = \mathbb{I}$, the proof for general $P \in \GLtwo(\R)$ being identical.
We want to show that $C_T^* \mathcal{J} = \mathcal{J} \mathcal{L}_T$.
By the density of Laurent polynomials in $H_{P, -\alpha, \gamma}$ and $H_{P, \alpha, -\gamma}$, it suffices to show this for monomials, i.e.
$(C_T^* \mathcal{J}(p_n))(p_m) = (\mathcal{J} \mathcal{L}_T(p_n))(p_m)$ for all $n, m \in \Z^2$.
For any $n, m \in \mathbb{Z}^2$ we have that
\[
(C_T^* \mathcal{J}(p_n))(p_m)
= (\mathcal{J}(p_n)) (C_T p_m)
= \sum_{\sigma \in \Sigma} \langle \hat{\Pi}^\sigma (C_T p_m), \check{\Pi}^{-\sigma} p_n \rangle_{\mathbb{T}^\sigma_{\delta(\sigma)}}
\]
with $\delta(\sigma) = \alpha$ for $\sigma \in \Sigma^1$, and $\delta(\sigma) = -\gamma$ otherwise.
We note that $\mathbb{Z}^2 = \bigcup_{\sigma \in \Sigma} \check{\mathbb{Z}}^{\sigma}$
is a disjoint union, so that
for every $n\in \mathbb{Z}^2$ there exists exactly one $\sigma' \in \Sigma$ such that
$\check{\Pi}^{- \sigma'} p_n = p_n$ and $\check{\Pi}^{- \sigma} p_n = 0$ for all $\sigma \neq \sigma'$. Moreover we have
$\langle \hat{\Pi}^\sigma f, \check{\Pi}^{\sigma'} g \rangle_{\mathbb{T}^\sigma_\delta} = 0$
whenever $\sigma' \neq -\sigma$. It follows that
\[
\sum_{\sigma \in \Sigma} \langle \hat{\Pi}^\sigma (C_T p_m),
\check{\Pi}^{-\sigma} p_n\rangle_{\mathbb{T}^\sigma_{\delta(\sigma)}}
= \langle \hat{\Pi}^{\sigma'}(C_T p_m), \check{\Pi}^{- \sigma'} p_n\rangle_{\mathbb{T}^{\sigma'}_{\delta(\sigma')}}
= \langle C_T p_m, p_n\rangle_{\mathbb{T}^{\sigma'}_{\delta(\sigma')}},
\]
where the second step follows from $\sum_{\sigma \in \Sigma} \hat{\Pi}^\sigma(C_T p_m) = C_T p_m$.
Analogously,
\[
(\mathcal{J} \mathcal{L}_T(p_n))(p_m)
= \sum_{\sigma \in \Sigma} \langle \hat{\Pi}^\sigma p_m, \check{\Pi}^{-\sigma} (\mathcal{L}_T(p_n))\rangle_{\mathbb{T}_{\delta(\sigma)}^\sigma}
= \langle p_m, \mathcal{L}_T(p_n)\rangle_{\mathbb{T}_{\delta(\sigma'')}^{\sigma''}},
\]
with suitable $\sigma''$.
Finally, we have that
\begin{align*}
&\langle C_T p_m, p_n\rangle_{\mathbb{T}_{\delta(\sigma')}^{\sigma'}}
= \int_{\mathbb{T}_{\delta(\sigma')}^{\sigma'}} (p_m \circ T) p_n \,dm
= \int_{\mathbb{T}^2} (p_m \circ T) p_n \,dm \\
= &\int_{\mathbb{T}^2} p_m  (p_n \circ T^{-1}) w \,dm
= \int_{\mathbb{T}_{\delta(\sigma'')}^{\sigma''}} p_m  (p_n \circ T^{-1}) w \,dm
= \langle p_m, \mathcal{L}_T(p_n)\rangle_{\mathbb{T}_{\delta(\sigma'')}^{\sigma''}},
\end{align*}
where $w(z) = \omega_T \det (DT^{-1}(z))\cdot z/T^{-1}(z)$ and we have used that the integrands are holomorphic
on a neighbourhood of $\cl{A_{\gamma}}$. Combining the above, we obtain the claim of the proposition.
\end{proof}

\begin{rem}
Using Theorem \ref{thm:thmA} with the weight function being (the complex version of) the determinant of $DT$ gives rise to a transfer operator corresponding
to the map $T^{-1}$, which is well defined and trace-class on a suitable $H_{P, \alpha, -\gamma}$. Now, using the previous proposition,
it follows that the operator $f \mapsto f\circ T^{-1}$ is well-defined and trace-class on $H_{P,-\alpha, \gamma}$.
\end{rem}

\section{Resonances for certain rational Anosov maps}\label{sec:resonances_rational}

This section is devoted to proving Theorem \ref{thm:thmB}, that is, determining
the explicit form of eigenvalues of composition operators associated to analytic maps
with holomorphic extensions to certain domains of $\Ctwohat$. In order to capture all the
intricacies of the resonances in this result, we will use a fundamental result by
Rudin and Stout \cite[Theorem 5.2.5]{R} characterising inner functions on polydisks.
An inner function on $\D^n$ is a function $f\in H^\infty(\D^n)$ whose radial
boundary values satisfy $|f^*(z)|=1$ almost everywhere\footnote{
Knese \cite[Cor.~14.6]{Kn} proved that the radial limit exists and is unimodular at every point $z\in \T^n$.
However, in general $f^*$ need not be continuous on $\T^n$.} on $\T^n = \partial \D^n$.
We denote by $U(\D^n)$ the
class of all continuous complex functions on $\cl{\D^n}$
whose restriction to $\D^n$ is holomorphic.

\begin{theorem}\label{thm:Rud}
Every rational inner
function $f$ on $\D^n$, $n \in \N$, has the form
\begin{equation}\label{eq:Rud}
f(z) = \frac{M(z) \tilde{Q}(1/z)}{Q(z)},
\end{equation}
where $M$ is a monomial, $Q$ a polynomial with no
zeros in $\D^n$, and $\tilde{Q}$ is the polynomial whose
coefficients are the complex conjugates of the coefficients of $Q$.
Moreover, every function
$f\in U(\D^n)$ which is inner is
rational, and in this case $Q$ has no zeros in $\cl{\D^n}$.
\end{theorem}

A direct consequence of this theorem is a characterization of the form that
analytic Anosov diffeomorphisms on $\T^2$ from Theorem \ref{thm:thmB} can take,
namely that each component is necessarily a rational function satisfying a certain set of properties.
For for $k,l\in\{0,1\}$, we denote by $I_{kl}\colon \Ttwo \to \Ttwo$ the map
\[I_{kl}(z_1, z_2)=(z_1^{1-2k}, z_2^{1-2l}),\]
noting that $I_{kl} = \tau_{I^\sigma}$ with $\sigma = (1-2k, 1-2l)$.

\begin{cor}\label{cor:TA}
Let $T$ be an analytic diffeomorphism of $\T^2$ with holomorphic extension to a neighbourhood of $\T^2$.
Assume there exist $\sigma,\sigma' \in \Sigma$ so that $T$ holomorphically extends to $D^\sigma$ with $T(D^\sigma) \subset D^{\sigma'}$.
Then each component of $T$ can be written as a rational function. Moreover, writing $T = T_A$ with $A = (a_1, \ldots, a_n)$ the collection of
all coefficients occurring in $T$ (in any order), and denoting $\cj{A} = (\cj{a_1}, \ldots, \cj{a_n})$,
we have the following properties:
\begin{enumerate}[(i)]
  \item $I_{11} \circ T_A \circ I_{11} = T_{\cj{A}}$,
  \item $T_{\cj{A}}(\cj{z}) = \cj{T_A(z)}$ for any $z \in \Ctwohat$.
\end{enumerate}
\end{cor}

\begin{proof}
We define $\phi^\sigma \colon \D^2 \to D^\sigma$ by $\phi^\sigma(z) = z^{-\sigma} = (z_1^{-\sigma_1}, z_2^{-\sigma_2})$, then
each component of $\hat T = (\phi^{\sigma'})^{-1} \circ T \circ \phi^\sigma \colon \D^2 \to \D^2$ is a rational inner function in $\D^2$ continuous
on all of $\cl{\D^2}$. By Theorem \ref{thm:Rud}, each component is a rational function of the form \eqref{eq:Rud},
and this property is preserved under composition with $\phi^\sigma$.
Furthermore, property $(i)$ holds for maps whose components are of the form \eqref{eq:Rud}.
Since $I_{11}$ and $\phi^\sigma$ commute for any $\sigma \in \Sigma$, we have that
\[
I_{11} \circ T_A \circ I_{11}
 = \phi^{\sigma'} \circ I_{11} \circ \hat{T}_A \circ I_{11} \circ (\phi^\sigma)^{-1}
= \phi^{\sigma'} \circ \hat{T}_{\cj{A}} \circ (\phi^\sigma)^{-1}
= T_{\cj{A}},
\]
proving that property $(i)$ is preserved under composition with $\phi^\sigma$. Lastly, $(ii)$ holds for all maps whose components
are rational functions, and hence for the given map $T = T_A$.
\end{proof}

We shall require the following lemma, a direct consequence of the maximum modulus principle.
\begin{lem}\label{lem:max_mod}
Fix $\sigma, \sigma' \in \Sigma$, $a \in \R^2_{>0}$, and let $T \colon D^\sigma \to D^{\sigma'}$ be holomorphic with
$T(\T^\sigma_a) \subset D^{\sigma'}_a$. Then $T(D^\sigma_a) \cc D^{\sigma'}_a$.
\end{lem}

\begin{proof}
We write $\hat \sigma = (-1, -1)$, and begin with the case $\sigma = \sigma' = \hat \sigma$
(i.e.~$D^\sigma = D^{\sigma'} = \D^2$), whose proof is a direct application of the
(multivariate) maximum modulus principle. By compactness,
there exists $z^* = (z_1^*, z_2^*) \in \cl{D^\sigma_a}$ such that $| T_1 (z^*) | = \max_{z \in \cl{D^\sigma_a}} | T_1 (z) |$.
Defining $f \colon \D \to \D$ by $f(z_1) = T_1(z_1, z^*_2)$, the maximum modulus principle implies that
\[ | T_1(z^*) | = \max_{|z_1| \leq e^{-a_1}} | f(z_1) | = \max_{|z_1| = e^{-a_1}} | f(z_1) |,
\]
and hence without loss of generality we can assume $|z_1^*| = e^{-a_1}$.
Analogously, $z_2^*$ can be assumed to have modulus $e^{-a_2}$.
It follows that $\max_{z \in \cl{D^\sigma_a}} | T_1(z) | = \max_{z \in \T^\sigma_a} | T_1(z) | < e^{-a_1}$
(using $T(\T^\sigma_a) \subset D^\sigma_a$).
Repeating the argument for $T_2$ yields $\max_{z \in \cl{D^\sigma_a}} | T_2(z) | < e^{-a_2}$,
and hence $T(D^\sigma_a) \cc D^{\sigma'}_a$.

For general $\sigma, \sigma' \in \Sigma$, let $\phi^\sigma \colon \D^2 \to D^\sigma$
be given by $\phi^\sigma(z) = z^{-\sigma} = (z_1^{-\sigma_1}, z_2^{-\sigma_2})$,
so that $\phi(D^{\hat\sigma}_a) = D^\sigma_a$ and $\phi(\T^{\hat \sigma}_a) = \T^\sigma_a$.
Then $\tilde T = (\phi^{\sigma'})^{-1} \circ T \circ \phi^\sigma \colon \D^2 \to \D^2$ is holomorphic
and satisfies $\tilde T(\T^{\hat \sigma}_a) \subset D^{\hat \sigma}_a$.
It follows that $\tilde T(D^{\hat \sigma}_a) \cc D^{\hat \sigma}_a$, and
the assertion $T(D^\sigma_a) \cc D^{\sigma'}_a$ follows.
\end{proof}

The explicit form of the spectral determinant in Theorem \ref{thm:thmB} will follow by computing
traces of certain trace-class operators.
If $L\colon H\to H$ is a trace-class operator on a separable Hilbert space $(H, (\cdot, \cdot))$ and
$\{e_n\}_{n\in \mathcal{I}}$ is an orthonormal basis of $H$ with some index set $\mathcal{I}$, then the trace
of $L$ is
\begin{equation}\label{eq:tr_def}
\operatorname{Tr}(L) = \sum_{n \in \mathcal{I}} (L e_n, e_n),
\end{equation}
and its determinant is given by
\begin{equation}\label{eq:det_def}
\det{(\operatorname{Id} - z L)} =
\exp \left( -\sum_{k=1}^{\infty} \frac{z^k}{k} \operatorname{Tr} (L^k)\right),
\end{equation}
for all $z\in \C$ in a suitable neighbourhood of $0$. Moreover, both
$\operatorname{Tr}$ and $\det$ are spectral, that is,
$\operatorname{Tr}(L) = \sum_{k=1}^\infty \lambda_k(L)$, and counting multiplicities,
the zeros of the entire function $z \mapsto \det{(\operatorname{Id} - z L)}$ are precisely the
reciprocals of the eigenvalues $\lambda_k(L)$ of $L$ (see \cite[4.62, 4.7.14, 4.7.15]{Pi}).

\begin{notation}
The multipliers at a fixed point $z^*\in\Ctwohat$ of a rational map $T\colon \Ctwohat \to \Ctwohat$
are given by the eigenvalues of $D T(z^*)$ if $z^*\in\Ctwo$, and
by the eigenvalues of $D \hat{T}(I_{kl} (z^*))$ with $(k,l) = (1,1), (1,0)$ or $(0,1)$
for $z^* = (\infty, \infty), (\infty, w)$ or $(w, \infty)$ respectively, where $w\in\C$ and
$\hat{T} = I_{kl} \circ T \circ I_{kl}$.
\end{notation}

\begin{lem}\label{lem:trace_fp}
For $\sigma\in\Sigma$ and $\delta \in \Rtwo$ let $\varphi\colon \Dsd \to \Dsd$ be a holomorphic map
such that $\varphi(\Dsd) \cc \Dsd$. Let $C_\varphi$ denote the corresponding composition operator
and $M_w$ the multiplication operator with $w$ a holomorphic function on a neighbourhood of $\cl{\Dsd}$.
Then $M_w C_\varphi$ is trace-class on $H^\sigma_\delta$ and
\[\operatorname{Tr} ((M_w C_\varphi)^k) = \frac{w(z^*)^k}{\det(\mathbb{I} - D\varphi^k(z^*))} = \frac{w(z^*)^k}{(1-\mu^k_1)(1-\mu^k_2)},\]
with $\mu_1, \mu_2$ the (not necessarily distinct) multipliers at the unique attracting fixed point $z^*\in \Dsd$ of $\varphi$.
\end{lem}

\begin{proof}
Let $\hat \sigma = (-1, -1)$ and consider the case $\sigma=\hat\sigma$.
As $H^{\hat{\sigma}}_\delta $ is a `favourable Hilbert space' (see \cite[Definition 2.7]{BJ}),
the result follows by \cite[Proposition 2.10 and Theorem 4.2]{BJ}.

For general $\sigma \in \Sigma$, $\Dsd$ is biholomorphically equivalent to $D^{\hat\sigma}_\delta$
under the map $\phi^\sigma\colon D^{\hat\sigma}_\delta \to \Dsd$ given by $\phi^\sigma(z) = z^{-\sigma}$,
and $\hat\varphi = (\phi^\sigma)^{-1} \circ \varphi \circ \phi^\sigma\colon D^{\hat\sigma}_\delta \to D^{\hat\sigma}_\delta$
satisfies $\hat\varphi(D^{\hat\sigma}_\delta) \cc D^{\hat\sigma}_\delta$.
Defining $\hat w = w \circ \phi^\sigma$ on a neighbourhood of $\cl{D^{\hat\sigma}_\delta}$, the first case implies
the statement of the lemma for the operator $M_{\hat{w}} C_{\hat{\varphi}}$ on $H^{\hat{\sigma}}_\delta$.
By Proposition~\ref{prop:isoiso} the operator $C_{\phi^\sigma}\colon H^\sigma_\delta \to H^{\hat{\sigma}}_\delta$ is an isometric isomorphism, which conjugates $M_w C_{\varphi}$ to $M_{\hat{w}} C_{\hat{\varphi}}$.
The statement for $M_w C_{\varphi}$ follows,
as the multipliers of $z^* \in D^{\sigma}_\delta$ for $\varphi$ coincide with those of the
unique attracting fixed point $(\phi^\sigma)^{-1} (z^*) \in  D^{\hat\sigma}_\delta$ for $\hat \varphi$.
\end{proof}

\begin{lem}\label{lem:off_conjugation}
Let $T$ be a smooth diffeomorphism of $\Ttwo$, and let $\omega_T = 1$ if $T$ is orientation-preserving,
and $\omega_T = -1$ otherwise. Let $r, s \in \{0, 1\}$, and write
$\hat{T} = I_{rs} \circ T^{-1} \circ I_{rs}$ and $\hat{w} = \omega_T \det D\hat{T}$. Then for any $n \in \mathbb{Z}^2$and $k \in \N$ we have
\[
(C_T^k p_n, p_n)_{L^2(\Ttwo)} = ((M_{\hat{w}} C_{\hat{T}})^k p_m, p_m)_{L^2(\Ttwo)}  \text{ with } m = ((-1)^{1-r} n_1-1, (-1)^{1-s} n_2-1).
\]
\end{lem}

\begin{proof}
All the steps follow by change of variables. For any $n \in \mathbb{Z}^2$, we have
\begin{align*}
(C_T^k p_n, p_n)_{L^2(\Ttwo)}
&= \frac{1}{(2\pi)^2}\int_\Ttwo p_n(T^k(z)) p_{-n}(z) \, \frac{dz_1}{z_1} \frac{dz_2}{z_2} \\
&= \frac{\omega_T^k}{(2\pi)^2}\int_\Ttwo p_n(w) p_{-n}(T^{-k}(w)) \det{D T^{-k} (w)} \,
\frac{dw_1}{(T^{-k}(w))_1} \frac{dw_2}{(T^{-k}(w))_2} \\
&= \frac{\omega_T^k}{(2\pi)^2}\int_\Ttwo p_{n+1}(w) p_{-(n+1)}(T^{-k}(w)) \det{D T^{-k} (w)} \,
\frac{dw_1}{w_1} \frac{dw_2}{w_2},
\end{align*}
using the shorthand $n+1 = (n_1+1, n_2+1)$. Next, observe that
\[
\det{D(T^{-k} \circ I_{rs})(z)} = \det{D (I_{rs} \circ \hat{T}^{k})(z)}
= (-1)^{r+s} p_{-2r,-2s}(\hat{T}^k(z)) \cdot \det{D \hat{T}^k(z)} .
\]
Substituting $w = I_{rs}(z)$ and using that $I_{rs}$ is orientation-preserving iff $r+s$ is even, we obtain
\begin{align*}
(C_T^k p_n, p_n)_{L^2(\Ttwo)}
&= \frac{(-1)^{r+s}\omega_T^k}{(2\pi)^2}\int_\Ttwo p_{n+1}(I_{rs}(z)) p_{-(n+1)}(T^{-k}(I_{rs}(z))) \det{D (T^{-k} \circ I_{rs})(z)} \,
\frac{dz_1}{z_1^{1-2r}} \frac{dz_2}{z_2^{1-2s}}, \\
&= \frac{\omega_T^k}{(2\pi)^2} \int_\Ttwo p_{-m}(z) p_m(\hat{T}^k(z)) \det{D \hat{T}^k (z)} \,
\frac{dz_1}{z_1} \frac{dz_2}{z_2} \\
&= ((\omega_T^k \det{D \hat{T}^k}) \cdot C_{\hat{T}}^k p_{m}, p_{m})_{L^2(\Ttwo)},
\end{align*}
with $m = (-(-1)^r n_1 -1, -(-1)^s n_2 -1)$, as claimed.
\end{proof}

\begin{rem}\label{rem:en_pn_equiv}
If $T\colon \Ttwo \to \Ttwo$ has an analytic extension to a neighbourhood of $\cl{\mathcal{A}_\delta}$ for some $\delta$, then
one can check that $(C_T p_n, p_n)_{L^2(\Ttwo)} =\langle C_T e_n, \cj{e_n} \rangle_{\Tsd} = (C_T e_n, e_n)_{H^\sigma_\delta}$
for any $n\in \Ztwo, \sigma\in\Sigma$.
\end{rem}

We recall that $\mathcal{H}^{1}_\delta = \hat{H}^{--}_{\delta} \oplus \hat{H}^{++}_{\delta}$
and $\mathcal{H}^{-1}_\delta = \hat{H}^{-+}_{\delta} \oplus \hat{H}^{+-}_{\delta}$,
and that $\{e_n : n \in \hat{\Z}^\sigma\}$ forms an orthonormal
basis for $\hat{H}^\sigma_\delta$ for $\sigma\in \Sigma$.

\begin{lem}\label{lem:tr_1}
Let $\delta \in \mathbb{R}^2$ and $T \colon \Ctwohat \to \Ctwohat$
be holomorphic on a neighbourhood of $\cl{\mathcal{A}_{\delta}}$. Assume additionally
that $T$ is holomorphic on $\Dsd$ and $T(\Dsd) \cc D^{-\sigma}_\delta$ for every $\sigma\in \Sigma^1$.
Let $w$ be a holomorphic function on a neighbourhood of $\cl{\mathcal{A}_{\delta}}$.
Then $M_w C_T$ is a well-defined and trace-class operator on $\mathcal{H}^1_\delta$
with trace
$\operatorname{Tr} (M_w C_T) =
\sum_{\sigma \in \Sigma^1} \sum_{n\in \hat\Z^\sigma} \langle M_w C_T e_n, e_n\rangle_{\T^\sigma_\delta}$.
Moreover,
\begin{enumerate}[(i)]
    \item if $w \equiv 1$, then $\operatorname{Tr} (M_w C_T) = \operatorname{Tr} (C_T) = 1$,
    \item if $w = \det DT$, then $\operatorname{Tr} (M_w C_T) = 0$.
\end{enumerate}
\end{lem}

\begin{proof}
We define $\tilde T = I_{11} \circ T$,
so that $\tilde T$ is holomorphic on a neighbourhood of $\cl{\mathcal{A}_\delta}$,
and moreover, for every $\sigma \in \Sigma^1$ it is holomorphic on $\Dsd$ with $\tilde T(\Dsd) \cc \Dsd$.
By Lemma \ref{lem:trace_fp} the composition operator $C_{\tilde T}$ is trace-class on $\hat{H}^\sigma_{\delta}$,
and hence $C_T$ is well-defined and trace-class viewed as an operator on $\mathcal{H}^1_\delta$.
Since $M_w$ is well-defined and bounded as an operator on $\mathcal{H}^1_\delta$,
it follows that $M_w C_T$ is also trace-class.
Since $\cj{e}_n(z) = p_{-n} (z) \nu(n)$ for $z\in \Tsd$ and $n \in \Ztwo$, we have that
\[
(M_w C_T e_n, e_n)_{\hat{H}^\sigma_\delta}
= \langle M_w C_T e_n, \cj{e}_n \rangle_{\T^\sigma_\delta}
= \langle w \cdot p_n \circ T, p_{-n} \rangle_{\T^\sigma_\delta}
= \langle w \cdot p_{-n} \circ \tilde T, p_{-n} \rangle_{\T^\sigma_\delta}.
\]
Moreover, $p_{-n} \circ \tilde{T} \in H^\sigma_\delta$ for any $\sigma \in \Sigma^1$ and $-n \in \hat{\mathbb{Z}}^\sigma$,
that is, $p_{-n} \circ \tilde{T} = \sum_{m \in \Z^\sigma} c_m p_m $.
For $w \equiv 1$, recalling that $\langle p_k, p_l \rangle_{\Tsd} = 0$ whenever $k \neq -l$,
and noting $(-\hat{\Z}^\sigma) \cap \Z^\sigma \subset \{(0,0)\}$,
it follows that
\[
(M_w C_T e_n, e_n)_{\hat{H}^\sigma_\delta} = \delta_{n,0} \cdot \langle p_0 \circ \tilde{T}, p_0 \rangle_{\T^\sigma_\delta} = \delta_{n,0},
\]
which yields $\operatorname{Tr}(C_T) = 1$.

For $w = w_T = \det DT$, it is easy to see that $w = \sum_{m \in \Z^2_{<0}} w_m p_m$ for suitable $w_m \in \C$.
We consider first the case $n \in \hat{\Z}^{--}$. In this case we have $p_n \circ T = \sum_{m \in \Z^2_{\leq 0}} c_m p_m$,
and hence $w \cdot p_n \circ T = \sum_{m \in \Z^2_{<0}} d_m p_m$, for suitable coefficients $c_m, d_m \in \C$. It follows
that $\langle w_T \cdot p_n \circ T, p_{-n} \rangle_{\T^\sigma_\delta} = 0$.
For $n \in \hat{\Z}^{++}$, let $\tilde{T} = I_{11} \circ T \circ I_{11}$. Direct calculation using a change of
variables $y = I_{11}(z)$  yields
\[
\langle w_T \cdot p_n \circ T, p_{-n} \rangle_{\T^\sigma_\delta}
= \langle w_{\tilde T} \cdot p_{-(n+2)} \circ \tilde T, p_{n+2} \rangle_{\T^\sigma_\delta}.
\]
Noting that $-(n+2) \in \hat{Z}^{--}$ and that $\tilde T$ and $w_{\tilde T}$ also satisfy the assumptions of the lemma,
we can apply the previous case and obtain that again $\langle w_T \cdot p_n \circ T, p_{-n} \rangle_{\T^\sigma_\delta} = 0$.
The conclusion $\operatorname{Tr}(M_w C_T) = 0$ follows.
\end{proof}

\begin{lem}\label{lem:detTA}
Let $\ell \in \{\pm 1\}$, $\delta \in \R^2$, and let $T$ be an analytic diffeomorphism of $\T^2$, holomorphic on a neighbourhood
of $\cl{\mathcal{A}_\delta}$. Assume additionally that $T$ extends holomorphically to
 $D^\sigma_\delta$ with $T(D^\sigma_\delta) \cc D^{-\sigma}_\delta$ for every $\sigma \in \Sigma^\ell$.
Then for any $\sigma \in \Sigma^\ell$, $T \circ T$ has a unique fixed point $z^\sigma \in D^\sigma_{\delta}$, and
\[D(T \circ T) (z^\sigma) = \cj{D \hat{T}(z^\sigma)} D \hat{T} (z^\sigma),\]
where $ \hat{T} = I_{11} \circ T$.
\end{lem}

\begin{proof}
We first note that $(T \circ T)(\Dsd) \cc \Dsd$ for $\sigma \in \Sigma^\ell$, implying the existence
of a unique fixed point $z^\sigma \in D^\sigma_\delta$ for $\sigma \in \Sigma^\ell$.
The map $\hat{T} = I_{11} \circ T$ also
satisfies $\hat{T}(\Dsd) \cc \Dsd$ for $\sigma\in\Sigma^\ell$, so
by Corollary \ref{cor:TA} each component of $\hat{T}$ is a rational function,
and we write $\hat{T}=\hat{T}_A$ with $A = (a_1, \ldots, a_n)$ the collection
of all coefficients occurring in $T$ (in any order).
Using Corollary \ref{cor:TA}$(i)$ we obtain
\[T \circ T = I_{11} \circ \hat{T}_A \circ I_{11} \circ \hat{T}_A = \hat{T}_{\cj{A}} \circ \hat{T}_A.\]
Next, observe that on the one hand, we have
\[(\hat{T}_A \circ \hat{T}_{\cj{A}})(\hat{T}_A(z^\sigma)) = \hat{T}_A ( \hat{T}_{\cj{A}} \circ \hat{T}_A (z^\sigma) ) = \hat{T}_A (z^\sigma),\]
and on the other hand, using Corollary \ref{cor:TA}$(ii)$, we have
\[(\hat{T}_A \circ \hat{T}_{\cj{A}}) \left(\cj{z^\sigma}\right) = \hat{T}_A\left( \cj{\hat{T}_A(z^\sigma)} \right)
= \cj{(\hat{T}_{\cj{A}} \circ \hat{T}_A) (z^\sigma)} = \cj{z^\sigma},\]
so that $\hat{T}_A (z^\sigma) = \cj{z^\sigma}$.
It follows that
\[
D(T \circ T) (z^\sigma)
= D(\hat{T}_{\cj{A}} \circ \hat{T}_A)(z^\sigma)
= D \hat{T}_{\cj{A}} \left(\cj{z^\sigma}\right) D \hat{T}_A (z^\sigma)
= \cj{D \hat{T}_A(z^\sigma)} D \hat{T}_A (z^\sigma).\qedhere.
\]
\end{proof}

\begin{rem}\label{rem:antiblaschke}
The above lemma implies that the two multipliers of a fixed point $z^\sigma$ of $T \circ T$ are either
real or complex conjugates of each other, as
\[\det(D (T \circ T) (z^\sigma)) = \cj{\det (D \hat{T}_A(z^\sigma))} \cdot \det (D \hat{T}_A(z^\sigma))
= | \det (D \hat{T}_A(z^\sigma)) |^2\] and
$\operatorname{Tr}(D (T \circ T) (z^\sigma)) = \operatorname{Tr}\left(\cj{D \hat{T}_A(z^\sigma)} D \hat{T}_A (z^\sigma)\right) \in \mathbb{R}$.
In contrast to the one-dimensional setting of anti-Blaschke products \cite{BN}, examples of orientation-reversing
circle maps allowing for explicit determination of resonances, the multipliers are no longer necessarily real.
Note also that under the assumptions of the lemma for $\ell = 1$ or $\ell = -1$
the two attracting fixed points of $T \circ T$ in $\mathcal{D}^\ell_\delta$
($z^\sigma$, $\sigma \in \Sigma^\ell$) have identical sets of multipliers.
\end{rem}

We are now ready to prove our second main theorem.

\begin{proof}[Proof of Theorem \ref{thm:thmB}]
By Theorem \ref{thm:thmA}, $C_T$ is trace-class on $H_{\alpha, -\gamma}$ for suitable $\alpha, \gamma \in \R_{>0}^2$,
and its trace is given by $\operatorname{Tr} C_T = \sum_{n\in \Ztwo} \langle C_T e_n, e_n \rangle_{\nu_{\alpha, -\gamma}}$.
Using the isometric isomorphism between $H_{\alpha,-\gamma}$ and $\mathcal{H}^1_{\alpha} \oplus \mathcal{H}^{-1}_{-\gamma}$
and the fact that $C^k_T = C_{T^k}$, for every $k \in \N$ we have
\begin{equation}\label{eq:trace_ag}
\operatorname{Tr}(C_T^k) =
\sum_{\sigma \in \Sigma^1} \sum_{n\in \hat \Z^\sigma} ( C_T^k e_n, e_n )_{\hat H^\sigma_\alpha} +
\sum_{\sigma \in \Sigma^{-1}} \sum_{n\in \hat \Z^{-\sigma}} ( C_T^k e_n, e_n )_{\hat H^{-\sigma}_{-\gamma}} =: S_1(k) + S_{-1}(k),
\end{equation}
as well as
\begin{equation}\label{eq:logdet}
\log \det (\operatorname{Id} - z C_T) = - \sum_{k=1}^\infty \frac{z^k}{k} S_1(k) - \sum_{k=1}^\infty \frac{z^k}{k} S_{-1}(k).
\end{equation}
We note that the assumptions and Lemma \ref{lem:max_mod} imply that for every $\ell \in \{\pm 1 \}$
and $\sigma \in \Sigma^\ell$, $T^\ell$ is holomorphic in a neighbourhood of $\cl{\Ds{\delta}}$
and $T^\ell(D^\sigma_\delta) \cc D^{\pm \sigma}_\delta$,
where $\delta=\alpha$ for $\ell=1$ and $\delta=-\gamma$ for $\ell=-1$.
We will calculate \eqref{eq:logdet} by handling the two sums $S_\ell, \ell \in \{ \pm 1 \}$, separately,
considering for each the two possible cases $T^\ell(D^\sigma_\delta) \subseteq D^\sigma_\delta$ and
$T^\ell(D^\sigma_\delta) \subseteq D^{-\sigma}_\delta$ for all $\sigma \in \Sigma^\ell$.
The claim of the theorem will follow with
$(1-z) \chi_T^1(z) = \exp (- \sum_{k=1}^\infty \frac{z^k}{k} S_1(k))$
and $\chi_T^{-1}(z) = \exp (- \sum_{k=1}^\infty \frac{z^k}{k} S_{-1}(k))$.
We first calculate $S_1(k)$:

\begin{enumerate}[(1)]
\item Consider first the case $T(D^\sigma_\alpha) \cc D^\sigma_\alpha$ for $\sigma \in \Sigma^1$.
For $\sigma=(-1, -1) \in \Sigma^1$, the composition operator $\tilde C_T$ associated to
$T$ on $H^\sigma_\alpha = \hat H^\sigma_\alpha$ is trace-class,
and its trace, computed using \eqref{eq:tr_def}, coincides with the term in \eqref{eq:trace_ag} corresponding to $\sigma$
(note that $\tilde{C}^k_T = \tilde{C}_{T^k}$).
By Lemma \ref{lem:trace_fp}, we obtain the value $((1-\lambda_1^k) (1-\lambda_2^k))^{-1}$,
where $\lambda_1, \lambda_2$ are the multipliers of the unique fixed point $z^\sigma \in \Ds{\alpha}$.
Similarly, for $\sigma=(+1, +1) \in \Sigma^1$, the associated composition operator $\tilde{C}_T$ is trace-class on $H^\sigma_\alpha$.
Writing $T=T_{A}$ for some
$A\in \C^m$, $m\in \N_0$, by Corollary \ref{cor:TA}$(i)$ we have
$T_A  = I_{11} \circ T_{\cj{A}}  \circ I_{11}$,
where $T_{\cj{A}}$ is a holomorphic map on $D^{-\sigma}_\alpha$ with unique attracting
fixed point $\cj{z^\sigma}$ (see Corollary \ref{cor:TA}$(ii)$). Moreover, it follows that
$DT_{\cj{A}}(\cj{z^\sigma}) = \cj{D T_{A}(z^\sigma)}$, thus by Lemma \ref{lem:trace_fp} we have
$\operatorname{Tr} \tilde{C}^k_T = ((1-\cj{\lambda_1}^k)(1-\cj{\lambda_2}^k))^{-1}$.
Since
\[\operatorname{Tr} \tilde{C}^k_T = {( \tilde{C}^k_T e_0, e_0 )_{H^\sigma_\alpha} +
\sum_{n\in \hat{\mathbb{Z}}^{++}}} ( \tilde{C}^k_T e_n, e_n )_{H^\sigma_\alpha} =
1 +  \sum_{n\in \hat{\mathbb{Z}}^{++}} ( C^k_T e_n, e_n )_{\hat H^\sigma_\alpha},\]
we obtain
\[
S_1(k) = 1 + D(\lambda_1^k, \lambda_2^k) + D(\cj{\lambda_1}^k, \cj{\lambda_2}^k),
\]
where $D(a, b) := \frac{1}{(1-a)(1-b)} - 1 = \sum_{n \in \mathcal{N}^1} a^{n_1} b^{n_2}$ for $a, b \in \D$,
$\mathcal{N}^{1} = \mathbb{N}^2_0\setminus\{(0,0)\}$. Calculating
\[
- \sum_{k=1}^\infty \frac{z^k}{k} D(a^k, b^k)
= - \sum_{n \in \mathcal{N}^1} \sum_{k=1}^\infty \frac{z^k a^{kn_1} b^{kn_2}}{k}
= \sum_{n \in \mathcal{N}^1} \log ( 1 - z a^{n_1} b^{n_2} ),
\]
we finally obtain
\[
-\sum_{k=1}^\infty \frac{z^k}{k} S_1(k) = \log(1-z) + \sum_{\sigma \in \Sigma^1} \sum_{n \in \mathcal{N}^1} \log (1 - z \lambda^n_\sigma).
\]

\item Next we consider the case $T(D^\sigma_\alpha) \cc D^{-\sigma}_\alpha$ for $\sigma \in \Sigma^1$.
For $k \in \N$ odd, $T^k$ satisfies the assumptions of Lemma \ref{lem:tr_1}, and the trace of the
composition operator associated to $T^k$ on $\mathcal{H}^1_\alpha$ exactly corresponds to the first sum in
\eqref{eq:trace_ag}, yielding $S_1(k) = 1$.
For $k$ even, $T^k$ satisfies $T^k(D^\sigma_\alpha) \cc D^\sigma_\alpha$ for $\sigma \in \Sigma^1$, and
so we can apply case $(1)$. Moreover, by Remark \ref{rem:antiblaschke} in this case the fixed point multipliers
$\lambda_1, \lambda_2$ of $T^2$ are either real or complex conjugates of each other, and hence
\[
S_1(k) = \begin{cases}
1 \qquad \qquad \qquad \qquad \;\; \text{for $k$ odd}, \\
1 + 2 D(\lambda_1^{k/2}, \lambda_2^{k/2}) \quad \text{for $k$ even}.
\end{cases}
\]
A straightforward calculation using the fact that $\lambda_\sigma = \lambda_{-\sigma}$ now yields
\[
-\sum_{k=1}^\infty \frac{z^k}{k} S_1(k) = \log(1-z) + \frac{1}{2} \sum_{\sigma \in \Sigma^1} \sum_{n \in \mathcal{N}^1} \log (1 - z^2 \lambda^n_\sigma).
\]
\end{enumerate}

Next, we proceed to calculate $S_{-1}(k)$. The approach to calculating $S_1(k)$ does not immediately translate to this case,
as the bidisks $D^\sigma_\delta, \sigma \in \Sigma^{-1},$ are not invariant under the map $T$, and so do not directly give
rise to trace-class composition operators on the respective spaces $H^\sigma_\delta$. Instead, we will show that the sums in
$S_{-1}(k)$ correspond to the traces of certain weighted composition operators $M_{\hat{w}} C_{\hat{T}}$ on $H^{--}_\delta$,
where $\hat{T}$ will be a map conjugated to $T^{-1}$ from Lemma \ref{lem:off_conjugation}.
Combining Remark \ref{rem:en_pn_equiv} with Lemma \ref{lem:off_conjugation}, we calculate
for $\sigma = (-1, +1)$ that
\[
\sum_{n \in \hat{\Z}^{-\sigma}} (C^k_T e_n, e_n)_{\hat{H}^{-\sigma}_{-\gamma}}
= \sum_{n \in \hat{\Z}^{-\sigma}} (C^k_T p_n, p_n)_{L^2(\T^2)}
= \sum_{n \in \hat{\Z}^{--}} ((M_{\hat{w}_\sigma} C_{\hat{T}_\sigma})^k p_n, p_n)_{L^2(\T^2)},
\]
with $\hat{w}_\sigma = \omega_T \det D\hat{T}_\sigma$ and $\hat{T}_\sigma = I_{01} \circ T^{-1} \circ I_{01}$.
For $\sigma = (+1, -1)$, Lemma \ref{lem:off_conjugation} yields the exact same equality
 with $\hat{T}_\sigma = I_{10} \circ T^{-1} \circ I_{10}$.
 We now consider two cases again.

\begin{enumerate}[(1')]
\item If $T^{-1}(D^\sigma_\gamma) \cc D^\sigma_\gamma$ for $\sigma \in \Sigma^{-1}$,
then $\hat{T}_\sigma(D^{--}_{\gamma}) \cc D^{--}_{\gamma}$.
We can then apply the same argument as in the above case (1), using that $\hat{w}_\sigma(\zeta^*) = \omega_T \mu_1 \mu_2$
for $\mu_1, \mu_2$ the multipliers of the unique fixed point $\zeta^* \in D^{--}_\gamma$ of $\hat{T}_\sigma$,
which corresponds to the unique fixed point $z^\sigma \in D^\sigma_\gamma$ of $T^{-1}$.
By the same argument as before, the multipliers of the respective fixed points in $D^\sigma_\gamma$ and $D^{-\sigma}_\gamma$
are complex conjugates of each other, and we obtain
\[
S_{-1}(k) = \frac{(\omega_T \mu_1 \mu_2)^k}{(1-\mu_1^k)(1-\mu_2^k)} +
\frac{(\omega_T \cj{\mu_1} \cj{\mu_2})^k}{(1-\cj{\mu_1}^k)(1-\cj{\mu_2}^k)}
= (\omega_T)^k \sum_{n \in \mathcal{N}^{-1}} ((\mu_1^{n_1}\mu_2^{n_1})^k + (\cj{\mu_1}^{n_1}\cj{\mu_2}^{n_2})^k),
\]
where $\mathcal{N}^{-1} = \mathbb{N}^2$. A similar calculation to above yields
\[
-\sum_{k=1}^\infty \frac{z^k}{k} S_{-1}(k) = \sum_{\sigma \in \Sigma^{-1}} \sum_{n \in \mathcal{N}^{-1}} \log (1-z \omega_T \lambda_\sigma^n).
\]

\item Finally, we consider the case $T^{-1}(D^\sigma_\gamma) \cc D^{-\sigma}_\gamma$ for $\sigma \in \Sigma^{-1}$, which
implies $\hat{T}_\sigma(D^\sigma_\gamma) \cc D^{-\sigma}_\gamma$ for $\sigma \in \Sigma^1$.
If $k \in \N$ is odd, we can apply Lemma \ref{lem:tr_1} to $\hat{T}_\sigma^{k}$ and the weight function $\hat{w}_\sigma$.
The trace of $(M_{\hat{w}_\sigma} C_{\hat{T}_\sigma})^k$ on $\mathcal{H}^1_\gamma$ in the lemma exactly coincides with $S_{-1}(k)$,
yielding $S_{-1}(k) = 0$. For $k$ even, we can apply case (1') to $T^{-2}$ instead of $T^{-1}$, again using the fact
that the fixed point multipliers $\mu_1, \mu_2$ are either both real or complex conjugates of each other, which yields
\[
S_{-1}(k) =
\begin{cases}
0 \quad \qquad \qquad \qquad \; \text{for $k$ odd}, \\
\frac{2 (\mu_1 \mu_2)^{k/2}}{(1 - \mu^{k/2}_1)(1 - \mu^{k/2}_2)}\quad \text{for $k$ even},
\end{cases}
\]
and again using $\lambda_\sigma = \lambda_{-\sigma}$ we obtain
\[
-\sum_{k=1}^\infty \frac{z^k}{k} S_{-1}(k) =  \frac{1}{2} \sum_{\sigma \in \Sigma^{-1}} \sum_{n \in \mathcal{N}^{-1}} \log (1-z^2 \lambda_\sigma^n).
\]
\end{enumerate}
Claims $(i)$ and $(ii)$ of the theorem now follow by combining the cases $(1)$-$(2)$ and $(1')$-$(2')$ for $\ell=1$ and $\ell=-1$, respectively.
\end{proof}

Using the explicit form of the zeros of $\det (\operatorname{Id} - z C_T)$ obtained in Theorem \ref{thm:thmB},
we can calculate the decay rate of their reciprocals, the eigenvalues of $C_T$.
For a map $T$ satisfying the assumptions of Theorem \ref{thm:thmB}, for any $\ell \in \{\pm 1\}$ and $\sigma \in \Sigma^\ell$,
we denote by $\lambda_\sigma = (\lambda_{\sigma,1}, \lambda_{\sigma,2})$ the multipliers of the unique attracting fixed point in $D^\sigma$
of $T^\ell$ if $T^\ell(D^\sigma) \subseteq D^\sigma$, and of $T^{2\ell}$ otherwise.

\begin{cor}
Let $T$ satisfy the assumptions of Theorem \ref{thm:thmB},
and let $\lambda = \lambda_{(-1,-1)}$
and $\mu = \lambda_{(-1,+1)}$. Let $\omega_T = 1$
if $T$ is orientation-preserving\footnote{$T$ is orientation-preserving exactly if either both,
or neither of $T$ and $T^{-1}$ satisfy the case $(i)$ in Theorem \ref{thm:thmB}.}, and $\omega_T = -1$
otherwise.
Then the nonzero eigenvalues of $C_T$ on $H_{\alpha,-\gamma}$
are $\{1\} \cup \mathcal{E}_1 \cup \mathcal{E}_{-1}$, where
\begin{align*}
\mathcal{E}_1 &= \begin{cases}
\{ \lambda^n, \cj{\lambda}^n : n \in \mathcal{N}^1\}, \qquad \qquad \; \; &\text{ if } T(D^\sigma) \subseteq D^\sigma, \sigma \in \Sigma^1,  \\
\{ \pm \lambda^{n/2} : n \in \mathcal{N}^1 \}, &\text{ if } T(D^\sigma) \subseteq D^{-\sigma}, \sigma \in \Sigma^1,
\end{cases}\\
\mathcal{E}_{-1} &= \begin{cases}
\{ \omega_T \cdot \mu^n, \omega_T \cdot \cj{\mu}^n : n \in \mathcal{N}^{-1} \}, &\text{ if } T^{-1}(D^\sigma) \subseteq D^\sigma, \sigma \in \Sigma^{-1},  \\
\{ \pm \mu^{n/2} : n \in \mathcal{N}^{-1} \}, &\text{ if } T^{-1}(D^\sigma) \subseteq D^{-\sigma}, \sigma \in \Sigma^{-1}.
\end{cases}
\end{align*}
Moreover, the algebraic multiplicity of each nonzero eigenvalue is exactly the number of its occurrences in the above sets.
\end{cor}

\begin{cor} \label{cor:limrate}
Let $T$ satisfy the assumptions of Theorem \ref{thm:thmB},
 $(\lambda_n)_{n\in\N}$ be the sequence of eigenvalues of $C_T$ sorted in order of decreasing modulus,
and $N_T(r) = \#\{n \in \N: |\lambda_n| \geq r \}$.
Then
\[
\lim_{r \to 0} \frac{\log N_T(r)}{\log |\log r|} = d,
\]
where
\begin{enumerate}[(i)]
 \item if $\lambda_{\sigma,1} \cdot \lambda_{\sigma,2} \neq 0$ for some $\sigma \in \Sigma$, then
 $d = 2$ (stretched-exponential decay), and
 \[
\lim_{n \to \infty} \frac{-\log|\lambda_n|}{n^{1/2}} = \eta_2
\]
with $\eta_2 = \left(1/2 \sum_{\sigma \in \Sigma: \lambda_{\sigma,1}\cdot \lambda_{\sigma,2} \neq 0}
(\log |\lambda_{\sigma,1}|\cdot \log |\lambda_{\sigma,2}|)^{-1}\right)^{-1/2}$.

\item if $\lambda_{\sigma,1} \cdot \lambda_{\sigma,2} = 0$ for all $\sigma \in \Sigma$, and $\lambda_{\sigma,k} \neq 0$
for some $\sigma \in \Sigma^1$ and $k \in \{1, 2\}$, then $d = 1$ (exponential decay), and
\[
\lim_{n \to \infty} \frac{-\log|\lambda_n|}{n} = \eta_1,
\]
with $\eta_1 = \left(\sum_{\sigma \in \Sigma^1} \sum_{k: \lambda_{\sigma,k} \neq 0} (\log |\lambda_{\sigma,k}|)^{-1}\right)^{-1}$.

\item if $\lambda_\sigma = 0$ for all $\sigma \in \Sigma^1$,
and $\lambda_{\sigma,1}\cdot \lambda_{\sigma,2} = 0$ for all $\sigma \in \Sigma^{-1}$,
then $d = 0$ (super-exponential decay). In this case the set of eigenvalues
 of $C_T$ is trivial, and $\operatorname{spec}(C_T) = \{0, 1\}$.
\end{enumerate}
\end{cor}

\begin{proof}
This follows directly from Lemma \ref{lem:cwexp_decay} applied to the eigenvalues of $C_T$
written as the values of a cone-wise exponential function $f \colon \Z^2 \to \C$.
In the case when both $T$ and $T^{-1}$ satisfy the case $(i)$ in Theorem \ref{thm:thmB},
this function is given by $f(n) = \lambda_{\sigma}^{|n|}$ with $\sigma = \sigma(n) \in \Sigma$
such that either $n \in \Z^2 \cap R^{\sigma}$, $\sigma \in \Sigma^1$,
or
$n \in (\Z \setminus \{0\})^2 \cap R^{\sigma}$, $\sigma \in \Sigma^{-1}$.
The other cases are similar.
\end{proof}

\section{Anosov maps with different decay rates for resonances}\label{sec:resonances_blaschke}

Based on the results of the previous section, in this section we shall prove our last main result, Theorem \ref{thm:thmC}.
The proof  will use the classical result that every toral Anosov diffeomorphism is homotopic to a toral automorphism,
and exploit the algebraic structure of $\GLtwo(\Z)$, which is isomorphic to the group of toral automorphisms $\Aut(\T^2)$.
To explicitly construct diffeomorphisms
whose corresponding composition operators have resonances exhibiting a desired decay rate,
we shall introduce a special group $\F$ of toral diffeomorphisms, containing the automorphisms as a subgroup.
The extension will consist of diffeomorphisms each of which is homotopic to an automorphism in an explicit way.

Beyond its immediate usefulness for our proof, the group $\F$ provides a rich source of explicit examples of toral diffeomorphisms
whose resonances can often be explicitly computed, and which includes both area-presering and non-area-preserving,
orientation-preserving and -reversing examples, as well as examples satisfying various symmetries.
As we believe this might be of broader interest, we include in Appendix \ref{sec:ap_class_f} a more comprehensive discussion and
illustrative set of examples, while constraining ourselves to a minimal introduction in this section.

\subsection{A special group of toral diffemorphisms}\label{sec:blaschke_examples}

The group $\Aut(\T^2)$ of linear diffeomorphisms of $\T^2$ is isomorphic to $\GLtwo(\Z)$,
with any $A = (a_{ij}) \in \GLtwo(\Z)$ giving rise to the toral automorphism
$\tau_A(z_1,z_2) = (z_1^{a_{11}}z_2^{a_{12}}, z_1^{a_{21}} z_2^{a_{22}})$.
For our purposes it will be convenient to view $\Aut(\T^2)$ as generated by the following automorphisms,
which can also be viewed as rational maps of $\Chat^2$:
\begin{enumerate}[(i)]
  \item a map $F$ given by $F(z_1,z_2) = (z_1 z_2, z_2)$, with $F^{-1}(z_1, z_2) = (z_1/z_2, z_2)$,
  \item an involution $R$ given by $R(z_1, z_2) = (z_2, z_1)$,
  \item involutions $I_{kl}$ for $k,l\in \{0,1\}$ given by $I_{kl}(z_1, z_2) = (z_1^{1-2k}, z_2^{1-2l})$.
\end{enumerate}
The set $\Gamma = \{F, R, I_{01}\}$ generates $\Aut(\T^2)$.
To create a richer group of toral diffeomorphisms, we extend the above by a continuous family of maps.
Utilising automorphisms of $\D$, the so-called Moebius maps
$b_a\colon \Chat \to \Chat, a \in \D$, given by
\[b_a(z) = \frac{z-a}{1-\cj{a}z}, \]
we define the additional set of generators as
\begin{enumerate}[(i)]
\setcounter{enumi}{3}
  \item a family $\mathcal{G} = \{G_{a,b}: a,b\in \D\}$ of maps given by
  \[G_{a,b}(z_1, z_2) = (b_a(z_1), b_b(z_2)),\]
  satisfying $G^{-1}_{a,b} = G_{-a,-b}$.
\end{enumerate}

\begin{defn}
Denote by $\F$ the group of diffeomorphisms of $\T^2$ generated by the set $\Gamma \cup \mathcal{G}$.
\end{defn}

\begin{rem}
The proof of Theorem \ref{thm:thmC} will be based on Theorem \ref{thm:thmB}, in particular we will require
all constructed maps $T$ to analytically extend to a neighbourhood of $\T^2$,
and to satisfy that
\begin{equation}\label{eq:all_invariance}
T^\ell \text{ extends holomorpically to } D^\sigma \text{ and }
T^\ell(D^\sigma) \subseteq D^{\pm \sigma} \quad (\sigma \in \Sigma^\ell, \ell \in \{\pm 1\}).
\end{equation}
A convenient class of maps satisfying \eqref{eq:all_invariance} is the semigroup of finite compositions of
$\{F, R, I_{11}\} \cup \mathcal{G}$.

We remark that the choice of generators $\mathcal{G}$ is not the only possible, though arguably the simplest choice
of non-linear maps satisfying \eqref{eq:all_invariance}.
More generally, this property is satisfied by a certain class of rational inner skew products, see \cite{ST}, taking the form
\[(z_1, z_2) \mapsto ( e^{i \theta} \frac{\tilde{p}(z_1, z_2)}{p(z_1, z_2)} , z_2),\]
with $\theta \in \mathbb{R}$ and $p$ a polynomial of
bidegree $(1, k)$ for $k\in \mathbb{N}_0$,
that is, of degree $1$ in $z_1$ and degree $k$ in $z_2$. Here, $\tilde{p}$ is the reflection of $p$ defined
as $\tilde{p}(z_1, z_2) = z_1 z_2^k \cj{p(1/\cj{z_1}, 1/\cj{z_2})}$.
The map $G_{a,0}(z) =(b_a(z_1), z_2)$ corresponds
to the polynomial $p(z) = 1-\cj{a}z_1$ of bidegree $(1, 0)$, and a general $G_{a,b}$
can be written as $G_{a,b} = G_{a,0} \circ R \circ G_{b,0}$.
\end{rem}

\subsection{Explicit homotopies of Anosov diffeomorphisms}\label{sec:red_maps}

We proceed by stating an algebraic fact about conjugacy classes of $\GLtwo(\Z)$,
which will allow us to establish an analytic conjugacy between an arbitary hyperbolic automorphism of $\T^2$
and an element of $\F$. While our proof is based on results from \cite{H}, for variants of this result
see, e.g., \cite{K,BR} and references therein. We defer the proof of the lemma to Appendix \ref{sec:ap_proofs}.

\begin{lem}\label{lem:conjstandard}
Every hyperbolic matrix $M \in \GLtwo(\Z)$ is similar (in $\GLtwo(\Z)$)
to a matrix of the form
\begin{equation}\label{eq:factoredform}
\pm
\begin{pmatrix} k_1 & 1 \\ 1 & 0 \end{pmatrix}
\begin{pmatrix} k_2 & 1 \\ 1 & 0 \end{pmatrix}
\cdots
\begin{pmatrix} k_n & 1 \\ 1 & 0 \end{pmatrix}, \qquad k_1,\ldots,k_n \in \N, n \geq 1.
\end{equation}
\end{lem}

\begin{cor} \label{cor:conjftilde}
Every hyperbolic automorphism of $\mathbb{T}^2$ is conjugated via an (analytic) toral automorphism
to an automorphism of the form
\begin{equation}\label{eq:reduced_toral_auto}
I_{11}^s \circ (F^{k_1} \circ R) \circ (F^{k_2} \circ R) \circ \cdots\circ (F^{k_n} \circ R), \qquad k_1, \ldots, k_n \in \N, n \geq 1, s \in \{0, 1\}.
\end{equation}
\end{cor}

\begin{proof}
Let $\tau_A$ be a hyperbolic automorphism of $\mathbb{T}^2$ associated to a hyperbolic matrix $A \in \GLtwo(\Z)$.
By the previous lemma, $A$ is similar to a matrix $B$ of the form \eqref{eq:factoredform},
that is, there exists $Q \in \GLtwo(\Z)$, such that
$A = Q^{-1} B Q$ with $B$ decomposing into a product of matrices
\begin{equation*}
- \mathbb{I} =  \begin{pmatrix} -1 & 0 \\ 0 & -1 \end{pmatrix} \quad \text{ and } \quad
 \begin{pmatrix} k & 1 \\ 1 & 0 \end{pmatrix}
 = \begin{pmatrix} 1 & 1 \\ 0 & 1 \end{pmatrix}^k \begin{pmatrix} 0 & 1 \\ 1 & 0 \end{pmatrix}, k \in \N.
\end{equation*}
We note that for every $k \in \N$, the latter is equal to $(M_F)^k M_R$,
where $M_F$, $M_R$ correspond to the toral automorphisms $F$ and $R$
(i.e., $\tau_{M_F} = F$ and $\tau_{M_R} = R$), and $\tau_{-\mathbb{I}} = I_{11}$.
It follows that $\tau_B$ has the desired form, and $\tau_A = \tau_Q^{-1} \circ \tau_B \circ \tau_Q$.
\end{proof}

Next, for any hyperbolic automorphism of the form \eqref{eq:reduced_toral_auto} we construct
non-linear area-preserving Anosov diffeomorphisms from $\F$ in the same homotopy class with easily computable
resonances. For this, we derive from the linear map $F^k \circ R, k \in \N,$ the one-parameter family of
non-linear toral diffeomorphisms
\begin{equation*}
U_{k,a}(z) =
(G_{0, a} \circ F^k \circ R \circ G_{-a, 0})(z) = (b_{-a}(z_1)^k z_2, z_1), \qquad a \in \D.
\end{equation*}
Applying this to \eqref{eq:reduced_toral_auto}, we define
\begin{equation}\label{eq:psika}
\Psi_{K,A} = I_{11}^s \circ U_{k_1, a_1} \circ \dots \circ U_{k_n, a_n},
\end{equation}
where $s \in \{0, 1\}, n \in \N, A = (a_1, \ldots, a_n)\in \D^n, K = (k_1, \ldots, k_n) \in \N^n$.
We write $A_o = (a_1, a_3, \ldots)$ and $A_e = (a_2, a_4, \ldots)$ for the respective tuples only involving
odd or even indices (analogously for $K_o$, $K_e$). For convenience,
we shall use the multiindex notation $A^K := a_1^{k_1}\cdots a_n^{k_n}$ for arbitrary $n$-tuples $A$ and $K$, $n \in \N_0$,
with the convention $A^K = 1$ when $A$ and $K$ are of length $0$.

\begin{prop}[Area-preserving maps homotopic to \eqref{eq:reduced_toral_auto}]\label{prop:uka_class}
For $s \in \{0, 1\}$, $A = (a_1, \ldots, a_n)\in \D^n$ and $K = (k_1, \ldots, k_n) \in \N^n, n \in \N$,
the map $\Psi_{K,A}$ is an area-preserving hyperbolic
toral diffeomorphism satisfying the conclusions of Theorem \ref{thm:thmB}.
In particular, $\Psi_{K,A}$ satisfies Theorem \ref{thm:thmB}$(i)$
if $s=0$, and (ii) if $s=1$, and $\Psi_{K,A}^{-1}$ satisfies (i) if $n+s$ is even, and (ii) if $n+s$ is odd.
Moreover, denoting $\lambda_\sigma$ the multipliers of the unique attracting fixed point of $\Psi_{K,A}^\ell$ in $D^\sigma$ for $\sigma \in \Sigma^\ell, \ell \in \{\pm1\}$ if $\Psi_{K,A}^\ell(D^\sigma) \subset D^\sigma$ and of $\Psi_{K,A}^{2\ell}$ otherwise, we have:
\begin{enumerate}[(i)]
  \item If $s=0$ and $n$ is odd, then
  \[
  \lambda_{--} = \cj{\lambda_{++}} = ((A^K)^{1/2}, -(A^K)^{1/2}) \quad \text{ and } \quad
  \lambda_{-+} = \lambda_{+-} = (A_o^{K_o} \cj{A_e^{K_e}}, \cj{A_o^{K_o}} A_e^{K_e}).
  \]

  \item If $s=0$ and $n$ is even, then
  \[
  \lambda_{--} = \cj{\lambda_{++}} = (A^{K_o}_o, A^{K_e}_{e}) \quad \text{ and } \quad
  \lambda_{-+} = \cj{\lambda_{+-}} = (\cj{A^{K_o}_o}, A^{K_e}_{e}).
  \]

  \item If $s=1$ and $n$ is odd, then
   \[
  \lambda_{--} = \lambda_{++} = (\cj{A^{K_o}_o} A^{K_e}_e, A^{K_o}_o \cj{A_e^{K_e}}) \quad \text{ and } \quad
  \lambda_{-+} = \cj{\lambda_{+-}} = ((A^{K_o}_o \cj{A_e^{K_e}})^{1/2}, -(A^{K_o}_o \cj{A_e^{K_e}})^{1/2}).
  \]

  \item If $s=1$ and $n$ is even, then
  \[\lambda_\sigma = (|A^{K_o}_o|^2, |A^{K_e}_{e}|^2) \text{ for all } \sigma \in \Sigma.\]

\end{enumerate}
\end{prop}

\begin{proof}
We begin by showing that $C_{\Psi_{K,A}}$ is trace-class on a suitable Hilbert space
$H_{\alpha,-\gamma}$. In the case $n > 1$ this will follow from
Theorem \ref{thm:thmB} by proving that
$\Psi_{K,A}$ satisfies the \textit{(p-sec)} condition, while the case $n=1$ will be handled separately.

For any map $T \colon \T^2 \to \T^2$ and $M = ([0, 2\pi]/\sim)^2$, we denote by $\tilde{T} \colon M \to M$ the map determined by
$\pi \circ \tilde{T} = T \circ \pi$ with $\pi(x) = e^{ix}$ for $x\in M$, and analogously for maps on $\T$ and $([0, 2\pi] / \sim)$.
By Lemma \ref{lem:ba_tilde} in the appendix, we have
$\tilde{b}_{-a}(e^{i \theta}) = \theta + g_{-a}(\theta)$
with $g_{-a}'(\theta) > -1$ for all $a \in \D$, and $\tilde{U}_{k,a}(x_1,x_2) = (k(x_1 + g_{-a}(x_1)) + x_2, x_1) $.
Thus we obtain
\begin{equation}\label{eq:duka}
D \tilde{U}_{k,a}(x) = \begin{pmatrix} s_{k,a}(x) & 1 \\ 1 & 0 \end{pmatrix},
\end{equation}
where $s_{k,a}(x) = k(1+g_{-a}'(x_1)) > 0$ for all $x = (x_1, x_2) \in M$.

Consider first the case $\Psi_{K,A} = I_{11}^s \circ U_{k_1,a_1} \circ \dots \circ U_{k_n,a_n}$ with $s=0$ and $n>1$.
Then $D \tilde{\Psi}_{K, A}$ is positive and hence
$(D\tilde{\Psi}_{K,A}(x))(\R^2_{\geq 0}) \subset \R^2_{>0}\cup\{0\}$ for every $x \in M$.
Using \eqref{eq:duka} we obtain
\begin{equation}\label{eq:dpsika}
D \tilde{\Psi}_{K,A}(x) = M_x + \begin{cases}
\begin{pmatrix} 1 & 0 \\ 0 & 1 \end{pmatrix}, \qquad \text{ if $n$ is even,}\\
\begin{pmatrix} 0 & 1 \\ 1 & 0 \end{pmatrix}, \qquad \text{ if $n$ is odd,}
\end{cases}
\end{equation}
with $M_x \geq 0$. By the criterion in Remark~\ref{rem:secc},
this implies that the first half of the \textit{(p-sec)} condition \eqref{eq:secc} is satisfied.
Since $\det D\tilde{\Psi}_{K,A}(x) = (-1)^n$, we also have that
$D \tilde{\Psi}^{-1}_{K,A}(x) = \begin{pmatrix} 1 + a_x & -b_x \\ -c_x & 1 + d_x \end{pmatrix}$ if $n$ is even,  and
$D \tilde{\Psi}^{-1}_{K,A}(x) = \begin{pmatrix} -a_x & 1 + b_x \\ 1+c_x & -d_x \end{pmatrix}$ for $n$ odd, with  $a_x, b_x, c_x, d_x > 0$.
It is easy to see that these are conjugated to matrices of the form \eqref{eq:dpsika}
via the matrix $\begin{pmatrix} 1 & 0 \\ 0 & -1 \end{pmatrix}$ or $\begin{pmatrix} -1 & 0 \\ 0 & 1 \end{pmatrix}$,
which via the criterion in Remark~\ref{rem:secc} implies the second half of the \textit{(p-sec)} condition \eqref{eq:secc}.
The case of $s = 1$ follows immediately, since
condition \eqref{eq:secc} holds for a map $\tilde{T}$ if and only if it holds for $-\tilde{T}$,
finishing the proof of the \textit{(p-sec)} condition for $\Psi_{K,A}$ with $n > 1$.

In the case $n=1$ the \textit{(p-sec)} condition does not hold, however one can verify that the assumptions of
Theorem \ref{thm:boundedness_large2small} still apply (similar to the case in Remark \ref{rem:no_sepcond}),
so that by Corollary \ref{cor:traceclass}, $C_{\Psi_{K,A}}$ is trace-class in this case also.

Next, we observe that for any $k \in \N$, $a \in \D$ and $\sigma \in \Sigma^1$, the map $U_{k,a}$ extends holomorphically to $D^\sigma$ with
$U_{k,a}(D^\sigma)\subseteq D^{\sigma}$, and fixes $z^*=(0,0)$ with
\[D U_{k,a}(z^*) = \begin{pmatrix} 0 & a^k \\ 1 & 0 \end{pmatrix}, \]
while for $\sigma \in \Sigma^{-1}$ the map $U_{k,a}^{-1}$ extends holomorphically to $D^\sigma$
with $U_{k,a}^{-1}(D^\sigma)\subseteq D^{-\sigma}$,
which implies the claim about how the cases Theorem \ref{thm:thmB}$(i)$-$(ii)$ apply
to $\Psi_{K,A}$ and $\Psi_{K,A}^{-1}$.

We now proceed to show the assertions $(i)$-$(iv)$.
Denoting $U_{K,A} = U_{k_1, a_1} \circ \cdots \circ U_{k_n, a_n}$
with $A = (a_1, \ldots, a_n)$ and $K = (k_1, \ldots, k_n)$,
we note that $DU_{K,A}(z^*) = DU_{k_1,a_1}(z^*) \cdots DU_{k_n,a_n}(z^*)$, allowing
us to compute the relevant multipliers for $\Psi_{K,A} = I_{11}^s \circ U_{K,A}$,
starting with the case of $s = 0$.
\begin{enumerate}[(i)]
  \item If $s=0$ and $n$ is odd, then
  \[D \Psi_{K,A}({z^*}) = D U_{K,A}({z^*}) = \begin{pmatrix}
  0 & A^{K_o}_o  \\
  A^{K_e}_e & 0 \end{pmatrix}.\]
  Thus, the multipliers of $\Psi_{K,A}$ at $z^*$ are
  $\lambda =  (v, -v)$ with $v=(A^K)^{1/2}$.
  \item If $s=0$ and $n$ is even, then
  \[ D \Psi_{K,A}({z^*}) = D U_{K,A}(z^*) = \begin{pmatrix}
  A^{K_o}_o  & 0 \\
  0 & A^{K_e}_e  \end{pmatrix},\]
  and, the multipliers of $\Psi_{K,A}$ at $z^*$ are $\lambda = (A_{o}^{K_{o}}, A_{e}^{K_{e}})$.
\end{enumerate}

For $s=1$ we shall use Lemma \ref{lem:detTA}
with $\hat{T} = U_{K,A}$ and $T = I_{11} \circ U_{K,A} = \Psi_{K,A}$, which yields
$D(\Psi_{K,A} \circ \Psi_{K,A})(z^*)
= \cj{D U_{K,A} (z^*)} D U_{K,A}(z^*)$.
\begin{enumerate}[(i)]
    \setcounter{enumi}{2}
  \item If $s=1$ and $n$ is odd, then
  the multipliers of $\Psi_{K,A}^2$ at $z^*$ are
  $\lambda =  (v, \cj{v})$ with $v=\cj{A_o^{K_{o}}} A^{K^e}_e$.
  \item If $s=1$ and $n$ is even, then
  the multipliers of $\Psi_{K,A}^2$ at $z^*$ are
  $\lambda = (|A_{o}^{K_{o}}|^2, |A_{e}^{K_{e}}|^2)$.
\end{enumerate}

Now, it remains to compute the multipliers of the inverse of $I_{11}^s \circ U_{K,A}$.
We set $S = I_{01} \circ R$ (so that $S^2 = S^{-2} = I_{11}$), and observe using Lemma \ref{lem:compose_prop}
in the appendix that the inverse $U^{-1}_{k,a} = G_{a, 0}\circ R \circ F^{-1} \circ G_{0, -a}$
obeys the relations
\begin{align*}
U^{-1}_{k,a} \circ S &= S^{-1} \circ U_{k,\cj{a}}, \\
U^{-1}_{k,b} \circ S^{-1} &= S \circ U_{k,b},
\end{align*}
and hence we have $S^{-1} \circ U_{k,a}^{-1} \circ S = I_{11} U_{k, \cj{a}}$ and
$S^{-1} \circ (U_{k,a} \circ U_{k,b})^{-1} \circ S = U_{k, b} \circ U_{k, \cj{a}}$.
Iterating, we obtain the conjugation
\begin{equation}\label{eq:conj_uka}
S^{-1} \circ U_{K,A}^{-1} \circ S  = I_{11}^n \circ U_{k_n, {\tilde{a}_n}} \circ \cdots \circ U_{k_4, {a_4}} \circ U_{k_3, \cj{a_3}} \circ U_{k_2, {a_2}} \circ U_{k_1, \cj{a_1}} =: \tilde{U}_{K,A},
\end{equation}
where $\tilde{a}_n$ is $a_n$ if $n$ is even or $\cj{a}_n$ if n is odd.
With the same conjugacy $S$, we obtain a conjugation
\begin{equation}\label{eq:conj_iuka}
S^{-1} \circ (I_{11} \circ U_{K,A})^{-1} \circ S = I_{11}^{n-1} \circ U_{k_n, {\check{a}_n}} \circ \cdots
\circ U_{k_4, \cj{a_4}} \circ U_{k_3, {a_3}} \circ U_{k_2, \cj{a_2}} \circ U_{k_1, {a_1}}  =: \check{U}_{K,A},
\end{equation}
where $\check{a}^n$ is $\cj{a_n}$ if $n$ is even or ${a}_n$ if n is odd.
We can now compute the relevant multipliers for $\Psi_{K,A}^{-1} = (I_{11}^s \circ U_{K,A})^{-1}$.
For $s = 0$, by \eqref{eq:conj_uka} these are given by the fixed-point multipliers of $\tilde{U}_{K,A}$.
\begin{enumerate}[(i)]
  \item If $s=0$ and $n$ is odd, by Lemma \ref{lem:detTA}
  the multipliers of $\tilde{U}_{K,A}^2$ at $z^*$
  are $\lambda = (v, \cj{v}), v = \cj{A_o^{K_{o}}} A_{e}^{K_{e}}$.

  \item If $s=0$ and $n$ is even, then
  the multipliers of  $\tilde{U}_{K,A}$ at $z^*$ are
  $\lambda = (\cj{A_{o}^{K_{o}}}, A_{e}^{K_{e}})$.
\end{enumerate}
For $s = 1$, the fixed-point multipliers of $\Psi_{K,A}^{-1}$ can be computed via those of $\check{U}_{K,A}$ by \eqref{eq:conj_iuka}.
\begin{enumerate}[(i)]
    \setcounter{enumi}{2}
  \item If $s=1$ and $n$ is odd, the multipliers of $\check{U}_{K,A}$
  at $z^*$ are $\lambda = (v, -v)$ with $v = (A_{o}^{K_{o}} \cj{A_{e}^{K_{e}}})^{1/2}$.
  \item If $s=1$ and $n$ is even, by Lemma \ref{lem:detTA}
  the multipliers of $\check{U}_{K,A}^2$ at $z^*$ are
  $\lambda = (|A_{o}^{K_{o}}|^2, |A_{e}^{K_{e}}|^2).$
\end{enumerate}
Claims $(i)$-$(iv)$ follow by combining the respective cases for the multipliers of
$\Psi_{K,A}$ and $\Psi_{K,A}^{-1}$.
\end{proof}

\begin{cor} \label{cor:existing_cases}
Let $\Psi_{K,A}$ be as in Proposition~\ref{prop:uka_class}. Then:
\begin{enumerate}[(i)]
  \item If $a_i \neq 0$ for all $i$, then the fixed point multipliers $\lambda_\sigma$ in Theorem \ref{thm:thmB}
  satisfy $\lambda_{\sigma,1} \cdot \lambda_{\sigma,2} \neq 0$ for all $\sigma \in \Sigma$.
  Moreover, $|\lambda_{\sigma,1}|$ and $|\lambda_{\sigma,2}|$ can (independently) be chosen to take any value in $(0, 1)$
  via suitable choice of $A \in \D^n$.

  \item If $n$ is even and exactly one of the $a_i$ is zero, then either case $(i)$ or $(ii)$ of
  Theorem \ref{thm:thmB} applies to both $\Psi_{K,A}$ and $\Psi^{-1}_{K,A}$,
  and all fixed point multipliers are of the form $(\lambda_1, 0)$ with $\lambda_1 \neq 0$.
  Moreover, $|\lambda_1|$ can be chosen to take any value in $(0, 1)$ via suitable choice of $A \in \D^n$,
  and in the case $(ii)$ of Theorem \ref{thm:thmB}, $\lambda_1 \in \mathbb{R}$.

  \item If $n>2$ and at most one of the $a_i$ is nonzero, then all multipliers are $(0, 0)$.
\end{enumerate}
\end{cor}

We shall next construct non-linear non-area-preserving maps in $\mathcal{F}$ homotopic to maps of the form \eqref{eq:reduced_toral_auto} with $n>1$,
yielding trivial resonances. For this we define the toral diffeomorphisms
\[W_{k,a}(z) = (F^k \circ R \circ G_{0,a})(z) = (z_1^k b_a(z_2), z_1) \quad (k \in \N, a\in \D, z\in \T)\]
and
\[\Xi_{K,a} = I_{11}^s \circ W_{k_1,0} \circ W_{k_2,0} \circ \cdots \circ W_{k_{n-1,0}}\circ W_{k_n,a},\]
where $s \in \{0,1\}$, $n > 1$, $a \in \D$ and $K = (k_1, \ldots, k_n) \in \N^n$.

\begin{lem}[Non-area-preserving maps homotopic to \eqref{eq:reduced_toral_auto}]\label{lem:nonareapres}
For any $s \in \{0,1\}$, $a \in \D$, $K \in \N^n$ with $n>1$, the map  $\Xi_{K,a}$ satisfies the assumptions of
Theorem \ref{thm:thmB}. Moreover $\chi_{\Xi_{K,a}} = 1.$
\end{lem}
\begin{proof}
Following similar calculations and notations as in
the proof of Proposition \ref{prop:uka_class} we first show that $\Xi_{K,a}$ satisfies the \textit{(p-sec)} condition.
We have that
$D\tilde{W}_{k,a}(x) = \begin{pmatrix} k &  s_{a}(x) \\ 1 & 0 \end{pmatrix}$
with $s_a(x) > 0$ for all $x\in M$. Thus, for $s=0$ and $n>1$ it follows that $D\tilde{\Xi}_{K,a}(x)$
can be written as $M_x+\begin{pmatrix} 1 &  0 \\ 0 & 0 \end{pmatrix}$ with $M_x \geq 0$ and thus
satisfies the first part of the \textit{(p-sec)} condition \eqref{eq:secc} by Remark \ref{rem:secc}. The case $s=1$ and
the second part of the \textit{(p-sec)} condition follow similarly.

To show that $\Xi_{K,a}$ does not yield any non-trivial resonances, first note that for $\ell \in \{1, -1\}$ the map $W_{k, a}^{\ell}$
extends holomorphically to $D^\sigma$ with $W_{k, a}^{\ell}(D^\sigma) \subseteq D^{\ell \sigma}$ for all $\sigma \in \Sigma^\ell$.
As the forward map $W_{k,a}$ fixes $z^* = (0, 0)$ and
\[D W_{k,a} (z^*) = \begin{pmatrix} -a \delta_{k, 1} & 0 \\ 1 & 0 \end{pmatrix},\]
we see that the multipliers of $\Xi_{K,a}$ at $z^*$ are trivial for any $n>1$
and $s\in\{0, 1\}$.
For the inverse map we use the conjugation $S=I_{01} \circ R$, obtaining the relations
\begin{align*}
W_{k,a}^{-1} \circ S &= S^{-1} \circ G_{-a, 0} \circ W_{k,0},\\
W_{k,a}^{-1} \circ S^{-1} &= S \circ G_{-\cj{a}, 0} \circ W_{k,0},
\end{align*}
yielding $S^{-1} \circ \Xi_{K,a}^{-1} \circ S =  G_{-\cj{a}, 0} \circ I_{11}^{n+s} \circ W_{k_n,0} \circ \cdots \circ W_{k_1, 0} := E_{K,a}$.
One can calculate that $(\cj{a},0)$ is a fixed point of $E_{K,a}$ for $n+s$ is even, and of $(E_{K,a})^2$ if $n+s$ is odd.
In both cases, its multipliers are $\lambda = (c, 0)$ for some $c \in \D$.
Thus, by Corollary \ref{cor:limrate}$(iii)$ all resonances are trivial.
\end{proof}

\begin{prop} \label{prop:existing_cases}
The homotopy class of every map of the form \eqref{eq:reduced_toral_auto}
contains non-linear Anosov diffeomorphisms $T\in \mathcal{F}$, such that
the corresponding operator $C_T$ is well defined and trace-class on
$H_{\alpha, -\gamma}$ for some $\alpha, \gamma \in \mathbb{R}^2_{>0}$. Moreover,
$T$ can be chosen such that the eigenvalue sequence of $C_T$ satisfies
either of the cases (i)-(iii) from the conclusions of Theorem \ref{thm:thmC}.
\end{prop}

\begin{proof}
To prove cases $(i)$ and $(ii)$ from Theorem \ref{thm:thmC},
we choose $T$ to be of the form $\Psi_{K,A}$ from \eqref{eq:psika},
noting that for any fixed $s \in \{0, 1\}$ and $K \in \N^n$, $n \in \N$, the map $\Psi_{K,0}$ is an area-preserving hyperbolic automorphism of the form \eqref{eq:reduced_toral_auto}, homotopic to any $\Psi_{K,A}$, $A \in \D^n$.
The claim follows directly from Proposition \ref{prop:uka_class} together
with Corollaries \ref{cor:existing_cases} and \ref{cor:limrate}.

The case $(iii)$ follows from Lemma \ref{lem:nonareapres},
noting that the maps $\Xi_{K,a}$ are not area-preserving and
hence not $C^1$-conjugated to toral automorphisms for $a \neq 0$, but have
trivial resonances.
\end{proof}

We will also need the following well-known result (see, e.g., \cite[Lemma 1.1]{F} and \cite[Theorem A]{M}):

\begin{lem}\label{lem:homaut}
For any Anosov diffeomorphism $f\colon\mathbb{T}^n \to \mathbb{T}^n$ there exists a
hyperbolic automorphism $g \colon \mathbb{T}^n \to \mathbb{T}^n$ which is homotopic to $f$.
That is, there exists a continuous one-parameter family of maps $h\colon [0, 1] \times \mathbb{T}^n \to \mathbb{T}^n$,
such that $h(0, \cdot) = f$ and $h(1, \cdot) = g$.
\end{lem}

We are now ready to prove our last main theorem, concluding that every homotopy class of toral Anosov diffeomorphisms contains
elements with resonances exhibiting any of stretched-exponential, exponential, or  trivially super-exponential decay rate.

\begin{proof}[Proof of Theorem \ref{thm:thmC}]
Let $\mathcal{H}$ be any homotopy class of toral Anosov diffeomorphisms.
By Lemma \ref{lem:homaut}, there exists a hyperbolic matrix $B \in \GLtwo(\Z)$ with $\tau_B \in \mathcal{H}$.
Corollary \ref{cor:conjftilde} yields that there are $A, Q \in \GLtwo(\Z)$ such that $\tau_A$ is of the form \eqref{eq:reduced_toral_auto}, and
$\tau_B$ is analytically conjugated to $\tau_A$, via $\tau_A = \tau_Q \circ \tau_B \circ \tau_Q^{-1}$.
We call $\mathcal{H}'$ the homotopy class containing $\tau_A$.
We note that because of its special form \eqref{eq:reduced_toral_auto} and using Remark \ref{rem:tuple2func},
the operator $C_{\tau_A}$ given by $f \mapsto f \circ \tau_A$
yields an isomorphism from $H_{\alpha,-\gamma}$ to itself,
while $\tau_Q$ gives rise to the isometric isomorphism
$C_{\tau_Q} \colon H_{Q,\alpha,-\gamma} \to H_{\alpha,-\gamma}$,
conjugating $C_{\tau_A}\colon H_{\alpha,-\gamma} \to H_{\alpha,-\gamma}$
and $C_{\tau_B}\colon H_{Q,\alpha,-\gamma} \to H_{Q,\alpha,-\gamma}$, for any
$\alpha, \gamma \in \R^2$.

Finally, by Proposition \ref{prop:existing_cases}, there exists an Anosov diffeomorphism
$\tau' \in \mathcal{H}'$ whose corresponding composition operator on $H_{\alpha',-\gamma'}$ for suitable $\alpha', \gamma' \in \R^2_{>0}$ has an eigenvalue sequence with any
one of the desired decay rates.
Writing out the homotopy explicitly, we have a one-parameter family of
 maps $h' \colon [0, 1] \times \mathbb{T}^2 \to \mathbb{T}^2$,
$h'_t = h'(t, \cdot) \in \mathcal{H}'$, such that $h'_0 = \tau_A$ and $h'_1 = \tau'$.
Conjugating with $\tau_Q$ we obtain a homotopy $h_t = \tau_Q^{-1} \circ h'_t \circ \tau_Q \in \mathcal{H}$
with $h_0 = \tau_B$ and $h_1 = \tau_Q^{-1} \circ \tau' \circ \tau_Q$.
Since $C_{\tau_Q} \colon H_{Q,\alpha',-\gamma'} \to H_{\alpha',-\gamma'}$ is an isometric isomorphism, the spectra of $C_{\tau'}$ and $C_{h_1} = C_{\tau^{-1}_Q} \circ C_{\tau'} \circ C_{\tau_Q}$ coincide,
and so the composition operator associated to $T = h_1$ satisfies the assertion of the proposition for $\nu = \nu_{P,\alpha,-\gamma}$ with $P = Q$, $\alpha=\alpha'$ and $\gamma=\gamma'$.
\end{proof}

\appendix

\section{Auxiliary results} \label{sec:ap_proofs}

Here we list a number of auxiliary results and proofs omitted but used in the main text.

\begin{lem} \label{lem:paq}
For every $P \in \GLtwo(\R)$, there exist $A \in \GLtwo(\Z)$ and $\tilde P \in \GLtwo(\R)$ with $\tilde P$ having only non-negative entries,
such that $P = A \tilde P$.
\end{lem}

\begin{proof}
It suffices to show that for $P \in \GLtwo(\R)$, there exists $B \in \GLtwo(\Z)$ such that $(B P)_{ij} > 0$ for $i,j=1,2$.
Writing the rows of $B$ as $b^u, b^s$, and the columns of $P$ as $p_u, p_s$ and denoting the cone
$\mathcal{C} = \{ v \in \R^2: \langle v, p_u \rangle > 0, \langle v, p_s \rangle > 0 \}$,
this is equivalent to there existing $b^u, b^s \in \Z^2 \cap \mathcal{C}$, such that
\begin{equation} \label{eq:det1}
b^u_1 b^s_2 - b^u_2 b^s_1 = 1.
\end{equation}
We fix any $b^u \in \Z^2 \cap \mathcal{C}$ (non-empty since $\mathcal{C}$ is an open convex cone),
without loss of generality satisfying $\operatorname{gcd}(b^u_1, b^u_2) = 1$. By Bezout's identity,
there exist $\hat{b}^s = (\hat{b}^s_1, \hat{b}^s_2) \in \Z^2$ such that $b^s_1(k) = \hat{b}^s_1 + k b^u_1, b^s_2(k) = \hat{b}^s_2 + k b^u_2$
are solutions to \eqref{eq:det1} for every $k \in \Z$. Moreover, it is easy to see that
for sufficiently large $k \in \Z$, $b^s = (b^s_1(k), b^s_2(k)) = \hat{b}^s + k \cdot b^u$ lies in the cone $\mathcal{C}$, finishing the proof.
\end{proof}

\begin{lem}\label{lem:matrix}
Let $\hat \sigma \in \Sigma^1$, and let $\{D_c : c \in C\} \subset \GLtwo(\R)$ be a
continuous family of matrices indexed by some compact set $C$,
satisfying $D_c(\mathbb{R}^2_{\geq 0}) \subset R^{\hat \sigma} \cup \{0\}$ for all $c \in C$.
Then, for any $\sigma \in \Sigma^1$, $\tilde \sigma \in \Sigma^{-1}$, $\delta, \tilde \delta \in \mathbb{R}^2_{>0}$,
there exists $q \in R^{\tilde \sigma}_{\tilde \delta}$
such that $D_c(q) \in R^\sigma_\delta$ for all $c \in C$.
\end{lem}

\begin{proof}
Let us first assume that $\hat\sigma = (1, 1)$, that is, $D_c(\mathbb{R}^2_{\geq 0}) \subset \mathbb{R}^2_{>0} \cup \{0\}$.
This implies $(D_c)_{kl} > 0$ for all $k,l \in \{1, 2\}$ and all $c \in C$.
By compactness of $C$ it follows that there are $\underline D, \overline D > 0$
such that $\underline D < (D_c)_{kl} < \overline D$ for all $k,l \in \{1, 2\}$ and all $c \in C$.
We fix $\delta, \tilde \delta \in \mathbb{R}^2_{>0}$,
and set $\overline \delta = \max\{\delta_1, \delta_2, \tilde\delta_1, \tilde\delta_2\}$.

We first consider the case $\sigma = (1, 1)$, $\tilde \sigma = (-1, 1)$. Setting $q_1 = - \overline \delta < -\tilde \delta_1$
and $q_2 = \overline \delta \max\{1,  \frac{1 + \overline D}{\underline D} \} > \tilde \delta_2$, we have
$q = (q_1, q_2) \in R^{\tilde \sigma}_{\tilde \delta}$, and since
$(D_c)_{k,1} q_1 + (D_c)_{k,2} q_2 \geq -\overline D \overline \delta +
\underline D \frac{1 + \overline D}{\underline D} \overline \delta = \overline \delta > \delta_k$
for $k = 1,2$ and all $c \in C$, we obtain
$D_c(q) \in R^2_{>0} + \delta = R^\sigma_\delta$, for all $c \in C$, as required.

The case $\sigma = (-1, -1)$, $\tilde \sigma = (1, -1)$ is similar, with $q' = -q \in R^{\tilde \sigma}_{\tilde \delta}$
and $D_c(q') = - D_c(q) \in R^2_{<0} - \delta = R^\sigma_\delta$ for all $c \in C$.
The other two cases ($\sigma = (1, 1)$, $\tilde \sigma = (1, -1)$ and $\sigma = (-1, -1)$, $\tilde \sigma = (-1, 1)$)
can be shown analogously by swapping the roles of $q_1$ and $q_2$ in the above construction.

Finally, for $\hat \sigma = (-1, -1)$, we note that the claim holds for $-D_c$ by the above,
and hence it follows for $D_c$ by replacing $q$ by $-q$.
\end{proof}

\begin{lem}\label{lem:cwexp_decay}
Let $f\colon \Z^2 \to \C$ be a cone-wise exponential function with cones being the quadrants $\hat{R}^{\sigma,o}$, $\sigma \in \Sigma$;
that is, for all $\sigma \in \Sigma$ there exist $\lambda_\sigma \in \D^2$, such that
$f(n) = \lambda_\sigma^n$ whenever $n \in \Z^2 \cap \hat{R}^{\sigma,o}$.
Denote by $(\lambda_n)_{n\in \N}$
be an enumeration of $\{f(n): n \in \Z^2\}$ sorted by decreasing modulus,
and $N(r) = \#\{n \in \N: |\lambda_n| \geq r \}$ for $r \in (0, 1)$.
Then $(\lambda_n)_{n\in \N}$ satisfies
\[
\lim_{r \to 0} \frac{\log N(r)}{\log |\log r|} = d,
\]
where:
\begin{enumerate}[(i)]
\item if $\lambda_{\sigma,1} \cdot \lambda_{\sigma,2} \neq 0$ for some $\sigma \in \Sigma$, then
$d = 2$ (stretched-exponential decay) and
\[
\lim_{n \to \infty} \frac{-\log|\lambda_n|}{n^{1/2}} = \eta_2
\]
with $\eta_2 = \left(1/2 \sum_{\sigma \in \Sigma: \lambda_{\sigma,1}\cdot \lambda_{\sigma,2} \neq 0}
(\log |\lambda_{\sigma,1}|\cdot \log |\lambda_{\sigma,2}|)^{-1}\right)^{-1/2}$.

\item if $\lambda_{\sigma,1} \cdot \lambda_{\sigma,2} = 0$ for all $\sigma \in \Sigma$, and $\lambda_{\sigma,k} \neq 0$
for some $\sigma \in \Sigma^1$ and $k \in \{1, 2\}$, then $d = 1$ (exponential decay), and
\[
\lim_{n \to \infty} \frac{-\log|\lambda_n|}{n} = \eta_1,
\]
with $\eta_1 = \left(\sum_{\sigma \in \Sigma^1} \sum_{k: \lambda_{\sigma,k} \neq 0} (\log |\lambda_{\sigma,k}|)^{-1}\right)^{-1}$.

\item if $\lambda_\sigma = 0$ for all $\sigma \in \Sigma^1$,
and $\lambda_{\sigma,1}\cdot \lambda_{\sigma,2} = 0$ for all $\sigma \in \Sigma^{-1}$,
then $d = 0$ (super-exponential decay). In this case $(\lambda_n)_{n \in \N}$ is the trivial sequence with $\lambda_n = \delta_{n, 1}$.
\end{enumerate}
\end{lem}

\begin{proof}
We write $\ell_s(r) = \log r / \log |s| $ for $s \in \D$ and $r \in (0, 1)$.
For $p, q \in \D, r \in (0, 1)$ we have
\begin{align*}
N^r_{p} :=& \#\{n \geq 1 : |p|^n \geq r \} = \lfloor \ell_p(r) \rfloor,  \\
N^r_{p,q} :=& \#\{n, m \geq 1 : |p|^n |q|^m \geq r \} \\
=& \#\{n, m \geq 1 : n (-\log |p|) + m (-\log |q|) \leq -\log r  \}\\
=& \frac{\lfloor\ell_p(r)\rfloor \lfloor\ell_q(r)\rfloor}{2} + \delta,
\qquad \text{ with } |\delta| \leq 1 + \lfloor\ell_p(r)\rfloor+ \lfloor\ell_q(r)\rfloor,
\end{align*}
where the last equality is obtained from counting the number of integer lattice points
in the triangle spanned by $(0, 0)$, $(\ell_p(r),0)$ and $(0, \ell_q(r))$, and
subtracting those lying on one of the axes.
From the definition of $f$ we obtain
\[
N(r) = 1 + \sum_{\sigma \in \Sigma}
\left( \frac{\lfloor\ell_{\lambda_{\sigma,1}}(r)\rfloor \lfloor\ell_{\lambda_{\sigma,2}}(r)\rfloor}{2} + \delta_\sigma\right)
+ \sum_{\sigma \in \Sigma^1} \left(\lfloor\ell_{\lambda_{\sigma,1}}(r)\rfloor + \lfloor\ell_{\lambda_{\sigma,2}}(r)\rfloor \right),
\]
with $|\delta_\sigma| \leq 1 + \lfloor\ell_{\lambda_{\sigma,1}}(r)\rfloor + \lfloor\ell_{\lambda_{\sigma,2}}(r)\rfloor$.

We now prove $(i)$.
Since $\ell_s(r) \to \infty$ as $r \to 0$ for any $s \in \D$ and
$\lambda_{\sigma,1} \cdot \lambda_{\sigma,2} \neq 0$ for some $\sigma \in \Sigma$,
we have that for every $\epsilon > 0$ there exists $r_\epsilon > 0$ such that for $r \in (0, r_\epsilon)$
it holds that
\begin{equation} \label{eq:ntr_estimate2}
 (1-\epsilon) \cdot \left(\frac{\log r}{\eta_2}\right)^2 \leq N(r) \leq   (1+\epsilon) \cdot \left(\frac{\log r}{\eta_2}\right)^2.
\end{equation}
The assertion $d = 2$ immediately follows. Moreover,
since $n > (1+\epsilon)\cdot (\log r / \eta_2)^2 \geq N(r)$ implies $|\lambda_n| < r$, a short calculation
yields that $\log |\lambda_n| < \log r$ holds for any sufficiently large $n$ and
$\log r \in \left(-(1+\epsilon)^{-1/2}\eta_2 \sqrt{n}, \log r_\epsilon\right)$, which implies
\[\frac{-\log |\lambda_n|}{\sqrt{n}} \geq (1+\epsilon)^{-1/2}\eta_2.\]
Conversely, $n < (1-\epsilon)\cdot (\log r / \eta_2)^2 \leq N(r)$ implies $|\lambda_n| \geq r$, and hence for large $n$ we obtain
\[\frac{-\log |\lambda_n|}{\sqrt{n}} \leq (1-\epsilon)^{-1/2}\eta_2.\]
Since the choice of $\epsilon > 0$ was arbitrary, assertion $(i)$ follows.

The proof of $(ii)$ is very similar; replacing \eqref{eq:ntr_estimate2} by
\begin{equation}
 (1-\epsilon) \cdot \frac{|\log r|}{\eta_1} \leq N(r) \leq   (1+\epsilon) \cdot \frac{|\log r|}{\eta_1}
\end{equation}
yields $d = 1$, as well as $\eta_1 / (1+\epsilon) < -\log |\lambda_n| / n < \eta_1 / (1-\epsilon)$ for sufficiently large $n$.

Finally, $(iii)$ follows directly by observing that in this case $f(n) = \delta_{n_1, 0} \delta_{n_2, 0}$.
\end{proof}

\begin{proof}[Proof of Lemma \ref{lem:J_stretched}]
The proof follows the same steps as \cite[Propositions 3.4 \& 3.5]{BJ2}. Let
$J^*\colon H_{P, A, -\Gamma} \to H_{P, \alpha, -\gamma}$ denote the Hilbert space adjoint of $J$,
then $J^*J$ is diagonal in the orthogonal basis of monomials,
as  $(J^* J p_n, p_m)_{H_{P, \alpha, -\gamma}} = (J p_n, J p_m)_{H_{P, A, -\Gamma}} =
(p_n, p_m)_{H_{P, A, -\Gamma}} =  \omega_n (p_n, p_m)_{H_{P, \alpha, -\gamma}} $ with
$\omega_n = \nu_{P, A, -\Gamma}(n)/\nu_{P, \alpha, -\gamma}(n)$.
Therefore the eigenvalues of $J^* J$, which are square roots of the singular values of $J$, are given by $\{\sqrt{\omega_n} : n \in \Z^2\}$,
which can be written as the set of all $\lambda^n, n \in \N_0^2,$ and $\mu^n, n \in \N^2$,
with $\lambda = e^{\frac{A-\alpha}{2}}$ and $\mu = e^{\frac{\gamma-\Gamma}{2}}$.
Since $A - \alpha,\gamma -\Gamma \in \R^2_{>0}$, Lemma \ref{lem:cwexp_decay}$(i)$ applies,
with \[
\eta = \eta_2 = \left( \frac{1}{\log (A_1 - \alpha_1) \log (A_2 - \alpha_2)} + \frac{1}{\log (\gamma_1- \Gamma_1) \log (\gamma_2-\Gamma_2)} \right)^{-1/2}.\qedhere
\]
\end{proof}

\begin{proof}[Proof of Lemma \ref{lem:conjstandard}]
We first note that it is sufficient to consider the case $\operatorname{Tr}M \geq 0$,
as the general case easily follows by considering $-M$ in the opposite case.
By \cite[Theorem 3]{H}, every $M \in \GLtwo(\Z)$ with $\operatorname{Tr}M \geq 0$
is similar to a so-called `standard matrix' (see \cite[Definition 1]{H}), which in the case of a hyperbolic matrix
(implying $\operatorname{Tr}M \neq 0$ and real eigenvalues $\neq -1, 1$)
reduces to two (alternative) cases:
\begin{enumerate}[(i)]
  \item $M$ is similar to $\begin{pmatrix} 1 & 1 \\ 1 & 0 \end{pmatrix}$ in the case $\operatorname{Tr}M = 1$,
  \item $M$ is similar to $\begin{pmatrix} a & b \\ c & d \end{pmatrix}, 0 \leq d \leq b, c < a$ in the case $\operatorname{Tr}M \geq 2$.
\end{enumerate}

We are left to show that every matrix in case (ii) is similar to one of \eqref{eq:factoredform}.
We note that every matrix of this form satisfies $a > 0$,
since $c < a \leq 0$ would imply $d = \operatorname{Tr}M - a \geq 2$,
and hence $1 \geq \det M = ad - bc \geq (a - c) d \geq 2$, a contradiction.
We also note that $c = 0$ yields a non-hyperbolic matrix, so we can assume $c \neq 0$.

For $c > 0$ the claim follows immediately from \cite[Theorem 4]{H}.
For $c < 0$, we note that $0 \leq d \leq b$ implies $b > 0$ (as otherwise $\det M = 0$),
and hence $ - b c \geq 1$. Since $a d - b c \leq 1$, it follows that $a d = 0$ and hence $d = 0$, as well as $b = 1$ and $c = -1$.
We obtain that $M$ is similar to a matrix of the form
$M_a = \begin{pmatrix}a & 1 \\ -1 & 0\end{pmatrix}$, which is only hyperbolic for $a \geq 3$. The claim follows by observing that $M_a, a \geq 3$, is similar to
\begin{equation*}
N_a = \begin{pmatrix}a-1 & 1 \\ a-2 & 1\end{pmatrix}
= \begin{pmatrix} 1 & 1 \\ 1 & 0\end{pmatrix} \begin{pmatrix} a-2 & 1 \\ 1 & 0\end{pmatrix} \in \GLtwo(\Z)
\end{equation*}
via $M_a = C^{-1} N_a C$ with $C = \begin{pmatrix}1 & 0 \\ 1 & 1\end{pmatrix}$.
\end{proof}

\section{Notes on the special group of toral diffeomorphisms}
\label{sec:ap_class_f}

Here we briefly restate the definition of the group of toral diffeomorphisms $\F$ from
Section \ref{sec:resonances_blaschke}, before expanding on the various types of maps that can
be constructed within this group via a set of examples.
For $a\in \D$, let $b_a\colon \Chat \to \Chat$ be an automorphism of $\D$ (or Moebius map) given by
\[b_a(z) = \frac{z-a}{1-\cj{a}z}.\]
Every such map satisfies $b_a(\T) = \T$, and a straightforward calculation
yields the following lemma.
\begin{lem}\label{lem:ba_tilde}
Fix $a=|a|e^{i\alpha} \in \D$ and let $b_a$ be as above.
For $M=([0, 2\pi]/\sim)$ let $\tilde{b}_a\colon M \to M$ be the map determined by $\pi \circ \tilde{b}_a = b_a \circ \pi$ with
$\pi(\theta) = e^{i \theta}$ for $\theta \in M$. Then
\[
\tilde{b}_{a}(\theta) = \theta + g_a(\theta),
\]
with $g_a(\theta) = 2 \arctan \left( \frac{|a|\sin(\theta - \alpha)}{1-|a|\cos(\theta-\alpha)}\right)$
and $g'_a(\theta) = 2  \left( \frac{|a|\cos(\theta - \alpha)-|a|^2}{1-2|a|\cos(\theta-\alpha)+|a|^2}\right) > -1$.
\end{lem}

We recall the definition of the maps $F \colon (z_1, z_2) \mapsto (z_1 z_2, z_2)$,
$R \colon (z_1, z_2) \mapsto (z_2, z_1)$, $I_{kl} \colon (z_1, z_2) \mapsto (z_1^{1-2k}, z_2^{1-2l})$ for $k, l \in \{0, 1\}$,
and $\mathcal{G} = \{G_{a,b}\colon (z_1, z_2) \mapsto (b_a(z_1), b_a(z_2)) : a, b \in \D\}$ from Section \ref{sec:resonances_blaschke},
as well as the definition of $\F$ as the group of toral diffeomorphisms generated by these maps.
The group of toral automorphisms $\Aut(\T^2)$ is generated by the set $\Gamma = \{F, R, I_{01}\}$.
Any $T \in \{F, I_{00}, I_{11}\}\cup \mathcal{G}$ yields an
orientation-preserving diffeomorphism of $\Ttwo$,
while all of $\{I_{01}, I_{10}, R\}$ are orientation-reversing.
The next lemma summarises some basic properties of these maps.
\begin{lem}\label{lem:compose_prop}
Let $a,b\in\D$ and $k,l,m,n\in\{0,1\}$. The following commutation relations hold.
\begin{enumerate}[(i)]
\item $F \circ I_{kl} = I_{kl} \circ F^{-1}$ and $F^{-1} \circ I_{kl} = I_{kl} \circ F$ for $k \neq l$,
\item $F \circ I_{11} = I_{11} \circ F$ and $F^{-1} \circ I_{11} = I_{11} \circ F^{-1}$,
\item $R \circ I_{kl} = I_{lk} \circ R$,
\item $I_{kl} \circ I_{mn} = I_{1-\delta_{km},\, 1-\delta_{ln}}$,
\item $G_{a,b} \circ I_{kl} = I_{kl} \circ G_{(1-k)a+k\cj{a},\; (1-l)b+l\cj{b}}$,
\item $G_{a,b} \circ R = R \circ G_{b,a}$.
\end{enumerate}
\end{lem}

\begin{proof}
All statements follow by direct computation, with part {\it (v)} using $b_a(z^{-1}) = b_{\cj{a}}(z)^{-1}$.
\end{proof}

We now provide examples from various interesting sub-classes of diffeomorphisms in $\F$.

\paragraph{Anosov diffeomorphisms in $\F$.}

For $a\in \mathbb{D}$ let $T_a\colon \Ttwo \to \Ttwo$ be given by
\[T_a = G_{0, -a} \circ F \circ R \circ G_{a,0}.\]
Using Lemma \ref{lem:ba_tilde}, we can see that the derivative of
the respective map $\tilde{T}_a\colon M \to M$ is given by
\[D \tilde{T}_a(x) =
\begin{pmatrix}
s_{a}(x) & 1 \\
1 & 0
\end{pmatrix},\]
with $s_a(x) = (1+ g_a'(x_1)) > 0$ for all $x=(x_1, x_2)\in M$.
\begin{example}\label{ex:hyp}~
\begin{enumerate}[(i)]
\item For $a=0$ the map $T_0(z) = (z_1z_2, z_1)$ is an Anosov automorphism induced by
$\begin{pmatrix}
1 & 1 \\
1 & 0
\end{pmatrix}$
with eigenvalues $\lambda_{u/s} = \varphi^{\pm 1}$ where
$\varphi = (1+\sqrt{5})/2$ is the golden mean.

\item One can check that $T_a(z) = (b_a(z_1) z_2, z_1)$ is Anosov for
all $a\in \mathbb{D}$ by finding suitable cone fields
(see Definition \ref{def:Anosov}).

\item The maps $T_b\circ T_a$ given by
\[(T_b \circ T_a)(z) = (b_b(b_a(z_1) z_2) z_1, b_a(z_1) z_2)\]
are Anosov for all $a,b\in \D$, and in fact $\R^2_{>0} \cup \R^2_{<0}$ can be chosen as the invariant expanding cone,
and its complementary cone as the invariant contracting one.
These are the maps considered in \cite{SlBaJu_NONL17} and \cite{PoS}.
\end{enumerate}
The maps in (iii) are orientation-preserving while the ones in (i) and (ii) are
orientation-reversing.
\end{example}

\paragraph{Area-preserving diffeomorphisms in $\F$.}~

\begin{lem}\label{lem:area}
Let $F_a = G_{0, -a} \circ F \circ G_{0,a}$ for $a\in\D$.
Then any finite composition of the elements of
$\Gamma_{ap} = \Gamma \cup \{F_a : a \in \D, k\in \N \}$
is area-preserving.
\end{lem}

\begin{proof}
For $z\in\Ttwo$ we have $F_a(z_1, z_2) = (z_1 b_a(z_2), z_2)$. As
$b_a$ preserves $\T$ we have $|\det{DF_a}(z)| = |b_a(z_2)| = 1$ for all $z\in\Ttwo$.
As all elements in $\Gamma_{ap}$ are area-preserving, so is their composition.
\end{proof}

\begin{example}\label{ex:ap}~
\begin{enumerate}[(i)]
\item The maps in Example \ref{ex:hyp} are area-preserving Anosov diffeomorphisms
as $T_a = G_{0,-a} \circ F \circ G_{0,a} \circ R = F_a \circ R$ with $F_a$ as in Lemma \ref{lem:area}.
\item The map $T = F \circ R \circ G_{0,a} \circ F \circ R$ given by
\[T(z_1, z_2) = (z_1 b_a(z_1) z_2, z_1 z_2)\] is not area-preserving
for $a\in\D\setminus\{0\}$ as
$|\det{DT}(z)| = |z^2_1 b_a'(z_1) z_2| = |b_a'(z_1)|$ for $z\in \Ttwo$.

\item The map $T = F \circ R \circ G_{0, -a} \circ F \circ R \circ G_{a,b}$ given by
\[T(z_1, z_2) = (z_1 b_a(z_1) b_b(z_2), b_a(z_1) b_b(z_2)), \quad (a\in \D),\]
is area-preserving for $b=0$ (see Example \ref{ex:hyp}(ii)), but is not area-preserving
for $b\neq 0$.
\end{enumerate}
\end{example}

\paragraph{Maps with symmetries.}

An automorphism $T$ of some topological space is said to have a \textit{symmetry}
if there exists an automorphism $H$ so that
\[ H^{-1} \circ T \circ H = T, \]
and to have a \textit{reversing symmetry} if there exists
an automorphism $H$ so that
\[H^{-1} \circ T \circ H = T^{-1}.\]
Clearly, for any rational map $T$ with only real coefficients, Corollary \ref{cor:TA}$(i)$
implies that $I_{11}$ is a symmetry of $T$, which by Theorem \ref{thm:thmB} and its proof induces
symmetry relations on the resonances.

On the other hand, presence of reversing symmetries is of considerable interest in classical and quantum mechanics.
For systems with {\it time-reversal symmetry} the reverse motion satisfies the same laws of
motion as the forward motion. Usually this time-reversal symmetry corresponds to a particular involution map $H$.
However, the notion of reversible systems was extended to include all involutions and
even non-involutary reversing symmetries, see for example the survey \cite{LR_98} adapted to dynamical systems
or \cite{BR} specifially for toral automorphisms. We will next present some Anosov diffeomorphisms
in $\mathcal{F}$ which have reversing symmetries.

\begin{lem}
Let $k\in \mathbb{N}$ and $a\in (0,1)$, and define the maps $T_{k,a} = F^k\circ R \circ G_{a, -a} \circ F^k \circ R$
and $U_{k,a} = G_{0, -a} \circ T_{k, a} \circ G_{a, 0}$.
\begin{enumerate}[(i)]
 \item The map $T_{k, a}$ is a non-area-preserving Anosov diffeomorphism with a reversing symmetry.
 \item The map $U_{k, a}$ is an area-preserving Anosov diffeomorphism with a reversing symmetry.
\end{enumerate}
\end{lem}

\begin{proof}
Using Lemma \ref{lem:ba_tilde} it is not difficult to see that both $T_{k,a}$ and $U_{k,a}$ are Anosov with
the first and third quadrant of $\mathbb{R}^2$ forming an unstable, and the second and forth quadrant forming
a stable invariant cone. Area preservation of $U_{k,a}$ and non-preservation for $T_{k,a}$ can be computed directly,
noting that $U_{k,a} = (G_{0,-a} \circ F^k \circ R \circ G_{a,0})^2$. For the symmetries,
a calculation with $H=I_{01} \circ R$ using Lemma \ref{lem:compose_prop} and the fact that $a \in \R$ reveals that
$H^{-1} \circ T_{k,a}^{-1} \circ H = T_{k,a}$.
Since $G_{0, a} \circ H = H \circ G_{a, 0}$ we also have $H^{-1} \circ U_{k,a}^{-1} \circ H = U_{k,a}$.
\end{proof}

\paragraph{Comparison to Blaschke product diffeomorphisms.}

In \cite{PuSh_ETDS08} the authors coined the notion of Blaschke product diffeomorphisms, which are maps of the form
\[T(z_1, z_2) = (A(z_1) B(z_2), C(z_1) D(z_2)),\]
where $A, B, C, D$ are Blaschke products in one variable.
They state that these are precisely the analytic maps on a neighbourhood of the open bidisk $\Dtwo$,
mapping $\Dtwo$ to itself and $\Ttwo$ diffeomorphically to itself,
and provide an explanation of this in \cite[Remark 5.2]{PuSh_ETDS08}.
Here we observe that this claim is inaccurate: while Examples \ref{ex:ap}(ii)-(iii)
are instances of Blaschke product diffeomorphisms, Example \ref{ex:hyp}(iii)
is a hyperbolic diffeomorphism in $\F$ containing Blaschke factors of a product of two variables,
and cannot be written as a Blaschke product diffeomorphism.

\paragraph{Acknowledgements}
All authors gratefully acknowledge the support for the research presented in this
article by the EPSRC grant EP/RO12008/1. JS gratefully acknowledges
the support by the ERC grant 833802-Resonances.

\end{document}